\documentclass[11pt,reqno]{amsart}
\usepackage[left=0.9in,top=0.9in,right=0.9in,bottom=0.9in]{geometry}
\usepackage{amsaddr}

\usepackage{amsfonts,amsmath,amsthm,amssymb,latexsym,mathrsfs,stmaryrd}
\usepackage{hyperref}
\usepackage{mathtools}
\usepackage{graphicx}
\usepackage{subcaption}
\usepackage{afterpage}

\usepackage{tikz}\usetikzlibrary{cd,fit,matrix,arrows,decorations.pathmorphing}
\tikzset{commutative diagrams/.cd}

\numberwithin{equation}{section}

\newtheorem{theorem}{Theorem}[section]
\newtheorem{corollary}[theorem]{Corollary}
\newtheorem{lemma}[theorem]{Lemma}
\newtheorem{proposition}[theorem]{Proposition}
\newtheorem{conjecture}[theorem]{Conjecture}

\theoremstyle{definition}
\newtheorem{definition}[theorem]{Definition}
\newtheorem{example}[theorem]{Example}

\newtheorem{assumption}[theorem]{Assumption}

\newtheorem{remark}[theorem]{Remark}

\theoremstyle{remark}
\newtheorem*{warning}{Warning}

\usepackage{enumitem}
\setenumerate[1]{label=\textup{(\alph*)}}
\setenumerate[2]{label=\textup{(\roman*)}}

\newcommand\Z{{\mathbb Z}}
\newcommand\N{{\mathbb N}}
\newcommand\Q{{\mathbb Q}}
\newcommand\C{{\mathbb C}}
\newcommand\A{{\mathbb A}}

\def\P{{\mathbb P}}

\newcommand\Fq{{{\mathbb F}_q}}

\newcommand\Zp{{{\mathbb Z}_p}}
\def\O{\mathcal O}

\DeclareMathOperator{\im}{im}
\DeclareMathOperator{\coker}{coker}

\DeclareMathOperator{\Hom}{Hom}
\DeclareMathOperator{\End}{End}
\DeclareMathOperator{\Aut}{Aut}

\DeclareMathOperator{\Spec}{Spec}

\DeclareMathOperator{\GL}{GL}

\DeclareMathOperator{\Mat}{Mat}
\DeclareMathOperator{\Sym}{Sym}

\newcommand\subeq{\subseteq}

\newcommand\bbar{\overline} 
\newcommand\ttilde{\widetilde} 
\newcommand\hhat{\widehat} 
\newcommand\onto{\twoheadrightarrow}
\newcommand\incl{\hookrightarrow}
\newcommand{\map}[1][]{{\xrightarrow{#1}}} 

\DeclarePairedDelimiter{\abs}{\lvert}{\rvert}

\DeclarePairedDelimiter{\set}{\{}{\}}

\DeclarePairedDelimiter{\parens}{\lparen}{\rparen}

\DeclarePairedDelimiter\floor{\lfloor}{\rfloor}

\newcommand{\ev}[2]{\left. {#1} \right\rvert_{#2}}

\newcommand{\tu}{\textup} 
\newcommand{\tuparens}[1]{\textup{(}#1\textup{)}} 

\DeclareMathOperator{\Hilb}{Hilb}
\DeclareMathOperator{\LT}{LT}
\DeclareMathOperator{\LM}{LM}

\newcommand{\Zhat}{\hhat{Z}}\DeclareMathOperator{\Coh}{Coh}
\DeclareMathOperator{\Quot}{Quot}
\DeclareMathOperator{\rk}{rk}
\newcommand{\m}{\mathfrak{m}}
\DeclareMathOperator{\Span}{span}
\newcommand{\Hilbnil}{\Hilb^0}
\renewcommand{\L}{\mathbb{L}}
\DeclareMathOperator{\LCM}{LCM}
\newcommand{\Ohat}{\hhat{\O}}
\DeclareMathOperator{\Nilp}{Nilp}
\newcommand{\cl}{\mathrm{cl}}
\newcommand{\cI}{\mathcal{I}}
\DeclareMathOperator{\Surj}{Surj}
\DeclareMathOperator{\Fil}{Fil}
\DeclareMathOperator{\gr}{gr}
\DeclareMathOperator{\Gr}{Gr}
\DeclareMathOperator{\Rel}{Rel}
\newcommand{\ul}{\underline}

\newcommand{\sbp}[1]{_{(#1)}}
\DeclareMathOperator{\Cont}{Cont}
\newcommand{\KVar}[1]{K_0(\mathrm{Var}_{#1})}
\newcommand{\mot}{\mathrm{mot}}
\newcommand{\NH}{\mathit{NH}}
\newcommand{\NQ}{\mathit{NQ}}

\newcommand{\KStck}[1]{{\hhat{K}_0(\mathrm{Var}_{#1})}}

\begin{document}
\title[Punctual Quot schemes and Cohen--Lenstra series of cusp]{Punctual Quot schemes and Cohen--Lenstra series of the cusp singularity}
\author{Yifeng Huang}
\address{Dept.\ of Mathematics, University of British Columbia}
\email{huangyf@math.ubc.ca}

\author{Ruofan Jiang}
\address{Dept.\ of Mathematics, University of Wisconsin--Madison}
\email{rjiang32@wisc.edu}

\date{\today}

\bibliographystyle{abbrv}

\begin{abstract}
The Quot scheme of points $\mathrm{Quot}_{d,n}(X)$ on a variety $X$ over a field $k$ parametrizes quotient sheaves of $\mathcal{O}_X^{\oplus d}$ of zero-dimensional support and length $n$. It is a rank-$d$ generalization of the Hilbert scheme of $n$ points. When $X$ is a reduced curve with only the cusp singularity $\{x^2=y^3\}$ and $d\geq 0$ is fixed, the generating series for the motives of $\mathrm{Quot}_{d,n}(X)$ in the Grothendieck ring of varieties is studied via Gröbner bases, and shown to be rational. Moreover, the generating series is computed explicitly when $d\leq 3$. The computational results exhibit surprising patterns (despite the fact that the category of finite length coherent modules over a cusp is wild), which not only enable us to conjecture the exact form of the generating series for all $d$, but also 
suggest a general functional equation whose $d=1$ case is the classical functional equation of the motivic zeta function known for any Gorenstein curve.  

As another side of the story, Quot schemes are related to the Cohen--Lenstra series. The Cohen--Lenstra series encodes the count of ``commuting matrix points'' (or equivalently, coherent modules of finite length) of a variety over a finite field, about which Huang \cite{huang2023mutually} formuated a ``rationality'' conjecture for singular curves. We prove a general formula that expresses the Cohen--Lenstra series in terms of the motives of the (punctual) Quot schemes, which together with our main rationality theorem, provides positive evidence for Huang's conjecture for the cusp.
\end{abstract}

\maketitle
\tableofcontents
\section{Introduction}
Rationality results are ubiquitous in generating series arising from moduli spaces. We consider several moduli spaces parametrizing sheaves with zero-dimensional support on a $k$-variety $X$, where $k$ is any field:

\begin{enumerate}
\item The Grothendieck Quot scheme $\Quot_{d,n}(X)$ parametrizing zero-dimensional quotient of $\O_X^{\oplus d}$ of length $n$, with generating series
\begin{equation}
Q_{d,X}^\mot(t):=\sum_{n\geq 0} [\Quot_{d,n}(X)]\, t^n \in \KVar{k}[[t]]
\end{equation}
with coefficients in the Grothendieck ring of $k$-varieties. 

\item The stack $\Coh_n(X)$ parametrizing zero-dimensional coherent sheaves of length $n$ on $X$, with generating series
\begin{equation}
\widehat{Z}_X^{\text{mot}}(t):=\sum_{n\geq 0}[\Coh_n(X)] t^n\in \KStck{k}[[t]]
\end{equation}
with coefficients in $\KStck{k}:=\KVar{k}[\L^{-1},(1-\L^{-1})^{-b}:b\geq 0] \subeq \KVar{k}[[\L^{-1}]]$. The stack  $[\Coh_n(X)]$ (or its punctual counterpart, see next paragraph) is a quotient stack of certain commuting variety by $\GL_n$, see \S\ref{sec:moduli}, so its motive is defined in the sense of Behrend--Dhillon \cite{behrenddhillon2007} or Ekedahl \cite{Ekedahl2009} and lives in $\KStck{k}$.
\end{enumerate}

When $p$ is a closed point of $X$, we analogously define the local (or \textbf{punctual}) versions $\Quot_{d,n}^\mot(X,p)$, $Q^\mot_{d,X,p}(t)$, $\Coh_n^\mot(X,p)$, $\Zhat_{X,p}^\mot(t)$ by requiring the zero-dimensional sheaves to be supported at $p$. These local series only depend on the completed local ring $\Ohat_{X,p}$. The global series $Q^\mot_{d,X}(t)$ and $\Zhat^\mot_{X}(t)$ are determined by the local ones at all closed points $p\in X$. This local-to-global principle can be written in terms of the notion of power structures due to \cite{guseinzade2004power}. If $k=\Fq$, the point-count version of the local-to-global principle can be more elementarily understood as Euler products (\cite[Proposition 4.2]{huang2023mutually} and Proposition \ref{prop:euler_quot}). The essential reason behind the local-to-global principle is simple: a zero-dimensional sheaf on $X$ is determined by its stalk at each point in its support. 

When $d=1$, the Quot scheme $\Quot_{d,n}(X)$ is the Hilbert scheme $\Hilb_n(X)$ of $n$ points on $X$, and the series $Q_{1,X}^\mot (t)$ is widely studied in the literature and is often called the \emph{\tuparens{motivic} Hilbert zeta function}. When $X$ is a smooth curve, the Hilbert zeta function is nothing but the motivic zeta function\footnote{When $k=\Fq$, taking the $\Fq$-point count specializes $Z^\mot_X(t)$ to the usual local zeta function \cite[Appendix C]{hartshorneAG}, where Weil's conjectures apply.}
\begin{equation}
Z^\mot_X(t):=\sum_{n\geq 0} [\Sym^n(X)] t^n \in \KVar{k}[[t]],
\end{equation}
which is rational and (when $X$ is projective) satisfies a functional equation analogous to the one in Weil's conjectures \cite{kapranov2000, litt2015zeta}. When $X$ is a smooth surface, it is well known that $\Hilb_n(X)$ is a desingularization of $\Sym^n(X)$, and the Hilbert zeta function is given by G\"{o}ttsche's formula \cite{goettsche2001motive} in terms of the motivic zeta function, which depends on a local result by Ellingsrud and Str\o{}mme \cite{ellingsrudstromme1987homology} using the Bia\l{}ynicki-Birula decomposition. See \cite{gns2017} for related research on singular surfaces.

For $d>1$ and $X$ is a smooth curve, the motive of $\Quot_{d,n}(X)$ is given in 1989 by Bifet \cite{bifet1989} in terms of the motivic zeta function of $X$:
\begin{equation}\label{eq:quot_smooth_curve}
Q_{d,X}^\mot(t)=Z_X^\mot(t)Z_X^\mot(\L t)\dots Z_X^\mot(\L^{d-1}t)\text{, if $X$ is a smooth curve.}
\end{equation}

In light of the local-to-global principle, the equivalent local formula is
\begin{equation}\label{eq:quot_smooth_point}
Q_{d,\A^1,0}^\mot(t)=\frac{1}{(1-t)(1-\L t)\dots (1-\L^{d-1}t)},
\end{equation}
where $(\A^1,0)$ is the origin on an affine line. (We only list the exact form of these two formula among other known formulas, since they will be helpful later in the introduction section.) Both formulas are in fact essentially equivalent to formulas due to Solomon \cite{solomon1977zeta} in 1977 that count full-rank lattices of $\Z^d$ and $\Zp^d$, as is observed and explained in \cite{huangjiang2022a}. Bifet's formula is later generalized to arbitrary vector bundles \cite{bfp2020motive} (in place of the trivial vector bundle $\O_X^{\oplus d}$), and the nested Quot schemes \cite{monavariricolfi2022}. Other works along this line consider high-rank-quotient versions of the Quot scheme \cite{bcs2008, chen2001hyperquot, stromme1987parametrized} (where (equivariant) Chow rings are often computed as well) or the case where $X$ is a smooth surface \cite{jop2021, op2021quot}.

The motive of the stack $\Coh_n(X)$, from an elementary point of view, is essentially counting commuting matrices \eqref{eq:stack_quotient}, at least when $X$ is affine and $k=\Fq$. When $X$ is a smooth curve, we have an infinite product formula for $\Zhat^\mot_X(t)$ in terms of $Z^\mot_X(t)$, essentially discovered by Cohen and Lenstra \cite{cohenlenstra1984heuristics}. When $X$ is a smooth surface, an infinite product formula for $\Zhat^\mot_X(t)$ can be found by ``globalizing'' a local result of Feit and Fine \cite{feitfine1960pairs} on counting pairs of commuting matrices. The globalization is treated in \cite{bryanmorrison2015motivic} using power structures and more elementarily in \cite{huang2023mutually} using Euler products.

The most active part of the story, however, is the Hilbert zeta function $Q_{1,X}^\mot(t)$ for a reduced singular curve $X$, which the present paper attempts to generalize. We partially follow the exposition in the lecture note \cite{gks2021link} and references therein. Let $X$ be a reduced singular curve over $k$, and for simplicity, we assume \textbf{for the rest of the paper} that $X$ has a normalization $\pi:\ttilde X\to X$ such that $\ttilde X$ is a finite disjoint union of geometrically connected smooth curves defined over $k$, and each singular point $p$ of $X$ is a $k$-point with $\pi^{-1}(p)$ consisting of only $k$-points. Whenever we talk about the analytic isomorphism class of a curve singularity $p$, we assume it has a model in such a curve $X$. Then the Hilbert zeta function of $X$ (or $(X,p)$) satisfies the following properties:
\begin{enumerate}
\item (Rationality; \cite{brv2020motivic}) The series $Q_{1,X}^\mot(t)$ is rational in $t$. Moreover, if $\ttilde X$ denotes the normalization of $X$, then $Q_{1,X}^\mot(t)/Q_{1,\ttilde X}^\mot(t)$ is a polynomial in $t$ that only depends on the singularities of $X$. Equivalently, if $p$ is a $k$-point of $X$, then the local factor $Q_{1,X,p}^\mot(t)$ is rational with denominator $(1-t)^s$, where $s$ is the number of branches through $p$. 

\begin{definition}
Define
\begin{equation}
\NQ_{1,X,p}^\mot(t):=(1-t)^s Q_{1,X,p}^\mot(t),
\end{equation}
where $s$ is the branching number of $p$. By the rationality above, $\NQ_{1,X,p}^\mot(t)$ is a polynomial in $t$. If the only singularity of $X$ is $p$, then we have
\begin{equation}
\NQ_{1,X,p}^\mot(t)=\frac{Q_{1,X}^\mot(t)}{Q_{1,\ttilde X}^\mot(t)},
\end{equation}
where $\ttilde X$ is the normalization of $X$.
\end{definition}

\item (Functional Equation; \cite{goettscheshende2014refined, maulikyun2013macdonald, pandharipandethomas2010}) If $X$ is Gorenstein, geometrically connected and projective with arithmetic genus $g_a$, then
\begin{equation}
Q_{1,X}^\mot(t) = (\L t^2)^{g_a-1} Q_{1,X}^\mot(\L^{-1}t^{-1})
\end{equation}
as rational functions in $t$. The proof uses the fact that the dualizing sheaf on a Gorenstein curve is a line bundle, and expresses $Q_{1,X}^\mot(t)$ in terms of a certain stratification of the compactified Jacobian $\bbar J(X)$ in the sense of Altman--Kleiman \cite{altmankleiman1, altmankleiman2}. The equivalent local statement is that the polynomial $\NQ_{1,X,p}(t)$ in $t$ has degree $2\delta$ and satisfies 
\begin{equation}
\NQ_{1,X,p}^\mot(t)=(\L t^2)^{\delta}\NQ_{1,X,p}^\mot(\L^{-1}t^{-1}),
\end{equation}
where $\delta=\dim_k \ttilde \O_{X,p}/\O_{X,p}$ is the $\delta$-invariant of the singularity $p$, where $\ttilde \O_{X,p}$ is the integral closure of $\O_{X,p}$. 

\item (Exact formulas and Oblomkov--Rasmussen--Shende conjecture \cite{ors2018homfly}) If $p$ is a planar singularity on a reduced curve $X$ over $k=\C$, the Hilbert scheme $\Hilb_n(X,p)$ is expected to have a affine cell decomposition, and \cite{ors2018homfly} conjectured an exact formula for $Q_{1,X,p}^\mot(t)$ in terms of the HOMFLY homology of the associated algebraic link.  
The ORS conjecture is verified for the singularity $\C[[t^m,t^n]], \gcd(m,n)=1$ with one Puiseux pair  \cite{gorskymazin2013} and the singularity $\C[[t^{nd},t^{md}+at^{md+1}+\dots]],\gcd(m,n)=1,a\neq 0$ with two Puiseux pairs \cite{gmo2022generic}. See also a related computation \cite{kivinen2020unramified} for the affine Springer fiber associated to the multibranch singularity $\set{x^n=y^{dn}}$. In these examples, the relevant spaces do have affine cell decompositions.

As for planar singularities in general, the Euler-characteristic specialization of the ORS conjecture is proposed by \cite{oblomkovshende2012} and is proved by Maulik \cite{maulik2016stable}. A recent result \cite[Corollary 1.11]{kivinentsai2022} by Kivinen and Tsai implies that $\Hilb_n(X,p)$ for any planar singularity is polynomial-count in the sense of Katz \cite{hausel2008mixed}, i.e., the point counts of $\Hilb_n(X,p)$ are governed by its weight polynomial; moreover, the weight polynomial has nonnegative integer coefficients. The method is based on harmonic analysis and representation theory, and \emph{not} stratification; in particular, it does not imply that the motive of $\Hilb_n(X,p)$ is a polynomial in $\L$. 
\end{enumerate}

\begin{remark}
Many of the cited results or conjectures in the literature are about various quantities that are almost (but not precisely) equivalent to $Q_{1,X,p}^\mot(t)$. These versions vary in the following aspects:

Two spaces closely related to $\Hilb_n(X,p)$ are the compactified Jacobians \cite{maulikyun2013macdonald, piontkowski2007topology} and the type-A affine Springer fibers \cite{gkm2004homology, gkm2006purity}.

Three notions closely related related to the motive are the point count over $\Fq$, the Poincar\'e polynomial and the weight polynomial; all of these are equivalent for a variety has an affine cell decomposition. In general, the toolsets of approaches suitable for each of these notions are different. 

The notion of Euler characteristic is strictly coarser than all of the above; for a variety with a cell decomposition, the Euler characteristic records the number of cells but not their dimensions. 
\end{remark}

Having discussed $Q_{1,X}^\mot(t)$ for reduced singular curves, we now move on to discuss $\Zhat_{X}^\mot(t)$. In \cite{huang2023mutually}, the first author discovered an analogous ``rationality'' phenomenon for $\Zhat_{X,p}^\mot(t)$, where $(X,p)$ is a node singularity $\set{xy=0}$. The ``rationality'' is put in quote because the ``denominator'' is in general an infinite product and the ``numerator'' is in general an infinite series in $\KStck{k}[[t]]$. Lacking a more suitable notion of when an infinite series qualifies as a ``numerator'', we work over finite fields and let $\#_q: \KStck{k} \to \Q$ be the natural ring homomorphism induced by taking the $\Fq$-point count. We restate the conjecture formulated by the first author. 

\begin{conjecture}[\cite{huang2023mutually}]
Let $X$ be a reduced curve over $\Fq$, and let $\ttilde X$ be its normalization. Then
\begin{equation}
\#_q \parens*{\frac{\Zhat^\mot_X(t)}{\Zhat^\mot_{\ttilde X}(t)}} \in \Q[[t]]
\end{equation}
is a power series in $t$ with infinite radius of convergence.
\label{conj:cohen_lenstra}
\end{conjecture}

The first author proved that Conjecture \ref{conj:cohen_lenstra} holds when $X$ has only nodal singularities $\set{xy=0}$, by expressing the ``numerator'' as an explicit infinite series in $\L^{-1}$ and $t$. As a consequence of this explicit expression, he showed that for any real number $q>1$, its specialization at $\L\mapsto q$ has infinite radius of convergence. His method is based on the observation that counting pairs of matrices with $AB=BA=0$ over a finite field turns out to be achievable by elementary means. 

Now we discuss $Q_{d,X}^\mot(t)$, the focus of the present paper. We propose that the rationality statement and the functional equation for $Q_{d,X}^\mot(t)$ should take the following form, generalizing the known results for $d=1$. We state both the global and the local versions below for ease of reference, and their equivalence can be directly verified using \eqref{eq:quot_smooth_curve} and \eqref{eq:quot_smooth_point}. 

\begin{conjecture} \label{conj:quot}
Let $X$ be a reduced singular curve over a field $k$, and $d\geq 0$. Then we have the following properties for $Q^\mot_{d,X}(t)\in \KVar{k}[[t]]$.
\begin{enumerate}
\item \tu{(Rationality)} The series $Q^\mot_{d,X}(t)$ is rational in $t$. Moreover, if $\ttilde X$ is the normalization of $X$, then $Q_{d,X}^\mot(t)/Q_{d,\ttilde X}^\mot(t)$ is a polynomial in $t$ that only depends on $d$ and the singularities of $X$. Equivalently, if $p$ is a $k$-point of $X$, then the local factor $Q_{d,X,p}^\mot(t)$ is rational with denominator $\parens[\big]{(1-t)(1-\L t)\dots (1-\L^{d-1}t)}^s$, where $s$ is the number of branches through $p$.

\begin{definition}
Define a power series
\begin{equation}\label{eq:nquot_def}
\NQ_{d,X,p}^\mot(t):=\parens[\big]{(1-t)(1-\L t)\dots (1-\L^{d-1}t)}^s Q_{d,X,p}^\mot(t)\in \KVar{k}[[t]],
\end{equation}
where $s$ is the branching number of $p$. The rationality conjecture is equivalent to saying that $\NQ_{d,X,p}^\mot(t)$ is a polynomial in $t$. If the only singularity of $X$ is $p$, then we have
\begin{equation}
\NQ_{1,X,p}^\mot(t)=\frac{Q_{1,X}^\mot(t)}{Q_{1,\ttilde X}^\mot(t)},
\end{equation}
where $\ttilde X$ is the normalization of $X$. This equality is unconditional on the rationality conjecture.
\end{definition}

\item \tu{(Functional Equation)} If $X$ is Gorenstein and projective with arithmetic genus $g_a$, then
\begin{equation}\label{eq:quot_func_eq_global}
Q_{d,X}^\mot(t) = (\L^{d^2} t^{2d})^{g_a-1} Q_{d,X}^\mot(\L^{-d}t^{-1})
\end{equation}
as rational functions in $t$. Equivalently, $\NQ_{d,X,p}(t)$ is a polynomial in $t$ of degree $2\delta$ and it satisfies
\begin{equation}\label{eq:quot_func_eq_punctual}
\NQ_{d,X,p}^\mot(t)=(\L^{d^2} t^{2d})^{\delta}\NQ_{d,X,p}^\mot(\L^{-d}t^{-1}),
\end{equation}
where $\delta$ is the $\delta$-invariant of the singularity $p$. 
\end{enumerate}
\end{conjecture}

{
As far as we know, no $d\geq 2$ examples for singular curves have been explored or conjectured. Difficulties of some ideas commonly used for the Hilbert zeta function are discussed in Remark \ref{rmk:method_difficulty}.}

\subsection{Main results}

As the main theorem, we prove the following, which is what leads to the form of the functional equation proposed above. Recall that the cusp singularity has branching number $s=1$ and $\delta$-invariant $\delta=1$.

\begin{theorem}\label{thm:quot_cusp}
If $(X,p)$ is the cusp singularity $\set{x^2=y^3}$, then $\NQ_{d,X,p}(t)$ satisfies Conjecture \ref{conj:quot}(a) for all $d\geq 0$, and Conjecture \ref{conj:quot}(b) for $0\leq d\leq 3$. Moreover, we have
\begin{equation}
\begin{aligned}
\NQ_{0,X,p}^\mot(t)& = 1; \\
\NQ_{1,X,p}^\mot(t)& = 1+\L t^2;\\
\NQ_{2,X,p}^\mot(t)& = 1 + (\L^2+\L^3) t^2+\L^4 t^4;\\
\NQ_{3,X,p}^\mot(t)& = 1 +(\L^3+\L^4+\L^5) t^2 + (\L^6+\L^7+\L^8) t^4+\L^9 t^6.
\end{aligned}
\end{equation}
\end{theorem}

The formulas for $\NQ_{2,X,p}^\mot(t)$, $\NQ_{3,X,p}^\mot(t)$ and the $t$-polynomiality of $\NQ_{d,X,p}^\mot(t)$ are new. 

The first step of our proof of Theorem \ref{thm:quot_cusp} involves expressing the motive of $\Quot_{d,n}(X,p)$ in terms of the motives of some related moduli spaces. For $0\leq r\leq \min(d,n)$, consider the scheme $\Quot_{d,n}^r(X,p)$ parametrizing zero-dimensional quotients of $\O_X^{\oplus d}$ of length $n$ supported at $p$ that require exactly $r$ sections to generate. For $0\leq r\leq n$, consider the stack $\Coh_n^r(X,p)$ parametrizing zero-dimensional coherent sheaves of length $n$ supported at $p$ that require exactly $r$ sections to generate. Their rigorous definitions via functors of points are given in Appendix \ref{appendix:motivic}. Consider the generating series
\begin{equation}
H_{d,X,p}^\mot(t):=\sum_{n\geq 0} [\Quot_{d,n+d}^d(X,p)] \, t^n \in \KVar{k}[[t]]
\end{equation}
and
\begin{equation}
\Zhat_{d,X,p}^{\mot}(t)=\sum_{n\geq 0} [\Coh_n^d(X,p)]\, t^n  \in \KStck{k}[[t]].
\end{equation}

We use the $q$-Pochhammer symbol
\begin{align}
(a;q)_n&=(1-a)(1-aq)\dots (1-aq^{n-1})\\
(a;q)_\infty&=(1-a)(1-aq)(1-aq^2)\dots
\end{align}
and consider the $q$-binomial coefficient for $d\geq r$:
\begin{equation}
{d \brack r}_q=\frac{(q;q)_d}{(q;q)_r (q;q)_{d-r}}.
\end{equation}

\begin{theorem}\label{Thm:Main2}
Let $X$ be any $k$-variety and $p$ be a $k$-point of $X$. For any $d\geq r\geq 0$, we have the following identity in $\KStck{k}$:
\begin{equation}\label{eq:coh_in_quot_mot}
[\Coh_n^r(X,p)] = \frac{[\Quot_{d,n}^r(X,p)]}{\mathbb{L}^{d(n-r)}[\Gr(r,d)][\GL_r]},
\end{equation}
where $\Gr(r,d)$ is the Grassmannian variety.

In terms of generating series, we have
\begin{equation}\label{eq:quot_in_hilb_mot}
Q_{d,X,p}^\mot(t)=\sum_{r=0}^d {d \brack r}_{\mathbb{L}} t^r H_{r,X,p}^\mot(\L^{d-r} t)
\end{equation}
and
\begin{equation}
\Zhat_{d,X,p}^\mot(t)=\frac{t^d}{\L^{d^2}(\L^{-1};\L^{-1})_d}\,H_{d,X,p}^\mot(\L^{-d}t).
\end{equation}
\end{theorem}

The point-count version of Theorem \ref{Thm:Main2} is in fact an elementary observation, which we state in Corollaries \ref{cor:quot_to_coh_const_fiber} and \ref{cor:series_conversion} and give a self-contained proof. However, the motivic version requires some extra care, which we handle in Appendix \ref{appendix:motivic}.

Having Theorem \ref{Thm:Main2}, proving Theorem \ref{thm:quot_cusp} reduces to proving the following result. It is routine to verify that Theorem \ref{thm:cusp_t_rationality} implies Theorem \ref{thm:quot_cusp} given \eqref{eq:quot_in_hilb_mot}.

\begin{theorem}
\label{thm:cusp_t_rationality}
If $(X,p)$ is the cusp singularity $\set{x^2=y^3}$, then $(1-t)(1-\L t)\dots (1-\L^{d-1}t)H_{d,X,p}^\mot(t)$ is a polynomial in $t$, which we denote by $\NH_{d,X,p}^\mot(t)$, and we have
\begin{equation}
\begin{aligned}
\NH_{0,X,p}^\mot(t)& = 1; \\
\NH_{1,X,p}^\mot(t)& = 1+\L t;\\
\NH_{2,X,p}^\mot(t)& = 1 + (\L^2+\L^3) t+\L^4 t^2;\\
\NH_{3,X,p}^\mot(t)& = 1 +(\L^3+\L^4+\L^5) t + (\L^6+\L^7+\L^8) t^2+\L^9 t^3.
\end{aligned}
\end{equation}
\end{theorem}

\subsection{Further conjectures}
Assume $(X,p)$ is the cusp singularity from now on. We observe from the above data that
\begin{equation}\label{eq:nq_in_nh}
\NQ_{d,X,p}^\mot(t)=\NH_{d,X,p}^\mot(t^2)
\end{equation}
for $0\leq d\leq 3$. The resemblance between $\NH_{d,X,p}^\mot(t)$ and $\NQ_{d,X,p}^\mot(t)$ for $d\leq 3$ is not expected, given that $Q_{d,X,p}^\mot(t)$ is computed as a weighted summation of $H_{r,X,p}^\mot(t)$ over $0\leq r\leq d$. The relation \eqref{eq:nq_in_nh} does not hold for the node singularity $\set{xy=0}$, since for this singularity we have $\NQ^\mot_1(t)=1-t+\L t^2$ and $\NH^\mot_1(t)=1-t+\L t$. It is unclear for what other curve singularities we have an analogue of \eqref{eq:nq_in_nh}.

In fact, the pattern \eqref{eq:nq_in_nh} is so strong that if true for all $d$, it would \emph{overdetermine} the exact formulas of $H_{d,X,p}^\mot(t)$ for all $d$ (the word ``overdetermine'' is in the sense of solving overdetermined linear system); the combinatorics is handled in Proposition \ref{prop:q-series}. We conjecture the following formulas of striking simplicity; Proposition \ref{prop:q-series} shows that Conjecture \ref{conj:quot}(b), Theorem \ref{Thm:Main2} and the conjectural \eqref{eq:nq_in_nh} would imply Conjecture \ref{conj:cusp_exact_formula}.

\begin{conjecture}
\label{conj:cusp_exact_formula}
If $(X,p)$ is the cusp singularity $\set{x^2=y^3}$, then we have
\begin{enumerate}
\item ${\displaystyle
\NQ_{d,X,p}^\mot(t):=\sum_{j=0}^d {d \brack j}_{\L} (\L^d t^2)^j}$.
\item ${\displaystyle\Zhat_{X,p}^\mot(t)=\frac{1}{(\L^{-1}t;\L^{-1})_\infty} \sum_{n=0}^\infty \frac{\L^{-n^2}t^{2n}}{(\L^{-1};\L^{-1})_n}.}$
\end{enumerate}

\end{conjecture}

In light of Proposition \ref{prop:q-series}, the existence of a guess in Conjecture \ref{conj:cusp_exact_formula}(a) can be understood as an assertion that Conjecture \ref{conj:quot}(b), Theorem \ref{Thm:Main2} and the conjectural \eqref{eq:nq_in_nh} are consistent in the cusp case. Due to the overdetermined nature of this combination of theorems and conjectures, the mere consistency would be a strong evidence that \eqref{eq:nq_in_nh} is expected to hold for all $d$.

Conjecture \ref{conj:cusp_exact_formula}(b) gives a guess of $\Zhat_{X,p}^\mot(t)$ for the cusp singularity for the first time. Since its ``numerator'' $\displaystyle{\sum_{n=0}^\infty \frac{\L^{-n^2}t^{2n}}{(\L^{-1};\L^{-1})_n}}$ substituted with $\L\mapsto q$ is an entire power series in $t$ for all $q>1$, Conjecture \ref{conj:cusp_exact_formula}(b) would verify the cusp case of Conjecture \ref{conj:cohen_lenstra}. See also \eqref{eq:matrix_count} for a concrete consequence of Conjecture \ref{conj:cusp_exact_formula}(b) in matrix counting.

\

We now take a side step. Note that Conjecture \ref{conj:cusp_exact_formula}(a), if true, would imply that $NQ_{d,X,p}^\mot(t)$ is a polynomial in $\L$ for all $d$, so that $\Quot_{d,n}(X,p)$ is a polynomial in $\L$ for all $d,n$. We formulate a weaker conjecture:

\begin{conjecture}[Weak form of Conjecture \ref{conj:cusp_exact_formula}]
Let $(X,p)$ be the cusp singularity $\set{x^2=y^3}$. Then the motive of $\Quot_{d,n}(X,p)$ is a polynomial in $\L$. \label{conj:L_rationality}
\end{conjecture}

It turns out that this statement is implied by Conjecture \ref{conj:L_rationality_strata} about commuting matrices in certain upper triangular matrix algebras determined by posets. This conjecture motivates the following theorem as an important special case we are able to prove. This is an unexpected side product of our main goal and is of separate interest.

\begin{theorem}[Corollary \ref{cor:staircase_formula}]
Consider the \tuparens{reduced} $k$-variety cut out by
\begin{equation}
V_d:=
\set*{
(X,Y)\in \Mat_d(k)^2
\ \vrule \
\begin{array}{l}
\text{$X,Y$ are strictly upper triangular}\\
X^2=Y^3, XY=YX
\end{array}
}.
\end{equation}

Then the motive of $V_d$ is a polynomial in $\L$. Moreover, there is an inductive formula for $[V_d]$ given in Corollary \ref{cor:staircase_formula}.
\label{thm:main_staircase}
\end{theorem}

\subsection{Methods and strategies}
We first overview some existing methods to compute the motive of Hilbert and Quot schemes. The following types of methods are all based on classification or stratification (i.e., locally closed decompoisition), as is what the definition of motive requires, though their flavors vary.
\begin{enumerate}
\item Local method: Work on the punctual Quot scheme $\Quot_{d,n}(X,p)$ and solve the commutative algebra problem of classifying ideals of the complete local ring $\Ohat_{X,p}$, or submodules of $\Ohat_{X,p}^{\oplus d}$;
\item Global method: Choose a projective model $X$ and use global properties for torsion free sheaves on $X$, such as Riemann--Roch and Serre duality;
\item Torus action: Choose a model $X$ that admits a ``good enough'' torus action, and make use of Bia\l ynicki-Birula decomposition.
\end{enumerate} 

Method (a) was used in example-based computations of $Q_{1,X,p}^\mot(t)$ for curve singularities (such as \cite{oblomkovshende2012}) and the proof of the general rationality result \cite{brv2020motivic}, to be elaborated later. Method (b) was used in the proof of rationality and the function equation for $Q_{1,X}^\mot(t)$ for Gorenstein curves \cite{goettscheshende2014refined, maulikyun2013macdonald, pandharipandethomas2010}. Method (c) was used in the computation of $Q_{1,\P^2}^\mot(t)$ in \cite{ellingsrudstromme1987homology} and generalizations of $Q_{d,\P^1}^\mot(t)$ in \cite{monavariricolfi2022}.

The main theorem we are proving is the rationality statement of Theorem \ref{thm:cusp_t_rationality}. Our proof belongs to Method (a).\footnote{See Remark \ref{rmk:method_difficulty} for difficulties of Methods (b)(c).} More precisely, we use the theory of Gr\"obner bases over power series rings (also known as ``standard bases'' in the terminology of Hironaka \cite{hironaka1964resolution}) to classify finite-$k$-codimensional submodules of the rank-$d$ free module over $\Ohat_{X,p}=k[[T^2,T^3]]$. We classify by splitting into cases (i.e., ``strata''), and identify a parametrization for each case. In general, the method of Gr\"obner bases would lead to tremendous computational complications, especially those caused by the Buchberger criterion. The complexity is not only in combinatorics (counting the number of parameters in a stratum), but also in completely determining the relation between the parameters in a given stratum. The computation sensitively depends on some artificial choices, such as the monomial order. We use an ``hlex'' monomial order such that the power of $T$ is more important than the ordering of the basis vectors of $k[[T^2,T^3]]^d$, see \S \ref{sec:strata}. Furthermore, we classify submodules of $(T^2,T^3)k[[T^2,T^3]]$ (which directly leads to $H_{d,X,p}^\mot(t)$) instead of submodules of $k[[T^2,T^3]]$ (which directly leads to $Q_{d,X,p}^\mot(t)$). Fortunately, these choices lead to manageable computations, and we are able to determine completely the relation between the parameters in a given stratum (Lemmas \ref{lem:decomposition} and \ref{lem:staircase}). 

We compare our method to the proof of the BRV theorem \cite{brv2020motivic} that $Q_{1,X,p}^\mot(t)$ is rational for any reduced curve singularity $(X,p)$. For simplicity, say $(X,p)$ is a unibranch singularity, and write $\O=\O_{X,p}$ and $\Ohat=\Ohat_{X,p}\cong k[[T]]$. Consider a ``branch-length stratification'' by defining the stratum $\Hilb_n(X,p;e)$ to consist of ideals $I$ of $\O$ such that $\dim_k \O/I=n$ and $\dim_k \Ohat/I\Ohat=e$ (namely, $I\Ohat=(T^e)k[[T]]$). The rationality of $Q_{1,X,p}^\mot(t)$ is then reduced to the following stabilization statement:
\begin{equation}
[\Hilb_{n+1}(X,p;e+1)]=[\Hilb_n(X,p;e)]\text{ for }e\gg 0.
\end{equation}

For multibranch singularities, the proof of the BRV theorem uses induction on the branching number whose induction step is based on a similar stabilization statement.

Our stratification is based on Gr\"obner theory discussed above. We denote our strata by $\Hilb(\alpha)$, where $\alpha$ is any $d$-tuple of nonnegative integers, each decorated by one of two available colors. The datum $\alpha$ encodes a submodule of $(T^2,T^3)k[[T^2,T^3]]$ generated by monomials, and $\Hilb(\alpha)$ consists of submodules of $(T^2,T^3)k[[T^2,T^3]]$ whose leading term submodule corresponds to $\alpha$. We are able to prove that the motive of $\Hilb(\alpha)$ is eventually stable up to a correction factor (Lemma \ref{lem:orbit} and Lemma \ref{lem:is_stable}) with respect to certain ``spiral shifting'' operators $\gamma\sbp 1,\dots,\gamma\sbp d$ that act on the index set of $\alpha$ (illustrated in Figure \ref{fig:spiral}). In particular, the stability is not with respect to raising one of the components of the $d$-tuple by one, but a mixture of permuting and raising the components. The spiral shifting operators originally arise in a similar stratification of the punctual Quot scheme of $(\A^1,0)$ due to the authors \cite{huangjiang2022a}.

To finish the comparison with BRV's proof, note that when $d=1$ (BRV's situation), the only operator $\gamma\sbp 1$ is $n\mapsto n+1$, which is, in combinatorial essence, the same as in BRV's proof. In effect, we answered what the high-$d$ generalization of BRV's stabilization statement should look like: in the case of the cusp singularity, the combinatorics of the corresponding stabilization statement involves the operators $\gamma\sbp j$. 

We conclude with a discussion about how our method is expected to generalize to arbitrary reduced curve singularities $(X,p)$. We likely need a stabilization statement for the unibranch case, and then generalize to the multibranch case using induction. It is expected that the stratification is obtained in a way inspired by Gr\"obner basis, but it is unclear what version of stratification is computationally manageable in general. However, the combinatorial nature of the stabilization seems to be not sensitive to the engineering details of the stratification: Our proof suggests that for a general unibranch singularity $(X,p)$, there should be a similar stratification for $\Quot_{d,n}(X,p)$ with strata indexed by some decorated version of $d$-tuples of integers (so the operators $\gamma\sbp j$ are defined), such that an eventual stabilization statement for stratum motives holds with respect to $\gamma\sbp j$. 

\begin{remark}\label{rmk:method_difficulty}
It is not clear how to use Method (c) in our case since $\Quot_{d,n}(X,p)$ is not smooth in general. {Even though Bia\l ynicki-Birula style decompositions also exist for certain singular varieties, such as a nodal or a cuspidal rational curve}, it is not clear that our $\Quot_{d,n}(X,p)$ is of that type.

The conjectured functional equation for $Q_{d,X,p}^\mot(t)$ for the cusp singularity suggests a potential proof using Method (b), but it is currently not clear how: a direct attempt to generalize the proof of the $d=1$ case in \cite{goettscheshende2014refined, maulikyun2013macdonald, pandharipandethomas2010} would involve identifying the moduli space of rank-$d$ torsion free sheaves on a Gorenstein curve $X$ (as a high-rank generalization of the compactified Jacobian, namely, the compactified $\mathcal{B}un_{\GL_d}$) and putting a suitable stratification on it which behaves well with Serre duality and is simple enough so that one can effectively extract information to determine $Q_{d,X}^\mot(t)$. One essential observation of \cite{goettscheshende2014refined} is that the embeddings of a rank $1$ torsion free sheaf $\mathcal{L}$ into $\mathcal{O}_X$ are classified by $H^0(X,\mathcal{L}^{\vee})-\{0\}$, which leads to a nice stratification of the compactified Jacobian by cohomological invariants. But this fails for higher rank torsion free sheaves, namely, the embeddings of a rank $d$ torsion free sheaf $\mathcal{E}$ into $\mathcal{O}_X^{d}$ is usually only a Zariski open subset $U_\mathcal{E}\subseteq H^0(X,\mathcal{E}^{\vee})-\{0\}$. At this moment, it is hard to give a satisfactory (geometrical or cohomological) characterization of $\mathcal{E}$ with fixed isomorphism class of $U_\mathcal{E}$. One may expect a stratification over the moduli space of rank-$d$ torsion free sheaves depending on the isomorphism classes of $U_\mathcal{E}$. But it seems to behave badly with Serre duality. 

We also give a remark on the complexity of torsion free sheaves (with no restraint on the rank) over a cusp. This is relevant to us because Fourier--Mukai transform establishes an equivalence between degree 0 semistable torsion free sheaves over a cuspidal cubic with the category of torsion sheaves, see \cite{Titus1999, LIYG06}. The innate non-semisimplicity of the category of torsion free sheaves is merely one aspect of its complexity. Even for indecomposable semistable torsion free sheaves, wild phenomena already happen. For example, the category of higher rank torsion free sheaves over a cuspidal curve is wild, in the sense that there are families of indecomposable torsion free sheaves depending on any prescribed number of parameters, see \cite{drozdgreuel2001}. The category of torsion sheaves supported at the singular point is also wild, see \cite{drozd1972}. These two facts are compatible from the theory of Fourier--Mukai transform. These facts suggest that studying $Q_{d,X}^\mot(t)$ from a representation theoretical point of view, or from Fourier--Mukai transform, may be hard. \end{remark}

\begin{remark}
In principle, our method involving Gr\"obner bases can be applied to any singularity $k[[T^m, T^n]]$ with $\gcd(m,n)=1$ with one Puiseux pair, but it is unclear whether the Buchberger criterion is manageable in any case other than $(m,n)=(2,3)$. Our Lemma \ref{lem:decomposition} and Lemma \ref{lem:staircase} seem coincidental and we do not know what their generalizations should look like. It might still be possible, though, to attack the rationality conjecture with an eventual stabilization statement analogous to Lemma \ref{lem:is_stable} without understanding the Gr\"obner strata as explicitly as in Lemmas \ref{lem:decomposition} and \ref{lem:staircase}.
\end{remark}

\subsection{Organization and structure of proof}
We take an unusual approach in presenting the proof, namely, we focus on the set of $k$-points and disregard the geometry when working with various moduli spaces. In particular, we will give a self-contained proof of the point-count version of our theorems that is accessible to readers with background in commutative algebra and basic algebraic geometry. We do so to avoid distractions and highlight the already complicated details in the Gr\"obner stratification and the combinatorics.

Since the proof is based on case-by-case parametrization, most of the proof directly translates to a locally closed stratification, except a few places that require extra care, which we handle in Appendix \ref{appendix:motivic}.

We introduce the preliminaries of the relevant moduli spaces in \S \ref{sec:moduli}, followed by an introduction of the needed theorems about Gr\"obner basis theory in \S \ref{sec:groebner} (with proofs in Appendix \ref{appendix:groebner}). Equipped with these, we give a complete description of our Gr\"obner stratification of $\Quot_{d,n+d}^d(X,p)$ for the cusp in \S \ref{sec:strata}. Next, in \S \ref{sec:combinatorics}, we sort out the combinatorics and finish the proof of the first part of Theorem \ref{thm:cusp_t_rationality}. 

In \S \ref{sec:staircase}, we take a detour to prove Theorem \ref{thm:main_staircase}; the proof in this section is independent of the rest of the paper. In \S \ref{sec:low_d}, we perform explicit computations of $Q_{d,X,p}^\mot(t)$ for the cusp, where $d\leq 3$. This finishes the proof of Theorem \ref{thm:cusp_t_rationality}. Further conjectures, including Conjecture \ref{conj:cusp_exact_formula}, are discussed in \S \ref{sec:conj}.

\subsection*{Acknowledgements}
The authors thank Dima Arinkin, Dori Bejleri, Jim Bryan, Daniel Erman, Asvin G, Eugene Gorsky, Nathan Kaplan and Yifan Wei for fruitful conversations. Huang thanks AMS-Simons Travel Grant. Jiang thanks NSF grant DMS-2100436.

\section{Some moduli spaces}\label{sec:moduli}
Let $k$ be a field and $X$ be a variety over $k$ (a separated $k$-scheme of finite type that is not necessarily smooth, proper or reduced). Let $p$ be a closed point of $X$, and consider the completed local ring $\hhat \O_{X,p}$. We recall the set of $k$-points of several moduli spaces over $X$ (or $(X,p)$), and their point counting over $k=\Fq$. The geometric definitions of these moduli schemes or stacks are given in Appendix \ref{appendix:motivic}.

We will organize the moduli spaces into two types, unframed and framed, following the usual treatment of quiver varieties by Nakajima \cite{nakajima1999hilbert}.

\subsection{Unframed moduli spaces}\label{subsec:uramifiedmodulispcae}
We follow \cite{bryanmorrison2015motivic} and \cite{huang2023mutually}. Recall that $\Coh_n(X)$ is the moduli stack of coherent sheaves $M$ on $X$ supported at $n$ points counting multiplicity (namely, $M$ has finite support and the space of global section has dimension $n$). Point counting of $\Coh_n(X)$ over $k=\Fq$ can be understood via \textbf{stacky sets}. A stacky set\footnote{Equivalently, a stacky set is the set of isomorphism classes of a finite groupoid with finite automorphism groups. We choose to delay the use of groupoids until Appendix \ref{appendix:motivic} to keep the point counting here transparent.} is a finite set where each point $x$ is equipped with a finite group $\Aut x$, called the automorphism group of $x$. The cardinality of a stacky set $\mathcal X$ is defined as
\begin{equation}
\abs{\mathcal X}^\mathrm{stck}:=\sum_{x\in \mathcal X} \frac{1}{\abs{\Aut x}}.
\end{equation}
 
Say $k=\Fq$. We think of the set of $k$-points of $\Coh_n(X)$ as a stacky set consisting of coherent sheaves on $X$ supported at $n$ points counting multiplicity up to isomorphism, each equipped with the automorphism group of the sheaf. When $X=\Spec R$ is affine, $\Coh_n(X)$ is canonically isomorphic to the stacky set of $R$-modules $M$ such that $\dim_k M=n$ up to isomorphism, equipped with the group of $R$-linear automorphisms of $M$. The point count of $\Coh_n(X)$ over $\Fq$ is defined as the stacky cardinality
\begin{equation}
\abs{\Coh_n(X)(\Fq)}:=\sum_{M\in \Coh_n(X)} \frac{1}{\abs{\Aut M}}.
\end{equation}

The stack $\Coh_n(X)$ can be realized as a quotient stack $[C_n(X)/\GL_n]$, where $C_n(X)$ is a well-defined scheme; the geometry is discussed in Appendix \ref{appendix:motivic}. If $X$ is affine, then the set of $k$-points of $C_n(X)$ has an explicit description as follows: say
\begin{equation}
X=\Spec R=\Spec \frac{k[T_1,\dots,T_m]}{(f_1,\dots,f_r)},
\end{equation}
then the set of $k$-points of $C_n(X)$ is the set of commuting matrices satisfying equations of $X$:
\begin{equation}
C_n(X)(k):=\set*{\underline{A}=(A_1,\dots,A_m)\in \Mat_n(k)^m: [A_i,A_j]=0, f_i(\underline{A})=0},
\end{equation}
and $\GL_n$ acts by simutaneous conjugation. For this reason, we call $C_n(X)$ the $n$-th \textbf{commuting variety} over $X$. A coordinate-free description is
\begin{equation}
C_n(X)(k)=\Hom_{\mathrm{AssoAlg}/k}(R,\Mat_n(k)),
\end{equation}
so one may as also think of $C_n(X)(k)$ as ``$X(\Mat_n(k))$'', or $n\times n$ matrix points on $X$.

The stack quotient $\Coh_n(X)=[C_n(X)/\GL_n]$, in terms of point count, reads
\begin{equation}\label{eq:stack_quotient}
\abs{\Coh_n(X)(\Fq)}=\frac{\abs{C_n(X)(\Fq)}}{\abs{\GL_n(\Fq)}}.
\end{equation}

This is in fact an elementary consequence of the orbit-stabilizer theorem.

Now we recall the punctual version of the above. Let $p\in X$ be a closed point, and $R=\Ohat_{X,p}$ be the completed local ring of $X$ at $p$. Let $\Coh_n(X,p)$ be moduli stack of coherent sheaves on $X$ supported at $p$ with multiplicity $n$. Let $\Coh_n(R)$ be the moduli stack of $R$-modules $M$ with $\dim_k M=n$. Then we have a canonical isomorphism $\Coh_n(X,p)\cong \Coh_n(R)$. 

If
\begin{equation}
R=\frac{k[[T_1,\dots,T_m]]}{(f_1,\dots,f_r)},
\end{equation}
with $f_i(0,\dots,0)=0$ for all $i$, define
\begin{equation}
C_n(X,p)=C_n(R):=\set*{\underline{A}=(A_1,\dots,A_m)\in \Nilp_n(k)^m: [A_i,A_j]=0, f_i(\underline{A})=0},
\end{equation}
where $\Nilp_n(k)$ is the variety of $n\times n$ nilpotent matrices over $k$. Then $\Coh_n(X,p)=[C_n(X,p)/\GL_n]$ as a stack quotient. 

The global and the punctual versions of stacks of zero-dimensional sheaves satisfy a local-to-global formula. If $k=\Fq$, the first author defined the \textbf{Cohen--Lenstra series}
\begin{align}
\Zhat_X(t)&:=\sum_{n\geq 0}\abs{\Coh_n(X)(\Fq)} t^n;\label{eq:cohen_lenstra_in_coh}\\
\Zhat_{X,p}(t)&:=\sum_{n\geq 0}\abs{\Coh_n(X,p)(\Fq)} t^n,
\end{align}
and proved an Euler product formula (\cite[Proposition 4.2]{huang2023mutually})
\begin{equation}
\Zhat_X(t)=\prod_{p\in X_\cl} \Zhat_{X,p}(t),
\end{equation}
where $p$ ranges over all closed points of $X$. 

\subsection{Framed moduli spaces}
For $n,d\geq 0$, recall that $\Quot_{d,n}(X)=\Quot_{\O_X^d,n}$ is the Quot scheme parametrizing quotient sheaves of $\O_X^d$ supported at $n$ points counting multiplicity. If $d=0$, then $\Quot_{d,n}(X)$ is a point if $n=0$ and empty if $n>0$. If $d=1$, then $\Quot_{1,n}(X)$ is the Hilbert scheme $\Hilb_n(X)$. The set of $k$-points is described as
\begin{equation}
\Quot_{d,n}(X)(k):=\set*{(M,f) \bigg\vert M\in \Coh_n(X), f: \O_X^d\onto M}\bigg/\sim,
\end{equation}
where $(M_1,f_1)\sim (M_2,f_2)$ if and only if there is an isomorphism $\varphi: M_1\to M_2$ such that $\varphi\circ f_1=f_2$. The surjection $f$ is called a \textbf{$d$-framing} of $M$. By passing to the kernel of the quotient map, we have the following equivalent description
\begin{equation}
\Quot_{d,n}(X)(k)=\set*{\cI\subeq \O_X^d: \O_X^d/\cI \in \Coh_n(X)}.
\end{equation}

Here, $\cI$ is a coherent subsheaf of $\O_X^d$. If $d=1$, then $\cI$ is the ideal sheaf of the closed subscheme of $X$ with structure sheaf $\O_X/\cI$.

The point count of $\Quot_{d,n}(X)$ over $\Fq$ is given by the usual cardinality of $\Quot_{d,n}(X)(\Fq)$.

Let $p\in X$ be a closed point. The $k$-points of the punctual Quot scheme
\begin{equation}
\Quot_{d,n}(X,p)(k):=\set*{(M,f) \bigg\vert M\in \Coh_n(X,p), f: \O_X^d\onto M}\bigg/\sim,
\end{equation}
can be described in terms of the completed local ring $R=\Ohat_{X,p}$ only. Let
\begin{equation}\label{eq:quot_ideal_def}
\Quot_{d,n}(R)(k):=\set*{I\subeq R^d: \dim_{k} R^d/I = n},
\end{equation}
then $\Quot_{d,n}(X,p)(k)\cong \Quot_{d,n}(R)(k)$ canonically. 

As in the unframed case, the point counts of the global and the punctual Quot schemes are related by a local-to-global formula. For any integer $d\geq 0$, consider the \textbf{rank-$d$ Quot zeta function}
\begin{equation}
\begin{aligned}
Q_{d,X}(t)&:=\sum_{n\geq 0} \abs{\Quot_{d,n}(X)(\Fq)} t^n;\\
Q_{d,X,p}(t)&:=\sum_{n\geq 0} \abs{\Quot_{d,n}(X,p)(\Fq)} t^n.
\end{aligned}
\end{equation}

Then we have the following Euler product formula.

\begin{proposition}\label{prop:euler_quot}
Let $X$ be a variety over $k=\Fq$, and $d\geq 0$ be an integer. Then we have
\begin{equation}
Q_{d,X}(t)=\prod_{p\in X_\cl} Q_{d,X,p}(t).
\end{equation}
\end{proposition}

\begin{proof}
As in the proof of \cite[Proposition 4.2]{huang2023mutually}, every finite-length coherent sheaf $M$ on $X$ has a unique decomposition $M=\bigoplus_{p\in X_\cl} M_p$ into finite-length coherent sheaves $M_p$ supported at $p$, and all but finitely many $M_p$ are zero. Note that $M_p$ can be identified as a finite-length $\Ohat_{X,p}$-module, and let $n_p$ be the $k$-dimension of $M_p$. 

Suppose we are given $(M,f)\in \Quot_{d,n}(X)(k)$, where $M\in \Coh_n(X)$ and $f:\O_X^d\onto M$. Localized at $p$, we get an element $(M_p,f_p)\in \Quot_{d,n_p}(X,p)(k)$, where $M_p, n_p$ are as above and $f_p:\O_{X,p}^d\onto M_p$ is a surjection of $\O_{X,p}$-modules. 

Conversely, given $(M_p,f_p)\in \Quot_{d,n_p}(X,p)(k)$ for all $p\in X_\cl$ such that all but finitely many $M_p$ are zero, construct $M=\bigoplus_{p\in X_\cl} M_p$. Consider the composition $\O_X^d\to \O_{X,p}^d\map[f_p] M_p$ and the induced map $f:\O_X^d\to \prod_{p\in X_\cl} M_p = M$ (note that a finite direct sum is the direct product). Then we recover $(M,f)\in \Quot_{d,n}(X)(k)$, where $n=\sum n_p$. 

It is routine to check that the above two constructions are inverse to each other. Therefore, we have a bijection
\begin{equation}
\Quot_{d,n}(X)(k)\cong \bigsqcup_{\sum n_p=n} \; \prod_{p\in X_\cl} \Quot_{d,n_p}(X,p)(k),
\end{equation}
which translates to the desired product formula in the generating series.
\end{proof}

\subsection{Forgetting the framing}
The first key step in Theorem \ref{thm:cusp_t_rationality} is to reduce the counting problem on $\Coh_n(X)$ to the counting problem on $\Quot_{d,n}(X)$ using the map $\Quot_{d,n}(X)\to \Coh_n(X)$ defined by forgetting the framing. We need to control the fiber size of the map, and to do so, we need to consider the punctual version. 

For the rest of this section, consider a $k$-variety $X$ and a closed point $p\in X$. Let $R=\Ohat_{X,p}$ and denote by $\m$ the maximal ideal of $R$. For simplicity, assume $p$ is a $k$-point of $X$, so the residue field of $R$ is $k$. 

From now on, the \textbf{dimension} of an $R$-module $M$ always refers to $\dim_k M$. Thus, $\Coh_n(R)$ parametrizes $n$-dimensional modules over $R$. We consider the \textbf{rank} of a finite-dimensional $R$-module $M$, denoted by $\rk M$, to be the minimal number of generators for $M$ as an $R$-module. By Nakayama's lemma, we have
\begin{equation}
\rk M=\dim_k M\otimes_R k=\dim_k M/\m M.
\end{equation}

From now on, we assume $k=\Fq$, and we will always refer to the set of $k$-points when referring to a moduli space. Consider
\begin{equation}
\Coh_n^r(R):=\set{M\in \Coh_n(R): \rk M = r},
\end{equation}
and
\begin{equation}
\Quot_{d,n}^r(R):=\set{(M,f)\in \Quot_{d,n}(R): \rk M=r}.
\end{equation}

We note that $\Coh_n^r(R)$ is empty unless $r\leq n$, and $\Quot_{d,n}^r(R)$ is empty unless $r\leq \mathrm{min}\set{d,n}$. We have $\Coh_n(R)=\bigsqcup_{r=0}^n \Coh_n^r(R)$ and $\Quot_{d,n}(R)=\bigsqcup_{r=0}^{\mathrm{min}\set{d,n}} \Quot_{d,n}^r(R)$. 

The rank will play a crucial role in determining the fiber size of the forgetful map $\Quot_{d,n}(R)\to \Coh_n(R)$. 

\begin{remark}\label{rmk:motivicstrata}
We briefly mention the geometric (scheme or stack theoretical) aspects of these moduli spaces. There is an upper semi-continuous function in Zariski topology on $\Quot_{d,n}(R)$ defined as $\rk(x)= \dim_k M/\mathfrak{m}M$, where $M$ is the module corresponding to the point $x$. This implies that, at least up to reduced structures, $\Quot_{d,n}^{\geq r}(R)$ is a closed subscheme of $\Quot_{d,n}(R)$ and $\Quot_{d,n}^{\geq r}(R)$ is a locally closed subscheme. Therefore we have a Zariski stratification $\Quot_{d,n}(R)=\bigsqcup_{r=0}^{\mathrm{min}\set{d,n}} \Quot_{d,n}^r(R)$.

Similarly, one can define an upper semi-continuous function $\rk$ on $\Coh_n(R)$. At least up to reduced structures, $\Coh_{n}^{\geq r}(R)$ is a closed substack of $\Coh_{n}(R)$ and $\Coh_{n}^{r}(R)$ is a locally closed substack. Therefore we have a Zariski stratification $\Coh_n(R)=\bigsqcup_{r=0}^n \Coh_n^r(R)$.

We will only use these observations in Appendix~\ref{appendix:motivic}.
\end{remark}


\begin{proposition}\label{prop:Fibercounting}
Assume the above setting and $d\geq r$. Then the forgetful map $\Phi:\Quot_{d,n}^r(R)\to \Coh_n^r(R)$ is surjective with preimage cardinality
\begin{equation}\label{eq:Fibercounting}
\abs{\Phi^{-1}(M)}=\frac{1}{\abs{\Aut M}}\abs{\Surj(\Fq^d,\Fq^r)}\, q^{d(n-r)}
\end{equation}
for any $M\in \Coh_n^r(R)$. 
\end{proposition}

\begin{proof}
Fix an $R$-module $M$ whose isomorphism class lies in $\Coh_n^r(R)$. Then $\Phi^{-1}(M)$ consists of surjections $f: R^d\onto M$ up to the equivalence $f_1 \sim f_2$ if and only if there is $\phi\in \Aut M$ such that $\phi\circ f_2=f_1$. In other words, we have
\begin{equation}
\Phi^{-1}(M)=\Surj(R^d,M)/\Aut M,
\end{equation}
where $\Surj(R^d,M)$ is the set of $R$-linear surjections from $R^d$ to $M$, and $\Aut M$ acts on $\Surj(R^d,M)$ by composition. We note that $\Aut M$ acts on $\Surj(R^d,M)$ freely: if $\phi\circ f=f$ and $f$ is a surjection, we must have $\phi=\mathrm{id}$. Therefore, we have $\abs{\Phi^{-1}(M)}=\abs{\Surj(R^d,M)}/\abs{\Aut M}$. It suffices to determine $\abs{\Surj(R^d,M)}$.

Determining an $R$-linear map from $R^d$ to $M$ is equivalent to choosing a $d$-tuple $(v_1,\dots,v_d)$ of elements of $M$. By Nakayama's lemma, the map associated to $(v_1,\dots,v_d)$ is surjective if and only if the mod $\m$ reductions of $v_1,\dots,v_d$ generate $M/\m M$ as a $k$-vector space. Hence,
\begin{equation}
\Surj(R^d,M)=\Surj(k^d,M/\m M) \times (\m M)^d.
\end{equation}

Noting that $\dim_k M/\m M=r$ and $\dim_k M=n$, we have $\dim_k \m M=n-r$, so
\begin{align}
\abs{\Surj(R^d,M)} &= \abs{\Surj(k^d,k^r)}\cdot \abs{\m M}^d \\
&= \abs{\Surj(\Fq^d,\Fq^r)}\, q^{d(n-r)}.
\end{align}
\end{proof}

\begin{corollary}\label{cor:quot_to_coh_const_fiber}
If $d\geq r$, we have
\begin{equation}
\abs{\Coh_n^r(R)(\Fq)}=\frac{\abs{\Quot_{d,n}^r(R)(\Fq)}}
{
q^{d(n-r)}\abs{\Surj(\Fq^d,\Fq^r)}
}.
\end{equation}

Note that $\abs{\Coh_n^r(R)(\Fq)}$ is a stacky point count.
\end{corollary}

Since the left-hand side does not depend on $d$, the count $\abs{\Quot_{r,n}^r(R)}$ determines $\abs{\Quot_{d,n}^r(R)}$ for all $d\geq r$. 

\begin{corollary}\label{cor:quot_non_full_rank}
If $d\geq r$, we have
\begin{equation}
\abs{\Quot_{d,n}^r(R)}=q^{(n-r)(d-r)} \abs{\Gr(r,d)(\Fq)}\cdot \abs{\Quot_{r,n}^r(R)},
\end{equation}
where $\Gr(r,d)(\Fq)$ is the Grassmannian parametrizing $r$-dimensional subspaces of $\Fq^d$.
\end{corollary}
\begin{proof}
Comparing Corollary \ref{cor:quot_to_coh_const_fiber} for $d:=d$ and $d:=r$, we have
\begin{align}
\frac{\abs{\Quot_{d,n}^r(R)}}{\abs{\Quot_{r,n}^r(R)}} 
&= \frac{q^{d(n-r)}\abs{\Surj(\Fq^d,\Fq^r)}}{q^{r(n-r)}\abs{\Surj(\Fq^r,\Fq^r)}} \\
&= q^{(n-r)(d-r)} \abs{\Gr(r,d)(\Fq)}.
\end{align}
\end{proof}

The above suggests that it suffices to study $\abs{\Quot_{d,n}^r(R)}$ for $r=d$. We introduce the notation
\begin{equation}
\Hilb_{d,n}(R):=\Quot_{d,n+d}^d(R)
\end{equation}
for $d,n\geq 0$. 

Recall from \eqref{eq:quot_ideal_def} that
\begin{equation}
\Quot_{d,n+d}^d(R)=\set{I\subeq R^d: \dim_k R^d/I=n+d, \rk_k R^d/I=d}.
\end{equation}

Note that
\begin{equation}
\rk_k \frac{R^d}{I} = \dim_k \frac{R^d}{I}\otimes_R \frac{R}{\m} = \dim_k \frac{R^d}{\m R^d+I} \leq \dim_k \frac{R^d}{\m R^d} = d,
\end{equation}
where the equality holds if and only if $I\subeq \m R^d$. Hence, we have a simplified description, to be used in \S \ref{sec:strata}:
\begin{equation}\label{eq:hilb_ideal_def}
\Hilb_{d,n}(R) = \set{I\subeq \m R^d: \dim_k \m R^d/I = n}.
\end{equation}
\begin{remark}\label{rmk:whyweneedappendixB} We warn the readers that the arguments in Proposition~\ref{prop:Fibercounting} and its corollaries do not generalize directly to the motivic setting. Even though the formula (\ref{eq:Fibercounting}) is true pointwise, we need it to hold stratum-wise in order to treat the motivic case. However, it is not clear how one can find a stratification over which the formula can generalize. We refer the reader to Appendix~\ref{appendix:motivic} for a treatment in the motivic case. 
\end{remark}

\subsection{Some generating series}
Let $k=\Fq$, $X$ be a $k$-variety, $p\in X(k)$, and $R=\Ohat_{X,p}$. Recall the (local) rank-$d$ Quot zeta function
\begin{equation}
Q_{d,R}(t)=Q_{d,X,p}(t):=\sum_{n\geq 0} \abs{\Quot_{d,n}(R)(\Fq)} t^n.
\end{equation}

We also define the \textbf{rank-$d$ framed Cohen--Lenstra series}
\begin{equation}
H_{d,R}(t)=H_{d,X,p}(t):=\sum_{n\geq 0} \abs{\Hilb_{d,n}(R)(\Fq)} t^n
\end{equation}
and the \textbf{rank-$d$ (unframed) Cohen--Lenstra series}
\begin{equation}
\Zhat_{d,R}(t)=\Zhat_{d,X,p}(t):=\sum_{n\geq 0} \abs{\Coh_n^d(R)(\Fq)} t^n.
\end{equation}

The following corollary states that the data of $Q_{d,R}(t)$ for all $d\geq 0$, $H_{d,R}(t)$ for all $d\geq 0$ and $\Zhat_{d,R}(t)$ for all $d\geq 0$ determine each other. It is the combinatorial repackaging of Corollary \ref{cor:quot_to_coh_const_fiber}. We recall the $q$-Pochhammer symbol
\begin{equation}
(a;q)_n=(1-a)(1-aq)\dots (1-aq^{n-1})
\end{equation}
and the $q$-binomial coefficient for $d\geq r$
\begin{equation}
{d \brack r}_q=\frac{(q;q)_d}{(q;q)_r (q;q)_{d-r}}=\abs{\Gr(r,d)(\Fq)}.
\end{equation}

The formulas \eqref{eq:zhat_in_hilb} and \eqref{eq:quot_in_hilb} below are the point count version of Theorem \ref{Thm:Main2}, but \eqref{eq:hilb_in_quot} will not be used in the rest of the paper.

\begin{corollary}\label{cor:series_conversion}
~
\begin{gather}
\Zhat_{d,R}(t)=\dfrac{t^d}{q^{d^2}(q^{-1};q^{-1})_d} H_{d,R}(tq^{-d});\label{eq:zhat_in_hilb}\\
Q_{d,R}(t)=\sum_{r=0}^d {d \brack r}_q t^r H_{r,R}(tq^{d-r});\label{eq:quot_in_hilb}\\
H_{d,R}(t)=t^{-d}\sum_{r=0}^d (-1)^{d-r} q^{-{d-r\choose 2}} {d \brack r}_{q^{-1}} Q_{r,R}(tq^{d-r}).\label{eq:hilb_in_quot}
\end{gather}
\end{corollary}
\begin{proof}
The formula \eqref{eq:zhat_in_hilb} follows from Corollary \ref{cor:quot_to_coh_const_fiber} with $r=d$ and
\begin{equation}
\abs{\Surj(\Fq^d,\Fq^d)}=\abs{\GL_d(\Fq)}=q^{d^2}(q^{-1};q^{-1})_d.
\end{equation}

The formula \eqref{eq:quot_in_hilb} follows from Corollary \ref{cor:quot_non_full_rank} and $\abs{\Quot_{d,n}(R)}=\sum_{r=0}^d \abs{\Quot_{d,n}^r(R)}$.

We now prove \eqref{eq:hilb_in_quot} using a $q$-Pascal inversion formula. Observe that \eqref{eq:quot_in_hilb} can be restated as
\begin{equation}\label{eq:quot_in_hilb_pascal}
q_d(t)=\sum_{r=0}^d {d \brack r}_{q^{-1}} h_r(t),
\end{equation}
where
\begin{align}
h_d(t)&:=t^d q^{-d^2} H_{d,R}(tq^{-d});\\
q_d(t)&:=Q_{d,R}(tq^{-d}).
\end{align}

Here, we have used the fact that ${d \brack r}_q=q^{r(d-r)} {d \brack r}_{q^{-1}}$.

It is well known (see for instance \cite{ernst2006}) that the $q$-Pascal matrix $P_{ij}={i \brack j}_{q^{-1}}$ has the inversion formula
\begin{equation}
[P^{-1}]_{ij}=(-1)^{i-j} q^{-{i-j \choose 2}} {i \brack j}_{q^{-1}}.
\end{equation}

It follows from \eqref{eq:quot_in_hilb_pascal} that
\begin{equation}
h_d(t)=\sum_{r=0}^d (-1)^{d-r} q^{-{d-r \choose 2}} {d \brack r}_{q^{-1}} q_r(t).
\end{equation}

Evaluating $h_d(tq^d)$, we get \eqref{eq:hilb_in_quot} as desired. 
\end{proof}

By definition, the rank-$d$ Cohen--Lenstra series gives a further decomposition of the Cohen--Lenstra series previously considered by the first author in \cite{huang2023mutually}:
\begin{equation}
\Zhat_R(t)=\sum_{d\geq 0} \Zhat_{d,R}(t).
\end{equation}

Thus, point counting on the punctual Quot scheme allows the study of $\Zhat_R(t)$ one piece at a time.

\section{Gr\"obner bases for monomial subrings}\label{sec:groebner}
In the classification problem of submodules of a free module $R^d$, a Gr\"obner basis theory is aimed at fulfilling the following goals:
\begin{itemize}
\item Provide a division algorithm;
\item Give a canonical generating set (\emph{reduced Gr\"obner basis}) for each submodule of $R^d$;
\item A criterion (\emph{Buchberger criterion}) to test whether a generating set is a reduced Gr\"obner basis. 
\end{itemize}

A Gr\"obner basis theory is classifically known if $R$ is a polynomial ring $k[T_1,\dots,T_N]$ over a field (see for instance \cite{adamsloustaunau, beckerweispfenning, mishrayap}) or a power series ring $k[[T_1,\dots,T_N]]$ over a field (developed by Hironaka \cite{hironaka1964resolution} under the terminology of ``standard bases''; see also \cite{becker1990stability}). To approach Theorem \ref{thm:cusp_t_rationality}, we shall develop a Gr\"obner basis theory for the ring $\Fq[[u,v]]/(u^2-v^3)$. The key observation is that $\Fq[[u,v]]/(u^2-v^3)$ is isomorphic to $\Fq[[T^2,T^3]]$, a subring of $\Fq[[T]]$.\footnote{Geometrically speaking, by doing this, we are using the fact that the cusp singularity is unibranch.} It turns out that the Gr\"obner basis theory for this subring can be developed in the same way as the classical theory. 

For the rest of this section, we will give the relevant statements in the general setting of monomial subrings, after setting up some basic terminology. The statements apply to both the polynomial ring and the power series ring with almost no modifications, but the rest of this paper only needs the theory for the power series ring. For this reason, we only give the proof for the power series ring; moreover, since the proof completely follows the classical treatment, we leave it to Appendix \ref{appendix:groebner}.

\subsection{Notation setup}
Let $k$ be a field and $\Omega$ be either a polynomial ring $k[T_1,\dots,T_N]$ or a power series ring $k[[T_1,\dots,T_N]]$. A \textbf{monomial subring} of $\Omega$ is a $k$-subalgebra (topologically) generated\footnote{If $\Omega$ is the power series ring, then the monomial subring topologically generated by monomials $\mu_1,\dots,\mu_r$ is $R=k[[\mu_1,\dots,\mu_r]]$. Note that $R$ is $(T_1,\dots,T_N)$-adic closed in $\Omega$.} by finitely many monomials. Let $R$ be a monomial subring. Fix $d\geq 1$, and let
\begin{equation}
F=R^d=Ru_1\oplus \dots Ru_d
\end{equation}
be a free $R$-module with distinguished basis $\set{u_1,\dots,u_d}$. Note that $R$ is Noetherian, so every submodule of $F$ is finitely generated. A \textbf{monomial} of $F$ is a monomial of $R$ multiplied by a basis vector $u_i$. 

A \textbf{monomial order} on $F$ is a total order $<$ on monomials of $F$ such that for any monomial $\tau\neq 1$ of $R$ and monomials $\mu<\nu$ of $F$, we have
\begin{itemize}
\item $\mu < \tau \mu$;
\item $\tau \mu < \tau \nu$.
\end{itemize}

In this section, we work with a fixed monomial order. If $\mu<\nu$, we say $\mu$ is lower than $\nu$ and $\nu$ is higher than $\mu$. We recall from standard literature that $<$ is always a well ordering. 

For two monomials $\mu,\nu$ of $F$, we define $\mu\prec \nu$ if $\mu<\nu$ and $\Omega$ is the power series ring, or $\mu>\nu$ and $\Omega$ is the polynomial ring. If $\mu\prec \nu$, we say $\mu$ is \textbf{before} $\nu$ and $\nu$ is \textbf{behind} $\mu$.

Any nonzero element $f$ of $F$ can be uniquely written as $f=\sum_i c_i \mu_i$, where $c_i\in k^\times$ and $\mu_i$ are distinct monomials of $F$. We say $c_i\mu_i$ to be a \textbf{term} of $f$. For terms $a\mu$ and $b\nu$ where $a,b\in k^\times$, we say $a\mu \prec b\nu$ if $\mu\prec \nu$. The \textbf{leading term} of $f\in F\setminus \set{0}$, denoted by $\LT(f)$, is the \textbf{foremost} (i.e., smallest in the $\prec$ ordering) term of $f$. Note that when $\Omega$ is the power series ring, $f$ may have infinitely many terms, but well-ordering ensures that a unique leading term exists. The \textbf{leading monomial} of $f\in F\setminus \set{0}$, denoted by $\LM(f)$, is the monomial associated to the leading term of $f$. 

By convention, we say $\LT(0)=0$, and $0\succ \mu$ for any monomial $\mu$. Therefore, we can talk about $\LT(f)\succeq \LT(g)$ even when $f$ or $g$ is zero. If $g=0$, this condition forces $f=0$. If $g\neq 0$, this condition means $f$ has no term before $\LT(g)$, which holds as well if $f=0$. 

We adopt the notation $(\cdot)$ to refer to the submodule of $F$ generated by a collection of elements. Given a submodule $M$ of $F$, the \textbf{leading submodule} of $M$ is defined as
\begin{equation}
\LT(M):=(\LT(f): f\in M),
\end{equation}
the submodule generated by leading terms of elements of $M$. A \textbf{monomial submodule} is a submodule of $F$ generated by monomials. The leading submodule of any submodule is a monomial submodule. 

\begin{lemma}[Lemma {\ref{a_lem:dimension_same_as_leading}}]
Let $M$ be a submodule of $F$. Then there is an isomorphism of $k$-vector spaces
\begin{equation}
\frac{F}{M}\cong \frac{F}{\LT(M)}.
\end{equation}

In fact, the set of monomials outside $\LT(M)$ is a $k$-basis for $F/M$.
\label{lem:dimension_same_as_leading}
\end{lemma}

Each monomial submodule $M$ has a unique minimal generating set consisting of monomials. We denote this minimal generating set by $C(M)$, and call it the set of \textbf{corners} of $M$. The set of corners of $M$ is always finite because $R$ is Noetherian. The set $C(M)$ is also the set of minimal monomials in $M$, ordered by divisibility. 

\subsection{Gr\"obner basis theory}
We will provide a division algorithm, a definition of reduced Gr\"obner bases and a Buchberger criterion in the above setting of monomial subrings. We first give an overview of what is different in the subring setting. The division algorithm and the definition of reduced Gr\"obner bases are the same as in the cases of the full polynomial ring or power series ring, except that divisibility must take place in the subring. For example, if $R=k[[T^2,T^3]]$, then $T^2$ does not divide $T^3$ in $R$. The Buchberger criterion now needs to consider a finite set of minimal common multiples rather than a (unique) least common multiple. For example, every two monomials in a polynomial ring or a power series ring have a least common multiple, but in $k[[T^2,T^3]]$, the monomials $T^2$ and $T^3$ have two mutually indivisible minimal common multiples: $T^5$ and $T^6$. 

\begin{definition}\label{def:divisible}
A monomial $\mu\in F$ \textbf{divides} a monomial $\nu\in F$ if there is a monomial $\tau\in R$ such that $\nu=\tau\mu$. Such a monomial $\tau$ must be unique, and we write
\begin{equation}
\frac{\nu}{\mu}:=\tau.
\end{equation}
\end{definition}

We now state the division algorithm, which is needed not only in the development of the theory, but also when applying the Buchberger criterion in \S \ref{sec:strata}. 

\begin{proposition}
[Division algorithm, Proposition {\ref{a_prop:division_algorithm}}]
Let $f,g_1,\dots,g_h$ be elements of $F$. Then there is a \tuparens{not necessarily unique} expression 
\begin{equation}
f=r+\sum_{i=1}^h q_i g_i
\end{equation}
with $r\in F, q_i\in R$, such that
 \begin{itemize}
  \item No term of $r$ is divisible by $\LT(g_i)$ for any $i$.
  \item $\LT(q_i g_i)\succeq \LT(f)$ for any $i$, namely, no $q_i g_i$ contains terms before \tuparens{i.e., $\prec$} $\LT(f)$. 
\label{prop:division_algorithm}
 \end{itemize}

Such an expression is called a \textbf{division expression} of $f$ by $g_1,\dots,g_h$, and $r$ is a \textbf{remainder} of $f$ divided by $g_1,\dots,g_h$.

Moreover, a possible division expression is given by the na\"ive algorithm of repetitively killing terms of $f$ using $\LT(g_i)$.
\end{proposition}

Next, we define the notion of (reduced) Gr\"obner bases and state the uniqueness theorem.

\begin{definition}\label{def:groebner}
A finite generating set $\set{g_1,\dots,g_h}$ of a submodule $M\subeq F$ is called a \textbf{Gr\"obner basis} for $M$ if
\begin{equation}
\LT(M)=(\LT(g_1),\dots,\LT(g_h)).
\end{equation}

A finite collection $G$ of elements of $F$ is called a Gr\"obner basis if $G$ is a Gr\"obner basis for the submodule generated by $G$. Equivalently, $\set{g_i}_i$ is a Gr\"obner basis if and only if
\begin{enumerate}
\item $\LT((g_i)_i) = (\LT(g_i))_i$.
\end{enumerate}

A Gr\"obner basis $G=\set{g_1,\dots,g_h}$ is \textbf{reduced} if
\begin{enumerate}
\setcounter{enumi}{1}
\item $\LT(g_1),\dots,\LT(g_h)$ are monomials and are mutually indivisible;
\item No nonleading term of $g_i$ is divisible by $\LT(g_j)$, for any $i,j$.
\end{enumerate}
\end{definition}

\begin{remark}
Note that $i,j$ are allowed to be equal in the condition (c), so $g_i$ must not have terms divisible by the leading term (other than the leading term itself). 
\end{remark}

\begin{lemma}[Lemma {\ref{a_lem:unique_remainder}}]
Let $f,g_1,\dots,g_h\in F$. If $G=\set{g_1,\dots,g_h}$ is a Gr\"obner basis, then the remainder of $f$ in a division expression by $G$ is unique \tuparens{even though the division expressions may not be unique}. Moreover, the remainder of $f$ by $G$ is zero if and only if $f\in (g_1,\dots,g_h)$. 
\label{lem:unique_remainder}
\end{lemma}

\begin{proposition}[Proposition {\ref{a_prop:reduced_unique}}]
Every submodule $M$ of $F$ has a unique reduced Gr\"obner basis.
\label{prop:reduced_unique}
\end{proposition}

As a result, the set of all submodules of $F$ is in one-to-one correspondence with the set of reduced Gr\"obner bases. Note that a finite collection $G=\set{g_1,\dots,g_h}$ is a reduced Gr\"obner basis if and only if the conditions (a)(b)(c) of Definition \ref{def:groebner} hold. 

We introduce the following nonstandard terminology for convenience in handling the conditions (b)(c).

\begin{definition}\label{def:prebasis}
A finite collection $G=\set{g_1,\dots,g_h}$ of elements of $F$ is called a \textbf{reduced pre-Gr\"obner basis} (or simply a \textbf{prebasis}) if the conditions (b)(c) of Definition of \ref{def:groebner} hold for $G$. In this case, we also say $G$ is a prebasis associated to the monomial submodule $M$, where $M:=(\LT(g_1),\dots,\LT(g_h))$.
\end{definition}

On the other hand, the Buchberger criterion below determines when the condition (a) holds for any finite collection of elements of $F$ (not necessarily a prebasis). We need an additional assumption on the monomial order, but it is satisfied in the setting of \S \ref{sec:strata}.

\begin{assumption}\label{assumption}
If $\Omega$ is a polynomial ring, no further assumption is required. 

If $\Omega$ is a power series ring, we assume that the monomial order is of order type $\omega$, namely, the well-ordered set of all monomials of $F$ under $<$ is isomorphic to the set of natural numbers under the natural ordering. Equivalently, any bounded subset of monomials is finite. 
\end{assumption}

We also need to address the issue that the least common multiple of two monomials may not exist.

\begin{definition}[LCM-set]
For two monomials $\mu_1,\mu_2$ of $F$, we define $\LCM(\mu_1,\mu_2)$ to be the set of all minimal elements (ordered by divisibility) among the monomials that are common multiples of $\mu_1,\mu_2$. The elements of $\LCM(\mu_1,\mu_2)$ are called the \textbf{minimal common multiples} of $\mu_1$ and $\mu_2$. 
\end{definition}

The LCM-set of $\mu_1,\mu_2$ is nonempty if and only if $\mu_1,\mu_2$ involve the same basis vector $u_i$ of $F$. As an equivalent definition, $\LCM(\mu_1,\mu_2)$ is the set of minimal monomial generators for the submodule $R \mu_1 \cap R\mu_2$ of $F$. Therefore, $\LCM(\mu_1,\mu_2)$ must be finite since $R$ is Noetherian. 

\begin{proposition}[Buchberger criterion]
Assume that Assumption \ref{assumption} holds. Let $G=\set{g_1,\dots,g_h}$ be a finite collection of nonzero elements of $F$. Assume without loss of generality that $g_i$ has leading coefficient $1$, namely, $\LT(g_i)=\LM(g_i)$ is a monomial, denoted by $\mu_i$. Then $G$ is a Gr\"obner basis if and only if for any $i\neq j$ and $\nu\in \LCM(\mu_i,\mu_j)$, the element
\begin{equation}
S^{\nu}_{ij}(g_i, g_j) := \frac{\nu}{\mu_i} g_i - \frac{\nu}{\mu_j} g_j
\end{equation}
leaves a remainder zero in some \tuparens{or equivalently, every} division expression by $g_1,\dots,g_h$, in the sense of Proposition \ref{prop:division_algorithm}.
\label{prop:buchberger}
\end{proposition}

In practice, to prove that $G$ is a Gr\"obner basis, it suffices to provide a division expression with remainder zero for every $i,j,\nu$; to prove that $G$ is not a Gr\"obner basis, it suffices to provide a division expression with a nonzero remainder for some $i,j,\nu$. 

\begin{remark}
Becker \cite{becker1990stability} shows that Assumption \ref{assumption} is unnecessary if $R$ is the full power series ring; his proof uses a stability argument to reduce to the case of order type $\omega$. We have not attempted to generalize Proposition \ref{prop:buchberger} since it is not needed here. 
\end{remark}

Finally, we need a technical lemma to work with Gr\"obner bases that are not necessarily reduced.

\begin{lemma}[Lemma {\ref{a_lem:almost_reduced}}]
Let $G=\set{g_1,\dots,g_h}, G'=\set{g_1',\dots,g_h'}$ be Gr\"obner bases for $M$ with $\LT(g_i)=\LT(g_i')=\mu_i$, and assume $G'$ is reduced. If $g_1-\mu_1$ has no terms divisible by any $\mu_i$, then we must have $g_1=g_1'$. 
\label{lem:almost_reduced}
\end{lemma}

\section{Counting points in Gr\"obner strata}\label{sec:strata}

\subsection{Notation setup}
Fix $d\geq 1$. Let $k=\Fq$, $R=k[[T^2,T^3]]$, $\m=(T^2,T^3)R$ and $F=R^d=Ru_1\oplus \dots \oplus Ru_d$. 

We choose to use the lexicographical order induced by $u_1\prec \dots\prec u_d\prec T$. Note that this monomial order is of order type $\omega$, so Assumption \ref{assumption} holds. For this reason, we prefer to think of it as the homogeneous lexicographical monomial order on $F$ induced by $u_1\prec \dots \prec u_d$ (the position of $T$ does not matter). This choice of monomial order is crucial in our proof.

Our monomial order satisfies that $\mu\in \m F$ and $\nu\succ \mu$ imply $\nu\in \m F$. This ensures that doing Gr\"obner theory never ``leaves'' $\m F$: if $f\in F$ satisfies $\LT(f)\in \m F$, then $f\in \m F$; if $f,g_1,\dots,g_h\in \m F$ and $r$ is a remainder of $f$ divided by $g_1,\dots,g_h$, then $r$ is also in $\m F$. We will often implicitly use this fact later. In the following discussions, the meanings of some notations and terminologies are different from those in \S \ref{sec:groebner}, in that we work within $\m F$ instead of $F$. 

Recall from \eqref{eq:hilb_ideal_def} that $\Hilb_{d,n}(R):=\Quot_{d,n+d}^d(R)$ is the set of $n$-codimensional\footnote{The (co)dimension always refers to the (co)dimension as a $k$-vector (sub)space.} $R$-submodules of $\m F$. Let $\set{M_\alpha: \alpha\in \Xi}$ denote the set of finite-codimensional monomial submodules of $\m F$; we will soon identify the index set $\Xi$ concretely. For any $\alpha\in \Xi$, let $\Hilb(\alpha)$ be the set of all $R$-submodules of $\m F$ whose leading submodule is $M_\alpha$. If we denote by $n(\alpha)$ the $k$-dimension of $\m F/M_\alpha$, then $\Hilb(\alpha)$ is a subset of $\Hilb_{d,n(\alpha)}(R)$ by Lemma \ref{lem:dimension_same_as_leading}. Thus, $\Hilb_{d,n}(R)$ is a disjoint union of $\Hilb(\alpha)$ for all $\alpha\in \Xi, n(\alpha)=n$. We call $\Hilb(\alpha)$ the \textbf{Gr\"obner stratum} indexed by the \textbf{leading term datum} $\alpha$.

Every finite-codimensional monomial submodule of $\m F$ is of the form $M_\alpha=I_1 u_1\oplus \dots \oplus I_d u_d$, where $I_i$ are nonzero\footnote{From now on, every submodule of $\m F$ and every ideal of $R$ is assumed to be finite-codimensional.} proper monomial ideals of $R$. Since There are two types of nonzero proper monomial ideals of $R$:
\begin{itemize}
\item $J(a):=(T^{a+1})$, $a\geq 1$;
\item $K(a):=(T^{a+2},T^{a+3})$, $a\geq 0$.
\end{itemize}

Note that $\dim_k \m/J(a) = \dim_k \m/K(a) = a$. Let $X=\set{J(a) \;(a\geq 1), K(a)\; (a\geq 0)}$ be the set of proper monomial ideals of $R$. Then we may associate to $\alpha$ the tuple $(I_1,\dots,I_d)$, where $I_i\in X$ and $M_\alpha=I_1 u_1\oplus \dots \oplus I_d u_d$. We shall write $\alpha=(I_1,\dots,I_d)$. This way, we identify $\Xi$ as the set $X^d$. 

From now on, we will usually use $I$ or $I_i$ to refer to a general element of $X$, while noting that $J(a)$ and $K(a)$ refer to specific elements of $X$. We will generally use lower-case Greek letters $\alpha,\beta,\dots$ to refer to elements of $\Xi$; we shall keep in mind they are leading term data that index the strata of $\Hilb_{d,*}(R)$. 

We overload the notation $n(\cdot)$ on the set $X$ by defining $n(I)=\dim_k \m/I$ for $I\in X$; therefore, $n(J(a))=n(K(a))=a$. For a monomial submodule $\alpha=(I_1,\dots,I_d)$, we have $n(\alpha)=n(I_1)+\dots+n(I_d)$. 

For a monomial ideal $I\in X$, we define $\Delta(I)$, the set of \textbf{standard monomials} of $I$, to be the set of nonconstant monomials not in $I$; it forms an $k$-basis for $\m/I$. Therefore, we have $\abs{\Delta(I)}=n(I)$ in general, and
\begin{itemize}
\item $\Delta(J(a))=\set{T^2,\dots,T^a}\cup \set{T^{a+2}}$;
\item $\Delta(K(a))=\set{T^2,\dots,T^{a+1}}$.
\end{itemize}

We also define $C(I)$, the set of \textbf{corners} of $I$, to be the minimal monomial generating set of $I$. We also label the corners of $C(I)$ by $\mu^0(I),\dots$ in the following specific way:
\begin{itemize}
\item $C(J(a))=\set{\mu^0(J(a))}=\set{T^{a+1}}$;
\item $C(K(a)):=\set{\mu^0(K(a)), \mu^1(K(a))}:=\set{T^{a+2},T^{a+3}}$ (in this order).
\end{itemize}

For $\alpha=(I_i)_{i=1}^d$, we define the set $\Delta(\alpha)$ of \textbf{standard monomials} of $\alpha$ to be the set of monomials in $\m F$ but not in $M_\alpha$. Then $\Delta(\alpha)=\bigcup_{i=1}^d \Delta(I_i)u_i$ is the set of monomials in $F$ of the form $\nu_i u_i$ where $\nu_i\in \Delta(I_i)$. Similarly, we define the set of \textbf{corners} of $\alpha$ to be the set of minimal monomial generating set of $M_\alpha$. We have $C(\alpha)=\bigcup_{i=1}^d C(I_i)u_i$. 

We now describe the elements of $\Hilb(\alpha)$, where $\alpha=(I_i)_{i=1}^d$. For $c=0$ or $1$, we define
\begin{equation}
\mu_i^c=\mu^c(I_i)u_i,
\end{equation}
recalling that $\mu^0(J(a))=T^{a+1}, \mu^0(K(a))=T^{a+2}$ and $\mu^1(K(a))=T^{a+3}$.

Of course, $\mu_i^1$ only makes sense if $I_i$ is of the form $K(a)$. For any monomial $\mu$ in $F$, we define $\Delta(\alpha)_{\succ \mu}$ to be the set of monomials in $\Delta(\alpha)$ that are higher than $\mu$. Then a reduced pre-Gr\"obner basis associated to $\alpha$ is a set of the form $G=\set{g_i^c}$, where
\begin{equation}
g_i^c \in \mu_i^c + \Span_\Fq \Delta(\alpha)_{\succ \mu_i^c}.
\end{equation}

We remark that the set of $c$ where $\mu_i^c$ or $g_i^c$ makes sense only depends on $\alpha$ and $i$; we will implicitly assume that $c$ lies in this set whenever we use the notation $\mu_i^c$ or $g_i^c$. Note also that the monomial $\mu_i^c$ is determined by $i,c$ and the leading term datum $\alpha$, and that the element $g_i^c$ is determined by $i,c$ and the reduced (pre-)Gr\"obner basis $G$. For this reason, we use the notation $\mu_i^c(\alpha)$ and $g_i^c(G)$ to extract the corresponding monomial of $\alpha$ and the corresponding basis element of $G$.

We have the following Buchberger criterion, restated explicitly in the context of our setting. Note that Assumption \ref{assumption} holds and $\LCM(T^{a+2}u_i,T^{a+3}u_i)=\set{T^{a+5}u_i,T^{a+6}u_i}$. 

\begin{lemma}\label{lem:buchberger_cusp}
Let $\alpha=(I_1,\dots,I_d)\in \Xi$. A reduced pre-Gr\"obner basis $G=\set{g_i^c}$ is a reduced Gr\"obner basis if and only if for any $i$ such that $I_i$ is of the form $K(a)$, each of the elements
\begin{equation}\label{eq:bbgforcusp}
T^3g_i^0-T^2g_i^1,\; T^4 g_i^0-T^3 g_i^1
\end{equation}
leaves a remainder zero when divided by the prebasis $G$. 
\end{lemma}

Thanks to uniqueness of the reduced Gr\"obner basis (Proposition \ref{prop:reduced_unique}), we can now identify $\Hilb(\alpha)$ with the set of reduced Gr\"obner bases associated to $\alpha$. 

\begin{example}
Let $d=3$ and $\alpha=(I_1,I_2,I_3)=(K(0),K(2),J(2))$. Then
\begin{equation}
\begin{aligned}
M_{\alpha}&=\Span\set{T^2 u_1, T^3 u_1, T^4 u_2, T^5 u_2, T^3 u_3},
\\
C(\alpha)&=\set{T^2 u_1, T^3 u_1, T^4 u_2, T^5 u_2, T^3 u_3}
\\
&=\set{\mu^0(I_1), \mu^1(I_1), \mu^0(I_2), \mu^1(I_2), \mu^0(I_3)}
\\
&=\set{\mu_1^0, \mu_1^1, \mu_2^0, \mu_2^1, \mu_3^0},
\\
\Delta(\alpha)&=\set{T^2 u_2, T^3 u_2, T^2 u_3, T^4 u_3},
\end{aligned}
\end{equation}

A reduced pre-Gr\"obner basis associated to $\alpha$ is of the form $G=\set{g_1^0, g_1^1, g_2^0, g_2^1, g_3^0}$, where
\begin{equation}
\begin{aligned}
g_1^0&\in \mu_1^0 + \Span \Delta(\alpha)_{\succ \mu_1^0} = T^2 u_1 + \Span\set{T^2 u_2, T^3 u_2, T^2 u_3, T^4 u_3},
\\
g_1^1&\in \mu_1^1 + \Span \Delta(\alpha)_{\succ \mu_1^1} = T^3 u_1 + \Span\set{T^3 u_2, T^4 u_3},
\\
g_2^0&\in \mu_2^0 + \Span \Delta(\alpha)_{\succ \mu_2^0} = T^4 u_2 + \Span\set{T^4 u_3},
\\
g_2^1&\in \mu_2^1 + \Span \Delta(\alpha)_{\succ \mu_2^1} = T^5 u_2 + \set{0},
\\
g_3^0&\in \mu_3^0 + \Span \Delta(\alpha)_{\succ \mu_3^0} = T^3 u_3 + \Span\set{T^4 u_3}.
\end{aligned}
\end{equation}

The prebasis $G$ is a Gr\"obner basis if and only if each of the four elements
\begin{equation}
T^3g_i^0-T^2g_i^1,\; T^4 g_i^0-T^3 g_i^1 \; (i=1,2)
\end{equation}
leaves a remainder zero when divided by the prebasis $G$. 
\end{example}

\begin{remark}
If $\alpha$ is of the form $\alpha=(J(a_1),\dots,J(a_d))$, then every reduced pre-Gr\"obner basis associated to $\alpha$ is Gr\"obner because there is only one monomial in $C(\alpha)$ belonging to each basis vector $u_i$, so the assumption of Lemma \ref{lem:buchberger_cusp} vacuously holds. However, if $\alpha$ contains terms of the form $K(a)$, the Buchberger criterion will lead to complications and it is the main difficulty we will address in the rest of the section.
\end{remark}

\subsection{A fiber decomposition of $\Hilb(\alpha)$}

It turns out that in order to count the cardinality of $\Hilb(\alpha)$, it suffices to count the reduced Gr\"obner bases $G$ in $\Hilb(\alpha)$ such that $g_i^0(G)$ is a monomial for all $1\leq i\leq d$. More precisely,
\begin{lemma}\label{lem:decomposition}
Let $d$ and $\alpha\in \Xi$ be given. Fix elements $h_i^0\in \mu_i^0(\alpha)+\Span \Delta(\alpha)_{\succ \mu_i^0(\alpha)}$. Then the number of reduced Gr\"obner bases $G$ such that $g_i^0(G)=h_i^0$ is independent of the choice of $h_i^0$. In particular, if we define
\begin{equation}
\Hilbnil(\alpha):=\set{G\in \Hilb(\alpha): g_i^0(G)=\mu_i^0(\alpha),\; 1\leq i\leq d},
\end{equation}
then we have
\begin{equation}
\abs{\Hilb(\alpha)}=\abs{\Hilbnil(\alpha)} \cdot q^{\sum_i \abs[\big]{\Delta(\alpha)_{\succ \mu_i^0(\alpha)}}}.
\end{equation}
\end{lemma}

From now on, we refer to $\Hilbnil(\alpha)$ as the \textbf{central Gr\"obner stratum} indexed by $\alpha$. 

\proof 
Let $\Hilb(\{h_i^0\};\alpha)$ be the set of submodules equipped with a reduced Gr\"obner basis $G$ such that $g_i^0(G)=h_i^0$. It suffices to construct a bijection $\Hilbnil(\alpha)\rightarrow\Hilb(\{h_i^0\};\alpha)$. The idea is to use a change of variable $u_i\mapsto \frac{h_i^0}{\mu_i^0/u_i}$, so that $\mu_i^0$ would be mapped to $h_i^0$. However, $\frac{h_i^0}{\mu_i^0/u_i}$ may involve terms with $T^1$, in which case it is not an element of $F=k[[T^2,T^3]]^d$. 

To make the idea precise, define an $R$-module endomorphism $\tau_{h}$ of $\mathfrak{m}F$ whose values on generators are given by  
\begin{equation}
\tau_{h}(T^b u_i)=\frac{T^b h_i^0}{\mu_i^0/u_i}, \qquad b\geq 2,
\end{equation}
where we note that the target turns out to lie in $\m F$ due to $\LT(h_i^0)= \mu_i^0$ and the definition of our monomial order on $F$. 

For a submodule $M\in \Hilbnil(\alpha)$, we shall prove that the following map
\begin{equation}
\iota_h:\Hilbnil(\alpha)\rightarrow\Hilb(\{h_i^0\};\alpha),\,\,\,M\mapsto \tau_h(M)
\end{equation}
is well defined and is the bijection we want. To prove that $\tau_h(M)$ is indeed in $\Hilb(\set{h_i^0};\alpha)$, let $G$ and $G'$ be the reduced Gr\"obner basis of $M$ and $\tau_h(M)$, respectively, and we need to show that $g_i^0(G')=h_i^0$. 

The morphism $\tau_h$ carries the criterion (\ref{eq:bbgforcusp}) satisfied by $G$ to $\tau_h(G)$, because $\tau_h$ carries the division expressions involved. (Here, we need to use the elementary observation that $\LT(\tau_h(f))=\LT(f)$ for any $f\in F$.) Therefore $\tau_h(G)$ is a Gr\"obner basis for $\tau_h(M)$. Even though $\tau_h(G)$ may not be reduced, thanks to Lemma \ref{lem:almost_reduced}, we have $g_i^0(G')=\tau_h(\mu_1^0)=h_i^0$ as desired. Hence $\iota_h$ is well defined. 

Note that $\tau_h$ is an automorphism of $\m F$, since modulo $\mathfrak{m}^2F$ it is. Indeed, if we choose an ordered basis of $\mathfrak{m}F/\mathfrak{m}^2F$ to be $\{T^2u_1,\dots,T^2u_d,T^3u_1,\dots,T^3u_d\}$, then the matrix for the reduction of $\tau_{h}$ mod $\mathfrak{m}^2F$ is lower triangular, with 1's on the diagonal. A similar argument as above shows that there is an inverse map 
\begin{equation}
\iota_h^{-1}:\Hilb(\{h_i^0\};\alpha)\rightarrow \Hilbnil(\alpha),\,\,\,M\mapsto \tau_h^{-1}(M).
\end{equation}

Therefore $\iota_h$ is a bijection. $\hfill\square$

\subsection{Explicit parametrization of $\Hilb^0(\alpha)$}
In Lemma \ref{lem:staircase} below, we will show that $\Hilb^0(\alpha)$ can be parametrized by an affine variety cut out by explicit matrix equations. The lemma will show in particular that the $J(a)$ terms in $\alpha$ do not complicate $\abs{\Hilb^0(\alpha)}$ too much. For $\alpha=(I_1,\dots,I_d)$, let $J_\alpha$ (resp., $K_\alpha$) denote the set of indices in $\set{1,\dots,d}$ whose corresponding $I_i$ is of the form $J(a)$ (resp., $K(a)$). Note that an element $\set{g_i^c}$ of $\Hilbnil(\alpha)$ is determined by $\set{g_i^1}_{i\in K_\alpha}$, and we will use $\set{g_i^1}_{i\in K_\alpha}$ thereafter to refer to an element of $\Hilbnil(\alpha)$. Of course, the indices in $J_\alpha$ still play a role as the nonleading terms of $g_i^1, i\in K_\alpha$ could involve $u_j$ for $j\in J_\alpha$. 

Let $\ev{\alpha}{K_{\alpha}}$ be the restriction of $\alpha$ to the index set $K_\alpha$. Working with $\ev{\alpha}{K_{\alpha}}$ is equivalent to working with the corresponding leading term datum in a lower $d$ (the new $d$ value is the size of $K_\alpha$), but we choose to keep the index set as $K_\alpha$ instead of $\set{1,\dots,\abs{K_\alpha}}$ for convenience. Note that $K_\alpha$ inherits the ordering from $[d]=\set{1,\dots,d}$, and this determines the monomial order when working with the index set $K_\alpha$. The same discussion applies to $\ev{\alpha}{J_\alpha}$. 

\begin{lemma}\label{lem:staircase}
There is a bijection between $\Hilbnil(\alpha)$ and four-tuples $(X,Y,Z,W)$ satisfying following conditions
\begin{align}
     \label{eq:stair1}&X,Y\in \Mat_{K_\alpha\times K_\alpha}(k),\\
     \label{eq:stair2}&Z,W\in \Mat_{K_\alpha\times J_\alpha}(k),\\
     \label{eq:stair3}&X^2=Y^3,\\
      \label{eq:stair4}&[X,Y]=0,\\
     \label{eq:stair5}& Z=XW,\\
     \label{eq:stair6}& X_{ij}=0 \text{ if } \mu^0_j\prec T^3\mu^0_i,\\
     \label{eq:stair7}& Y_{ij}=0 \text{ if } \mu^0_j\prec T^2\mu^0_i,\\
      \label{eq:stair8}& Z_{il}=0 \text{ if } \mu^0_l\prec T^3\mu^0_i,\\
     \label{eq:stair9}& W_{il}=0 \text{ if } \mu^0_l\prec \mu^0_i,
   \end{align}
under which $(X,Y,Z,W)$ corresponds to the reduced Gr\"obner basis $G$ such that
\begin{equation}\label{eq:stairG}
   g_i^1=\mu_i^1+T^{-2}\sum_{j\in K_\alpha}(X_{ij}\mu_j^0+Y_{ij}g_j^1)+T^{-2}\sum_{l\in J_\alpha}(Z_{il}+W_{il}T^3)\mu_l^0.
\end{equation}
   
 Moreover, the condition \eqref{eq:stair8} is implied by the rest. 
 \end{lemma}

In particular, 
\begin{equation}
    \abs*{\Hilbnil(\alpha)}=\abs*{\Hilbnil(\ev{\alpha}{K_\alpha})}\cdot {q^{\sum_{i\in K_\alpha}\abs*{C(\alpha|_{J_\alpha})_{\succ \mu_i^0}}}}.
\end{equation}
   
As a result, to compute the size of the (central) Gr\"obner strata in rank $d$, it suffices to understand the central Gr\"obner strata for ``pure-$K$'' leading term data in all ranks $d'\leq d$. 

\proof 
We first show that any reduced Gr\"obner basis can be written into the desired form.  Though unnecessary, we will work in the ambient ring $k[T^{-1}][[T]]$, which contains $k[[T^2,T^3]]$; the only reason is to make formulas containing $T^{-n}$ easier to understand. Given a reduced Gr\"obner basis $G$, the first formula of (\ref{eq:bbgforcusp}) implies that there are power series $X_{ij}(T)$, $Y_{ij}(T)$,$Z_{il}(T)$,$W_{il}(T)\in k[[T^2]]$ such that
\begin{equation}
g_i^1=\mu_i^1+T^{-2}\sum_{j\in K_\alpha}(X_{ij}(T)\mu_j^0+Y_{ij}(T)g_j^1)+T^{-2}\sum_{l\in J_\alpha}(Z_{il}(T)+W_{il}(T)T^3)\mu_l^0.
\end{equation}
Let $X_{ij}=X_{ij}(0)$, $Y_{ij}=Y_{ij}(0)$, $Z_{il}=Z_{il}(0)$ and $W_{il}=W_{il}(0)$. Since $G$ is reduced, $g_i^1\in \mu_i^1+\Span \Delta(\alpha)_{\succ \mu_i^1}$. It follows that $X_{ij}(T), Y_{ij}(T),Z_{il}(T),W_{il}(T)$ are all constants. Therefore $G$ can be put in the form of (\ref{eq:stairG}). Note that 
the coefficients must satisfy (\ref{eq:stair6})$\sim$(\ref{eq:stair9}) since $\mu_i^1=T\mu_i^0$ is the term with lowest degree.

We then use the second formula of (\ref{eq:bbgforcusp}) to show that 
the coefficients of (\ref{eq:stairG}) satisfy (\ref{eq:stair3})$\sim$(\ref{eq:stair5}).  Let $\widetilde{G}=G\cup \{T^3\mu_l^0|l\in J_\alpha\}=\{\mu_i^0,g_i^1|i\in K_\alpha\}\cup \{\mu_l^0,T^3\mu_l^0|l\in J_\alpha\}$. The strategy is looping (\ref{eq:stairG}) around itself to express $T^3g_i^1-T^4\mu_i^0, i\in K_\alpha$ into a $k$-linear combination of elements in $\bigcup_{m\in\mathbb{Z}} T^{2m}\widetilde{G}$. Note that elements in $\bigcup_{m\in\mathbb{Z}} T^{2m}\widetilde{G}$ are linearly independent, hence such expression is unique. Furthermore, $T^3g_i^1-T^4\mu_i^0$ being divisible by $G$ is equivalent to 
\begin{equation}\label{eq:fck0}
    T^3g_i^1-T^4\mu_i^0\in \Span_k \left\{\bigcup_{m\geq 0} T^{2m}\widetilde{G}\right\}.
\end{equation}
Observe that (\ref{eq:stairG}) induces following reduction formulas: 
\begin{align}
   \label{eq:fck1} &T\mu_i^0=g_i^1-T^{-2}\sum_{j\in K_\alpha}(X_{ij}\mu_j^0+Y_{ij}g_j^1)-T^{-2}\sum_{l\in J_\alpha}(Z_{il}\mu_l^0+W_{il}T^3\mu_l^0).\\
     \label{eq:fck2} &Tg_i^1=T^2\mu_i^0+ T^{-2}\sum_{j\in K_\alpha}(X_{ij}(T\mu_j^0)+Y_{ij}(Tg_j^1))+\sum_{l\in J_\alpha}(Z_{il}T^{-4}(T^3\mu_l^0)+W_{il}T^2\mu_l^0).
\end{align}
It is clear that, using these two formulas resursively, one is able to write $Tg_i^1$ into a $k$-linear combination of elements in $\bigcup_{m\in\mathbb{Z}} T^{2m}\widetilde{G}$. Since $T^3g_i^1-T^4\mu_i^0=T^2(Tg_i^1)-T^4\mu_i^0$,  one can also write $T^3g_i^1-T^4\mu_i^0$ as a $k$-linear combination of elements in $\bigcup_{m\in\mathbb{Z}} T^{2m}\widetilde{G}$. 

To compute the linear combination explicitly, write $X_{ij},Y_{ij},Z_{il},W_{il}$ into matrices $X,Y,Z,W$ satisfying (\ref{eq:stair1}), (\ref{eq:stair2}) and (\ref{eq:stair6})$\sim$(\ref{eq:stair9}). Let $\mathbf{J}_{0}$ \textit{resp}. $\mathbf{K}_0$ \textit{resp}. 
 $\mathbf{K}_1$ be the column vector whose entries are $\mu_l^0,l\in J_\alpha$  \textit{resp}. 
 $\mu_i^0,i\in K_\alpha$  \textit{resp}. $g_i^1,i\in K_\alpha$. For $m\in \mathbb{Z}$, let $\mathbf{J}_{m}=T^m\mathbf{J}_{0}$, $\mathbf{K}_{2m}=T^{2m}\mathbf{K}_0$ and $\mathbf{K}_{2m+1}=T^{2m}\mathbf{K}_1$. It follows from definition, formula (\ref{eq:fck1}) and formula (\ref{eq:fck2}) that following recursive formulas are true: 
   \begin{align}
\label{eq:fck3}&T\mathbf{J}_{m}=\mathbf{J}_{m+1},\,\,T^2\mathbf{K}_{m}=\mathbf{K}_{m+2},\\
  \label{eq:fck4}& T\mathbf{K}_{2m}=\mathbf{K}_{2m+1}-
  X\mathbf{K}_{2m-2}-Y\mathbf{K}_{2m-1}-Z\mathbf{J}_{2m-2}-W\mathbf{J}_{2m+1},\\
  \label{eq:fck5}& T\mathbf{K}_{2m+1}=\mathbf{K}_{2m+2}+
  XT\mathbf{K}_{2m-2}+YT\mathbf{K}_{2m-1}+ Z\mathbf{J}_{2m-1}+W\mathbf{J}_{2m+2}.
\end{align}
Now the column vector with entries $T^3g_i^1-T^4\mu_i^0$ is exactly $T\mathbf{K}_3-\mathbf{K}_4$. An iteration using (\ref{eq:fck3})$\sim$(\ref{eq:fck5}) gives the following
\begin{equation}\label{redequality}
   \begin{aligned}
     &T\mathbf{K}_3-\mathbf{K}_4=    Z\mathbf{J}_1+W\mathbf{J}_4+\sum_{m=0}^\infty Y^m(Y\mathbf{K}_{2-2m}+X\mathbf{K}_{1-2m})\\&-  
        \sum_{m=0}^\infty Y^{m}X(Y\mathbf{K}_{-1-2m}+X\mathbf{K}_{-2-2m})+\sum_{m=0}^\infty Y^{m+1}(Z\mathbf{J}_{-1-2m}+W\mathbf{J}_{2-2m})\\&- \sum_{m=0}^\infty Y^{m}X(Z\mathbf{J}_{-2-2m}+W\mathbf{J}_{1-2m}).
   \end{aligned}
\end{equation}
Note that the infinite sums are actually finite since $Y$ is nilpotent. The condition (\ref{eq:fck0}) implies that when $m\geq 0$,
\begin{equation}\label{cond2}
    \left\{\begin{aligned}
    & Y^{m}Z=Y^{m}XW,\,\,(\text{since the coefficient of }\mathbf{J}_{1-2m} \text{ is } 0);\\
    & Y^{m+3}W=Y^{m}XZ,\,\,(\text{since the coefficient of }\mathbf{J}_{-2-2m} \text{ is } 0);\\
    & Y^mXY=Y^{m+1}X,\,\,(\text{since the coefficient of }\mathbf{K}_{-1-2m} \text{ is } 0); \\
    & Y^{m}X^2=Y^{m+3}, \,\,(\text{since the coefficient of }\mathbf{K}_{-2-2m} \text{ is } 0).
\end{aligned}\right.
\end{equation}
Clearly \eqref{cond2} is equivalent to $[X,Y]=0, X^2=Y^3, Z=XW$.
Note that $Z$ is totally determined by $X$ and $W$ and there is no more condition imposed on $W$, therefore $X,Y,Z,W$ satisfy (\ref{eq:stair3})$\sim$(\ref{eq:stair5}). 

Conversely, given $(X,Y,Z,W)$ satisfying (\ref{eq:stair1})$\sim$(\ref{eq:stair9}), one just reverse the computation above to show that $G=\{g_i^1\}$ arising from (\ref{eq:stairG}) forms a reduced Gr\"obner basis. $\hfill\square$\\


See \S \ref{sec:staircase} for further discussions about the variety cut out by these matrix equations. The motivic versions of Lemmas \ref{lem:decomposition} and \ref{lem:staircase} are treated in Theorem \ref{Thm:groebner_strata_motivic}.

\section{Rationality in the $t$ variable}\label{sec:combinatorics}
Denote $H_d(t)=H_{d,R}(t)$, where $R=\Fq[[T^2,T^3]]$. We now prove the point count version of the first part of Theorem \ref{thm:cusp_t_rationality}, which says $H_d(t)$ is rational with denominator $(1-t)(1-qt)\dots (1-q^{d-1}t)$. Even though the point count on each Gr\"obner stratum is not fully understood, Lemma \ref{lem:staircase} implies that the stratum point count satisfies a stability property (Lemma \ref{lem:stability}) when we ``raise'' the associated leading term datum in a certain way, which is enough to prove the rationality statement. Moreover, we will reach Theorem \ref{thm:cusp_formula}, which provides a finite recipe to compute $H_d(t)$ for a fixed $d$ whenever the point count on certain finitely many strata are understood.

We first explore more structures in the set of all leading term data. 

\subsection{Further structures of leading term data}
Recall that $X=\set{J(a) \;(a\geq 1), K(a)\; (a\geq 0)}$ is the set of nonzero proper monomial ideals of $R=\Fq[[T^2,T^3]]$, and $\Xi=X^d$ is the set of leading term data. 

Given a leading term datum $\alpha=(I_1,\dots,I_d)$, we consider the pairs $\alpha_1:=(1, I_1), \dots, \alpha_d:=(d, I_d)$. Let $i\sbp 1,\dots,i\sbp d$ be the permutation of $[d]$ such that $\mu_{i\sbp 1}^0(\alpha) \prec \dots \prec \mu_{i\sbp d}^0(\alpha)$. We define 
\begin{equation}
\alpha \sbp j := (i\sbp j, I_{i\sbp j}),
\end{equation}
for $j\in [d]$,
or more instructively, we write
\begin{equation}
\alpha\sbp j = I_{i\sbp j}u_{i\sbp j}.
\end{equation}

We write $\alpha=(\alpha\sbp 1, \dots, \alpha\sbp d)$. In other words, we sort the \textbf{components} $I_i u_i$ of $\alpha$ in ascending order according to their $\mu^0$. (Recall that $\mu^0(J(a)u_i)=T^{a+1}u_i$ and $\mu^0(K(a)u_i)=T^{a+2}u_i$.) The sorting order can be read conveniently from
\begin{equation}\label{eq:component_order}
K(0)\sim J(1) \prec K(1) \sim J(2) \prec K(2) \sim J(3) \prec\dots
\end{equation}
and using the index $i$ in $u_i$ as the tiebreaker. For example, if $\alpha=(K(1), J(1), K(0))$, then $\alpha\sbp 1=J(1)u_2, \alpha\sbp 2=K(0)u_3, \alpha\sbp 3=K(1)u_1$. 

Note that we reserve the parenthesized subscript for indices associated to the \textbf{ranking} of a component in the sorting order above, while the normal subscript often refers to the \textbf{seat} of a component. The seat of $I_i u_i$ is $i$. We will also use the notation $i(\cdot)$ to extract the seat. For example, in the example above, $i(\alpha\sbp 2)=3$, and this says $\alpha\sbp 2=\alpha_3=I_3u_3$. The interplay between the concepts of ranking and seat will be important.

In light of the sorting order \ref{eq:component_order}, we define the \textbf{level} of a component by
\begin{equation}
l(J(a)u_i)=a-1, \ l(K(a)u_i)=a.
\end{equation}

We always have $\mu^0(\alpha_i) = T^{l(\alpha_i)+2}u_i$.

We also define the \textbf{color} of a component $\alpha_i$ by
\begin{equation}
c(J(a) u_i) = J, \ c(K(a) u_i) = K,
\end{equation}
where $\set{J,K}$ is the set of all two possible colors. Therefore, we may encode a leading term datum by the levels and colors of its components. More precisely, we have a bijection
\begin{align}
\Xi &\overset{\cong}{\longrightarrow}  \N^d \times \set{J, K}^d \\
\alpha=(I_1,\dots,I_d) & \longmapsto  (l(\alpha),c(\alpha)),
\end{align}
where $l(\alpha):=(l(\alpha_1),\dots,l(\alpha_d))$ is the \textbf{level vector} and $c(\alpha) := (c(\alpha \sbp 1),\dots,c(\alpha\sbp d))$ is the \textbf{color vector}. Here, we choose to arrange the levels in seats and the colors in ranks for future convenience. 

We now define the ``spiral raising'' operators $\gamma\sbp 1, \dots,\gamma\sbp d$ on $\Xi$. We first define the actions of these operators on $\N^d$, following the work \cite{huangjiang2022a} of the authors up to change of notation. Let $x=(l_1,\dots,l_d)\in \N^d$. Similar to the treatment of $\alpha\in \Xi$, we consider the components $x_i=(i,l_i)$. Sorting the components $x_1,\dots,x_d$ first by $l_i$, and then by $i$, we obtain an ascending list $x\sbp 1, \dots, x\sbp d$. Let $i\sbp j=i(x\sbp j)$ be the seat of the $j$-th lowest component. For $1\leq j\leq d$, we define $\gamma\sbp j(x)$ as the following, also illustrated in Figure \ref{fig:spiral}:
\begin{enumerate}
\item We fix the lowest $j-1$ components, namely, 
\begin{equation}
\parens*{\gamma\sbp j(x)}\sbp k := x\sbp k\text{ for }1\leq k \leq j-1;
\end{equation}

\item Consider the set of remaining seats $I=\set{i\sbp j,\dots,i\sbp d}$. We call these seats to be ``available''. Each other component is shifted to the next available seat to the right. Formally speaking, write $I=\set{a_1<\dots<a_{d-j+1}}$. Then
\begin{equation}
\parens*{\gamma\sbp j(x)}_{a_{k+1}} := (a_{k+1},l_{a_k})\text{ for }1\leq k\leq d-j.
\end{equation}

\item The rightmost component with seat in $I$ is shifted to the leftmost seat in $I$, but with level raised by 1. In other words,
\begin{equation}
\parens*{\gamma\sbp j(x)}_{a_1} := (a_1, l_{d-j+1}+1).
\end{equation}
\end{enumerate}

\begin{figure}[h]
\begin{subfigure}[b]{0.5\textwidth}
\begin{tikzpicture}
\node (10) at (0,0) {$\cdot$};
\node [right of=10] (20) {$\cdot$};
\node [right of=20] (30) {$\cdot$};
\node [right of=30] (40) {$\cdot$};
\node [right of=40] (50) {$\bullet$};
\node [right of=50] (60) {$\cdot$};
\node [right of=60] (70) {$\cdot$};

\node [above of=10] (11) {$\cdot$};
\node [above of=20] (21) {$\cdot$};
\node [above of=30] (31) {$\bullet$};
\node [above of=40] (41) {$\cdot$};
\node [above of=50] (51) {$\cdot$};
\node [above of=60] (61) {$\cdot$};
\node [above of=70] (71) {$\bullet$};

\node [above of=11] (12) {$\bullet$};
\node [above of=21] (22) {$\cdot$};
\node [above of=31] (32) {$\cdot$};
\node [circle,draw=black] [above of=41] (42) {$\bullet$};
\node [above of=51] (52) {$\cdot$};
\node [above of=61] (62) {$\circ$};
\node [above of=71] (72) {$\cdot$};

\node [above of=12] (13) {$\cdot$};
\node [circle,draw=black] [above of=22] (23) {$\bullet$};
\node [above of=32] (33) {$\cdot$};
\node [above of=42] (43) {$\circ$};
\node [above of=52] (53) {$\cdot$};
\node [circle,draw=black] [above of=62] (63)  {$\bullet$};
\node [above of=72] (73) {$\cdot$};

\node [above of=13] (14) {$\cdot$};
\node [above of=23] (24) {$\circ$};
\node [above of=33] (34) {$\cdot$};
\node [above of=43] (44) {$\cdot$};
\node [above of=53] (54) {$\cdot$};
\node [above of=63] (64) {$\cdot$};
\node [above of=73] (74) {$\cdot$};

\node [right of=74] (84) {};
\node [left of=13] (03) {};

\node [left of=10, node distance=3ex] (y0) {$0$};
\node [left of=11, node distance=3ex] (y1) {$1$};
\node [left of=12, node distance=3ex] (y2) {$2$};
\node [left of=13, node distance=3ex] (y3) {$3$};
\node [left of=14, node distance=3ex] (y4) {$4$};

\node [below of=10, node distance=3ex] (x1) {$1$};
\node [below of=20, node distance=3ex] (x2) {$2$};
\node [below of=30, node distance=3ex] (x3) {$3$};
\node [below of=40, node distance=3ex] (x4) {$4$};
\node [below of=50, node distance=3ex] (x5) {$5$};
\node [below of=60, node distance=3ex] (x6) {$6$};
\node [below of=70, node distance=3ex] (x7) {$7$};

\draw [-to] (23) -- (43);
\draw [-to] (42) -- (62);
\draw [-to, shorten >= 5ex] (63) -- (84);
\draw [-to, shorten <= 5ex] (03) -- (24);
\path[->] (12) edge [loop above] (12);
\path[->] (31) edge [loop above] (31);
\path[->] (50) edge [loop above] (50);
\path[->] (71) edge [loop above] (71);

\draw [-, shorten <= -2ex, shorten >= -2ex] (20) -- (24);
\draw [-, shorten <= -2ex, shorten >= -2ex] (40) -- (44);
\draw [-, shorten <= -2ex, shorten >= -2ex] (60) -- (64);
\end{tikzpicture}
\end{subfigure}
\begin{subfigure}[b]{0.3\textwidth}
$\bullet$ fixed components

$\odot$ moved from

$\circ$ moved to

$|$ available seats
\end{subfigure}
\caption{$\gamma\sbp{5}(2,3,1,2,0,3,1)=(2,4,1,3,0,2,1)$}
\label{fig:spiral}
\end{figure}
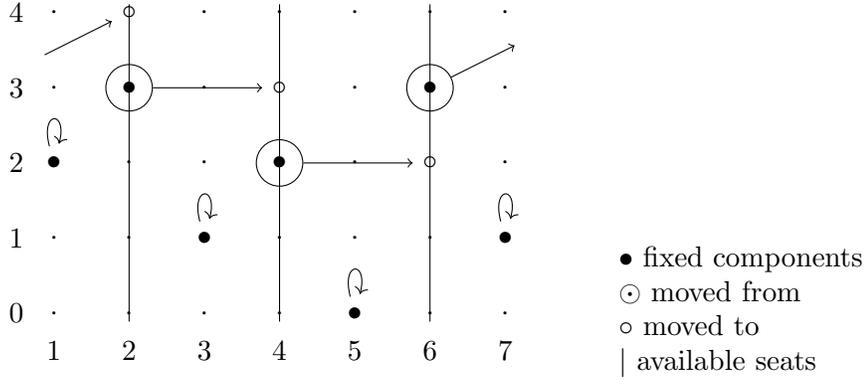

We define the action of $\gamma\sbp j$ on $\Xi$ by spiral raising the level vector and keeping the color vector unchanged. In other words, if $\alpha=(l(\alpha),c(\alpha))$, then
\begin{equation}
\gamma\sbp j(\alpha)=(\gamma\sbp j (l(\alpha)),c(\alpha)).
\end{equation}

\begin{example}
Let $\alpha=(K(1),J(1),K(0))$, and recall that $(\alpha\sbp 1,\alpha\sbp 2,\alpha\sbp 3)=(J(1)u_2, K(0)u_3, K(1)u_1)$. Then $\gamma\sbp 1$ shifts every component to the right
\begin{equation}
\gamma\sbp 1(\alpha)=(J(1)u_3, K(0+1)u_1, K(1)u_2) = (K(1),K(1),J(1)).
\end{equation}

The second operator $\gamma\sbp 2$ shifts all but the lowest component to the right: 
\begin{equation}
\gamma\sbp 2(\alpha)=(J(1)u_2, K(0+1)u_1, K(1)u_3)=(K(1),J(1),K(1)).
\end{equation}

Note that $u_1$ is skipped to $u_3$ because $u_2$ is fixed and thus unavailable. 

The last operator $\gamma\sbp 3$ ``shifts'' the highest component to the right, but since there is only one available seat, it simply rises by 1:
\begin{equation}
\gamma\sbp 3(\alpha)=(J(1)u_2, K(0)u_3, K(1+1)u_1)=(K(2),J(1),K(0)).
\end{equation}
\end{example}

The operators $\gamma\sbp j$ satisfy the following properties.

\begin{proposition}[{\cite[Theorem 1.2]{huangjiang2022a}}]
For all $1\leq j,j'\leq d$, we have
\begin{equation}
\gamma\sbp{j'}\circ \gamma\sbp j=\gamma\sbp j\circ \gamma\sbp{j'}
\end{equation}
as maps from $\N^d$ to $\N^d$.
\label{prop:spiral_commute}
\end{proposition}

Therefore, the operators $\gamma\sbp j$ induce an action of the free abelian semigroup $\Gamma:=\N^d$ on the set $\N^d$. For $a=(a_1,\dots,a_d)\in \Gamma$, the action is defined as
\begin{equation}
a\cdot x:=\gamma\sbp 1^{a_1}\circ \dots \circ \gamma\sbp d^{a_d}(x)
\end{equation}
for $x\in \N^d$. The next property states that the action is free and transitive.

\begin{proposition}[{\cite[Theorem 1.3]{huangjiang2022a}}]
Let $\mathbf{0}=(0,\dots,0)\in \N^d$. Then for any $x\in \N^d$, there exists a unique $a\in \Gamma$ such that
\begin{equation}
a\cdot \mathbf{0} = x.
\end{equation}
\label{prop:free_transitive}
\end{proposition}

The above properties imply that one may use the spiral raising operators to traverse the elements of $\Xi$; this way of enumeration will be suitable for our future analysis. We give a complete illustration of the $\Gamma$ action in $d=2$. 

\begin{example}\label{eg:grid_2}
We have the following infinite grid
\begin{equation*}
\begin{tikzcd}[sep=small]
	{(0,0)} & {(0,1)} & {(0,2)} & {(0,3)} & \cdots \\
	{(1,0)} & {(2,0)} & {(3,0)} & {(4,0)} & \cdots \\
	{(1,1)} & {(1,2)} & {(1,3)} & {(1,4)} & \cdots \\
	{(2,1)} & {(3,1)} & {(4,1)} & {(5,1)} & \cdots \\
	\vdots & \vdots & \vdots & \vdots
	\arrow[from=1-1, to=1-2]
	\arrow[from=1-1, to=2-1]
	\arrow[from=2-1, to=3-1]
	\arrow[from=3-1, to=4-1]
	\arrow[from=1-2, to=1-3]
	\arrow[from=1-2, to=2-2]
	\arrow[from=2-1, to=2-2]
	\arrow[from=3-1, to=3-2]
	\arrow[from=4-1, to=4-2]
	\arrow[from=2-2, to=3-2]
	\arrow[from=3-2, to=4-2]
	\arrow[from=1-3, to=1-4]
	\arrow[from=2-2, to=2-3]
	\arrow[from=1-3, to=2-3]
	\arrow[from=1-4, to=2-4]
	\arrow[from=2-3, to=2-4]
	\arrow[from=3-2, to=3-3]
	\arrow[from=2-3, to=3-3]
	\arrow[from=2-4, to=3-4]
	\arrow[from=3-3, to=3-4]
	\arrow[from=3-3, to=4-3]
	\arrow[from=4-2, to=4-3]
	\arrow[from=4-3, to=4-4]
	\arrow[from=3-4, to=4-4]
	\arrow[from=1-4, to=1-5]
	\arrow[from=2-4, to=2-5]
	\arrow[from=3-4, to=3-5]
	\arrow[from=4-4, to=4-5]
	\arrow[from=4-1, to=5-1]
	\arrow[from=4-2, to=5-2]
	\arrow[from=4-3, to=5-3]
	\arrow[from=4-4, to=5-4]
\end{tikzcd}
\end{equation*}
where the vertical arrow is $\gamma\sbp 1$ and the horizontal arrow is $\gamma \sbp 2$. Every tuple in $\N^2$ is found in exactly one place of this grid. 

The grid above, combined with all possible color vectors $(K, K), (K, J), (J, K)$ and $(J, J)$, gives the following grids that altogether list each element of $\Xi$ exactly once. 

\begin{equation*}
\begin{tikzcd}[sep=small]
	{(K(0),K(0))} & {(K(0),K(1))} & {(K(0),K(2))} & {(K(0),K(3))} & \cdots \\
	{(K(1),K(0))} & {(K(2),K(0))} & {(K(3),K(0))} & {(K(4),K(0))} & \cdots \\
	{(K(1),K(1))} & {(K(1),K(2))} & {(K(1),K(3))} & {(K(1),K(4))} & \cdots \\
	{(K(2),K(1))} & {(K(3),K(1))} & {(K(4),K(1))} & {(K(5),K(1))} & \cdots \\
	\vdots & \vdots & \vdots & \vdots
	\arrow[from=1-1, to=1-2]
	\arrow[from=1-1, to=2-1]
	\arrow[from=2-1, to=3-1]
	\arrow[from=3-1, to=4-1]
	\arrow[from=1-2, to=1-3]
	\arrow[from=1-2, to=2-2]
	\arrow[from=2-1, to=2-2]
	\arrow[from=3-1, to=3-2]
	\arrow[from=4-1, to=4-2]
	\arrow[from=2-2, to=3-2]
	\arrow[from=3-2, to=4-2]
	\arrow[from=1-3, to=1-4]
	\arrow[from=2-2, to=2-3]
	\arrow[from=1-3, to=2-3]
	\arrow[from=1-4, to=2-4]
	\arrow[from=2-3, to=2-4]
	\arrow[from=3-2, to=3-3]
	\arrow[from=2-3, to=3-3]
	\arrow[from=2-4, to=3-4]
	\arrow[from=3-3, to=3-4]
	\arrow[from=3-3, to=4-3]
	\arrow[from=4-2, to=4-3]
	\arrow[from=4-3, to=4-4]
	\arrow[from=3-4, to=4-4]
	\arrow[from=1-4, to=1-5]
	\arrow[from=2-4, to=2-5]
	\arrow[from=3-4, to=3-5]
	\arrow[from=4-4, to=4-5]
	\arrow[from=4-1, to=5-1]
	\arrow[from=4-2, to=5-2]
	\arrow[from=4-3, to=5-3]
	\arrow[from=4-4, to=5-4]
\end{tikzcd}
\end{equation*}

\begin{equation*}
\begin{tikzcd}[sep=small]
	{(K(0),J(1))} & {(K(0),J(2))} & {(K(0),J(3))} & {(K(0),J(4))} & \cdots \\
	{(J(2),K(0))} & {(J(3),K(0))} & {(J(4),K(0))} & {(J(5),K(0))} & \cdots \\
	{(K(1),J(2))} & {(K(1),J(3))} & {(K(1),J(4))} & {(K(1),J(5))} & \cdots \\
	{(J(3),K(1))} & {(J(4),K(1))} & {(J(5),K(1))} & {(J(6),K(1))} & \cdots \\
	\vdots & \vdots & \vdots & \vdots
	\arrow[from=1-1, to=1-2]
	\arrow[from=1-1, to=2-1]
	\arrow[from=2-1, to=3-1]
	\arrow[from=3-1, to=4-1]
	\arrow[from=1-2, to=1-3]
	\arrow[from=1-2, to=2-2]
	\arrow[from=2-1, to=2-2]
	\arrow[from=3-1, to=3-2]
	\arrow[from=4-1, to=4-2]
	\arrow[from=2-2, to=3-2]
	\arrow[from=3-2, to=4-2]
	\arrow[from=1-3, to=1-4]
	\arrow[from=2-2, to=2-3]
	\arrow[from=1-3, to=2-3]
	\arrow[from=1-4, to=2-4]
	\arrow[from=2-3, to=2-4]
	\arrow[from=3-2, to=3-3]
	\arrow[from=2-3, to=3-3]
	\arrow[from=2-4, to=3-4]
	\arrow[from=3-3, to=3-4]
	\arrow[from=3-3, to=4-3]
	\arrow[from=4-2, to=4-3]
	\arrow[from=4-3, to=4-4]
	\arrow[from=3-4, to=4-4]
	\arrow[from=1-4, to=1-5]
	\arrow[from=2-4, to=2-5]
	\arrow[from=3-4, to=3-5]
	\arrow[from=4-4, to=4-5]
	\arrow[from=4-1, to=5-1]
	\arrow[from=4-2, to=5-2]
	\arrow[from=4-3, to=5-3]
	\arrow[from=4-4, to=5-4]
\end{tikzcd}
\end{equation*}

\begin{equation*}
\begin{tikzcd}[sep=small]
	{(J(1),K(0))} & {(J(1),K(1))} & {(J(1),K(2))} & {(J(1),K(3))} & \cdots \\
	{(K(1),J(1))} & {(K(2),J(1))} & {(K(3),J(1))} & {(K(4),J(1))} & \cdots \\
	{(J(2),K(1))} & {(J(2),K(2))} & {(J(2),K(3))} & {(J(2),K(4))} & \cdots \\
	{(K(2),J(2))} & {(K(3),J(2))} & {(K(4),J(2))} & {(K(5),J(2))} & \cdots \\
	\vdots & \vdots & \vdots & \vdots
	\arrow[from=1-1, to=1-2]
	\arrow[from=1-1, to=2-1]
	\arrow[from=2-1, to=3-1]
	\arrow[from=3-1, to=4-1]
	\arrow[from=1-2, to=1-3]
	\arrow[from=1-2, to=2-2]
	\arrow[from=2-1, to=2-2]
	\arrow[from=3-1, to=3-2]
	\arrow[from=4-1, to=4-2]
	\arrow[from=2-2, to=3-2]
	\arrow[from=3-2, to=4-2]
	\arrow[from=1-3, to=1-4]
	\arrow[from=2-2, to=2-3]
	\arrow[from=1-3, to=2-3]
	\arrow[from=1-4, to=2-4]
	\arrow[from=2-3, to=2-4]
	\arrow[from=3-2, to=3-3]
	\arrow[from=2-3, to=3-3]
	\arrow[from=2-4, to=3-4]
	\arrow[from=3-3, to=3-4]
	\arrow[from=3-3, to=4-3]
	\arrow[from=4-2, to=4-3]
	\arrow[from=4-3, to=4-4]
	\arrow[from=3-4, to=4-4]
	\arrow[from=1-4, to=1-5]
	\arrow[from=2-4, to=2-5]
	\arrow[from=3-4, to=3-5]
	\arrow[from=4-4, to=4-5]
	\arrow[from=4-1, to=5-1]
	\arrow[from=4-2, to=5-2]
	\arrow[from=4-3, to=5-3]
	\arrow[from=4-4, to=5-4]
\end{tikzcd}
\end{equation*}

\begin{equation*}
\begin{tikzcd}[sep=small]
	{(J(1),J(1))} & {(J(1),J(2))} & {(J(1),J(3))} & {(J(1),J(4))} & \cdots \\
	{(J(2),J(1))} & {(J(3),J(1))} & {(J(4),J(1))} & {(J(5),J(1))} & \cdots \\
	{(J(2),J(2))} & {(J(2),J(3))} & {(J(2),J(4))} & {(J(2),J(5))} & \cdots \\
	{(J(3),J(2))} & {(J(4),J(2))} & {(J(5),J(2))} & {(J(6),J(2))} & \cdots \\
	\vdots & \vdots & \vdots & \vdots
	\arrow[from=1-1, to=1-2]
	\arrow[from=1-1, to=2-1]
	\arrow[from=2-1, to=3-1]
	\arrow[from=3-1, to=4-1]
	\arrow[from=1-2, to=1-3]
	\arrow[from=1-2, to=2-2]
	\arrow[from=2-1, to=2-2]
	\arrow[from=3-1, to=3-2]
	\arrow[from=4-1, to=4-2]
	\arrow[from=2-2, to=3-2]
	\arrow[from=3-2, to=4-2]
	\arrow[from=1-3, to=1-4]
	\arrow[from=2-2, to=2-3]
	\arrow[from=1-3, to=2-3]
	\arrow[from=1-4, to=2-4]
	\arrow[from=2-3, to=2-4]
	\arrow[from=3-2, to=3-3]
	\arrow[from=2-3, to=3-3]
	\arrow[from=2-4, to=3-4]
	\arrow[from=3-3, to=3-4]
	\arrow[from=3-3, to=4-3]
	\arrow[from=4-2, to=4-3]
	\arrow[from=4-3, to=4-4]
	\arrow[from=3-4, to=4-4]
	\arrow[from=1-4, to=1-5]
	\arrow[from=2-4, to=2-5]
	\arrow[from=3-4, to=3-5]
	\arrow[from=4-4, to=4-5]
	\arrow[from=4-1, to=5-1]
	\arrow[from=4-2, to=5-2]
	\arrow[from=4-3, to=5-3]
	\arrow[from=4-4, to=5-4]
\end{tikzcd}
\end{equation*}
\end{example}

\subsection{Spiral raising and Gr\"obner strata}
From Lemma \ref{lem:decomposition} and Lemma \ref{lem:staircase}, the size of $\Hilb(\alpha)$ is the product of three factors:
\begin{equation}\label{eq:strata_factorization}
\abs{\Hilb(\alpha)}=A(\alpha)\cdot B(\alpha)\cdot D(\alpha),
\end{equation}
where
\begin{align}
A(\alpha)&:=\abs{\Hilbnil(\ev{\alpha}{K_\alpha})},\\
B(\alpha)&:=q^{b(\alpha)}:=q^{\sum_{i\in K_\alpha}\abs*{C(\alpha|_{J_\alpha})_{\succ \mu_i^0}}},\\
D(\alpha)&:=q^{\delta(\alpha)}:=q^{\sum_i \abs[\big]{\Delta(\alpha)_{\succ \mu_i^0(\alpha)}}}.
\end{align}

The easiest factor is $B(\alpha)$, which only depends on the color vector. In particular, spiral raising does not change $B(\alpha)$.

\begin{lemma}\label{lem:B_factor}
Let the color vector of $\alpha$ be $(c\sbp 1, \dots, c \sbp d)$. Then
\begin{equation}
b(\alpha)=\#\set{(i,j):1\leq i<j\leq d, c\sbp i=K, c\sbp j=J}.
\end{equation}

In particular, $B(\gamma\sbp j(\alpha))=B(\alpha)$ for any $j\in [d]$.
\end{lemma}
\begin{proof}
This is clear because $C(\alpha\sbp j)=\set{\mu^0(\alpha\sbp j)}$ if $c(\alpha\sbp j)=J$, and the components of $\alpha$ are sorted according to their $\mu^0$. 
\end{proof}

The behavior of $A(\alpha)$ and $D(\alpha)$ under spiral raising depends on the "distances" between components of $\alpha$. We first define a notion of distance between two components $x_i=(i, l_i)$ and $x_j=(j,l_j)$ of an element $x=(l_1,\dots,l_d)$ of $\N^d$. Assume $x_i \prec x_j$, i.e., $l_i+\frac{i}{d}<l_j+\frac{j}{d}$. Then the distance between $x_i$ and $x_j$ is given by
\begin{equation}
\delta(x_i,x_j) := \floor{l_j+\frac{j}{d} - l_i - \frac{i}{d}}.
\end{equation}

We then define $\delta\sbp{ij}(x):=\delta(x\sbp i, x\sbp j)$. We call the data of $\delta\sbp{ij}(x)$ for all $1\leq i<j\leq d$ the \textbf{distance matrix} of $x$.

For $\alpha\in \Xi$, we define the distance matrix of $\alpha$ by $\delta\sbp{ij}(\alpha):=\delta\sbp{ij}(l(\alpha))$. In other words, we simply ignore the color vector.

Equivalently, 
\begin{equation}
\delta(\alpha_i,\alpha_j)=\max\set{N: T^N \mu^0(\alpha_i) \prec \mu^0(\alpha_j)}.
\end{equation}

The following lemma summarizes the effect of $\gamma\sbp j$ on the distance vector of $\alpha$. Since the color vector is immaterial, we state the lemma for any $x\in \N^d$. 

\begin{lemma}[Proof of {\cite[Lemma 5.2]{huangjiang2022a}}]
Let $x\in \N^d$ and $j\in [d]$. Define the set of \textbf{stretches} of $\gamma\sbp j$, denoted by $S\sbp j(x)$, as the following:
\begin{equation}\label{eq:def_stretch}
\set*{(b,h), 1\leq b<j\leq h\leq d \ \vrule
\begin{array}{l}
i(x\sbp h)<i(x\sbp b)<i(\gamma\sbp j(x)\sbp h)\\
\tu{or }i(\gamma\sbp j(x)\sbp h)<i(x\sbp h)<i(x\sbp b)\\
\tu{or }i(x\sbp b)<i(\gamma\sbp j(x)\sbp h)<i(x\sbp h)
\end{array}
}.
\end{equation}

Then we have
\begin{enumerate}
\item For any pair $1\leq b<h\leq d$, the change of distance
\begin{equation}
\delta\sbp{bh}(\gamma\sbp j(x)) - \delta\sbp{bh}(x)
\end{equation}
is $1$ if $(b,h)\in S\sbp j(x)$, and $0$ otherwise.

\item The map $(b,h)\mapsto b$ from $S\sbp j(x)$ to $[j-1]$ is a bijection. In particular, $S\sbp j(x)$ has $j-1$ elements. 
\end{enumerate}
\label{lem:stretch}
\end{lemma}

In other words, the spiral raising operator $\gamma\sbp j$ fixes the pairwise distances within the lowest $j-1$ (i.e., unmoved) components and the pairwise distances within the remaining (i.e., moved) components, but stretches the distances between unmoved and moved components. One may understand $\gamma\sbp j$ as raising the moved components ``as a whole'' away from the unmoved components. The crux of the lemma is the elementary observation that the distance between an unmoved component $x\sbp b$ and a moved component $x\sbp h$ is increased by 1 if $x\sbp h$ moves past the seat of $x\sbp b$ (cf. the definition of $S\sbp j(x)$ above), and unchanged otherwise. 

Given a stretch $(b,h)\in S\sbp j(x)$, we say $x\sbp b$ (resp.\ $x\sbp h$) is the lower (resp.\ higher) component of the stretch. We also call $\delta(x\sbp b,x\sbp h)$ the distance of the stretch; in particular, the distance refers to the distance \emph{before} the stretch. Similarly for $\alpha\in \Xi$.

We now return to the analysis of $A(\alpha)$ and $D(\alpha)=q^{\delta(\alpha)}$. 

We say a stretch in $S\sbp j(\alpha)$ to be \textbf{obstructed} if it has distance zero and its lower component is of color $J$. The following lemma says that the increase of $\delta(\alpha)$ caused by $\gamma\sbp j$ is given by the number of unobstructed stretches. Recall that there are $j-1$ stretches in total.

\begin{lemma}\label{lem:D_factor}
Fix $1\leq j\leq d$, and let
\begin{equation}
s=\#\set*{(b,h)\in S\sbp j(\alpha): c(\alpha\sbp b)=J, \ \delta_{bh}(\alpha)=0}
\end{equation}
be the number of obstructed stretches.

Then we have
\begin{equation}
\delta(\gamma\sbp j(\alpha))=\delta(\alpha)+j-1-s.
\end{equation}
\end{lemma}

\begin{proof}
Recall that $\delta(\alpha)=\sum_i \abs*{\Delta(\alpha)_{\succ \mu_i^0(\alpha)}}$. We separate the contributions of $\delta(\alpha)$ in terms of basis vectors. For $k\in [d]$, let $\Delta_k(\alpha)$ be the monomials in $\Delta(\alpha)$ involving $u_k$. Thus, 
\begin{equation}
\delta(\alpha)=\sum_{i,k\in [d]} \abs*{\Delta_k(\alpha)_{\succ \mu_i^0(\alpha)}}.
\end{equation}

The observations below can be verified elementarily by writing down the monomials in $\Delta_k(\alpha)$. We first look at the terms with $i=k$. We notice that $\abs*{\Delta_i(\alpha)_{\succ \mu_i^0(\alpha)}}$ is $0$ if $\alpha_i$ is of color $K$, and $1$ if $\alpha_i$ is of color $J$. 

For the terms with $i\neq k$, we consider two cases: $\alpha_i \succ \alpha_k$, and $\alpha_i\prec \alpha_k$. If $\alpha_i\succ \alpha_k$, then
\begin{equation}
\abs*{\Delta_k(\alpha)_{\succ \mu_i^0(\alpha)}}=
\begin{cases}
1,& \text{if }\delta(\alpha_k,\alpha_i) =0\text{ and }c(\alpha_k)=J;\\
0,& \text{otherwise.}
\end{cases}
\end{equation}

If $\alpha_i\prec \alpha_k$, then
\begin{equation}
\abs*{\Delta_k(\alpha)_{\succ \mu_i^0(\alpha)}}=
\begin{cases}
\delta(\alpha_i,\alpha_k)+1,& \text{if }c(\alpha_k)=J;\\
\delta(\alpha_i,\alpha_k),& \text{if }c(\alpha_k)=K.
\end{cases}
\end{equation}

In particular, for $\alpha_i\prec \alpha_k$, the two-term sum
\begin{equation}
\abs*{\Delta_k(\alpha)_{\succ \mu_i^0(\alpha)}}+\abs*{\Delta_i(\alpha)_{\succ \mu_k^0(\alpha)}}
\end{equation}
only depends on $\delta(\alpha_i,\alpha_k)$ and the colors of $\alpha_i$ and $\alpha_k$. Moreover, if $\delta(\alpha_i,\alpha_k)$ is increased by $1$, then the two-term sum is increased by 1 (in general) or unchanged (in the special case where $\delta(\alpha_i,\alpha_k)=0$ before the increase and $c(\alpha_i)=J$). Noting that the criterion for the special case is precisely the criterion for an obstructed stretch, the proof of Lemma \ref{lem:D_factor} is complete.
\end{proof}

The factor $A(\alpha)$ is invariant under $\gamma\sbp j$ if each stretch between color-$K$ components has distance at least $3$.

\begin{lemma}\label{lem:A_factor}
If any $(b,h)\in S\sbp j(\alpha)$ such that $c(\alpha\sbp b)=c(\alpha\sbp h)=K$ satisfies $\delta_{bh}(\alpha)\neq 1,2$, then
\begin{equation}
A(\gamma\sbp j(\alpha))=A(\alpha).
\end{equation}
\end{lemma}

\begin{proof}
Recall from Lemma \ref{lem:staircase} that $A(\alpha)=\abs{\Hilbnil(\ev{\alpha}{K_\alpha})}$ is the number of matrix pairs $X,Y\in \Mat_{K_\alpha\times K_\alpha}(\Fq)$ that satisfy $X^2=Y^3, XY=YX$, and
\begin{align}
X_{ik}&=0 \text{ if } \mu^0_k\prec T^3\mu^0_i,\\
Y_{ik}&=0 \text{ if } \mu^0_k\prec T^2\mu^0_i.
\end{align}

For $h\in [d]$, let $i\sbp h$ denote the seat of $\alpha\sbp h$. Change the indexing of the matrices by defining $X\sbp {bh}=X_{i\sbp b,i\sbp h}$ and similar for $Y\sbp {bh}$. The new index set is $H=\set{h: c(\alpha\sbp h)=K}$. We now understand $X,Y$ as matrices $(X\sbp{bh}), (Y\sbp{bh})$ in $\Mat_{H\times H}(\Fq)$. 

In this new indexing, $X,Y$ are strictly upper triangular, satisfy $X^2=Y^3, XY=YX$, and satisfy
\begin{align}
X\sbp {bh}&=0 \text{ if } \delta\sbp {bh}(\alpha)=0,1,2,\\
Y\sbp{bh}&=0 \text{ if } \delta\sbp{bh}(\alpha)=0,1
\end{align}
for $1\leq b<h\leq d$.

The defining equations for the pair $(X,Y)$ in the new indexing can only change when some $\delta\sbp{bh}(\alpha)$ increases from $1$ to $2$ or from $2$ to $3$. This finishes the proof.
\end{proof}

Putting Lemmas \ref{lem:B_factor}, \ref{lem:D_factor} and \ref{lem:A_factor} into the factorization \eqref{eq:strata_factorization}, we have proved the following property about $\abs{\Hilb(\alpha)}$.

\begin{lemma}\label{lem:stability}
Let $j\in [d]$ and $\alpha\in \Xi$. Assume $S\sbp j(\alpha)$ has no obstructed stretch, and the hypothesis of Lemma \ref{lem:A_factor} holds. Then
\begin{equation}
\abs{\Hilb(\gamma\sbp j(\alpha))} = q^{j-1} \abs{\Hilb(\alpha)}.
\end{equation}
\end{lemma}

In practice, we shall work with a stronger hypothesis.

\begin{definition}\label{def:stability}
We say $\alpha$ is \textbf{stable} under $\gamma\sbp j$ (or $\gamma\sbp j$-stable) if for all pairs $(b,h)$ such that $1\leq b<j\leq h\leq d$,
\begin{enumerate}
\item Whenever $c(\alpha\sbp b)=J$, we have $\delta\sbp{bh}(\alpha)\geq 1$;
\item Whenever $c(\alpha\sbp b)=c(\alpha\sbp h)=K$, we have $\delta\sbp{bh}(\alpha)\geq 3$.
\end{enumerate}
\end{definition}

The point is that if $\alpha$ is stable under $\gamma\sbp j$, then $\gamma\sbp k(\alpha)$ is also stable under $\gamma\sbp j$ for any $k\in [d]$. Therefore, we may iterate Lemma \ref{lem:stability} and prove the following result.

\begin{lemma}\label{lem:orbit}
Assume $\alpha$ is stable under $\gamma\sbp{j_1},\dots,\gamma\sbp{j_l}$ with $1\leq j_1<\dots<j_l\leq d$. Let
\begin{equation}
\gamma=\gamma\sbp{j_1}^{a_1}\circ \dots \circ \gamma\sbp{j_l}^{a_l}.
\end{equation}

Then
\begin{equation}
\abs{\Hilb(\gamma(\alpha))} = q^{a_1(j_1-1)+\dots+a_l(j_l-1)} \abs{\Hilb(\alpha)}.
\end{equation}
\end{lemma}
\begin{proof}
Repeatedly apply Lemma \ref{lem:stability}, noting that its hypothesis always holds.
\end{proof}

\subsection{The generating function}
The main idea to prove the $t$-rationality of
\begin{equation}
H_d(t)=\sum_{\alpha\in \Xi} \abs{\Hilb(\alpha)} \cdot t^{n(\alpha)}
\end{equation}
is to decompose $\Xi$ into a disjoint union of finitely many ``stable orbits'' in the sense of the lemma below, each of which contributes to a rational function in $t$.

For $\alpha\in \Xi$, we denote the \textbf{content} of $\alpha$ (or the \textbf{contribution} of $\alpha$ to $H_d(t)$) by
\begin{equation}
\Cont(\alpha):=\abs{\Hilb(\alpha)} \cdot t^{n(\alpha)}.
\end{equation}

\begin{lemma}\label{lem:orbit_content}
Assume $\beta$ is stable under $\gamma\sbp{j_1},\dots,\gamma\sbp{j_l}$ with $1\leq j_1<\dots<j_l\leq d$. Let $\Gamma'$ be the subsemigroup of $\Gamma$ generated by $\gamma\sbp{j_1},\dots,\gamma\sbp{j_l}$. Consider the orbit $\Gamma'\cdot \beta$. Then we have
\begin{equation}
\sum_{\alpha\in \Gamma'\cdot \beta} \Cont(\alpha) = \frac{1}{(1-q^{j_1-1}t)\dots (1-q^{j_l-1}t)} \Cont(\beta).
\end{equation}
\end{lemma}
\begin{proof}
The proof is the same as \cite[Theorem 5.4]{huangjiang2022a}. All we need is Proposition \ref{prop:free_transitive}, Lemma \ref{lem:stability}, and the observation that
\begin{equation}\label{eq:n_plus_1}
n(\gamma\sbp j(\alpha))=n(\alpha)+1
\end{equation}
for any $j\in [d]$. This is because there is exactly one component (the rightmost moved component) whose level is raised by $1$. 
\end{proof}

We give a sufficient condition for the stability of $\alpha$.

\begin{lemma}\label{lem:is_stable}
Any $\alpha\in \Xi$ is $\gamma\sbp 1$-stable. If $j>1$, let $l(\alpha)=a\cdot \mathbf{0}, \; a=(a_1,\dots,a_d)\in \Gamma$ be the unique expression in Proposition \ref{prop:free_transitive}. Then $\alpha$ is $\gamma\sbp j$-stable if $a_j\geq 3(d-j+1)$.
\end{lemma}
\begin{proof}
From the definition \eqref{eq:def_stretch}, there is no stretch for $\gamma\sbp 1$, so any $\alpha$ is vacuously stable. If $j>1$, we notice that for any $x=(x_1,\dots,x_d)\in \N^d$,
\begin{equation}
\gamma\sbp j^{d-j+1}(x)_i = 
\begin{cases}
x_i, & i=i(x\sbp b), b<j;\\
x_i+1, & i=i(x\sbp h), h\geq j.
\end{cases}
\end{equation}

This is because $\gamma\sbp j$ spiral-rotates the $d-j+1$ components $x\sbp j, \dots, x\sbp d$, so when one performs it $d-j+1$ times, these components return to their original seats and rise by one level. 

Therefore, for any $\beta\in \Xi$ and $b<j\leq h$, we have $\delta\sbp{bh}(\gamma\sbp j^{d-j+1}(\beta))=\delta\sbp{bh}(\beta)+1$. Hence for $\alpha$ satisfying the hypothesis, we have $\delta\sbp{bh}(\alpha)\geq 3$ for any $b<j\leq h$, so that $\alpha$ is $\gamma\sbp j$-stable. 
\end{proof}

This allows a (not necessarily minimal) decomposition of $\Xi$ into disjoint union of finitely many stable orbits, and proves the following theorem as a direct corollary.

\begin{theorem}\label{thm:cusp_formula}
For any color vector $c\in \set{J,K}^d$, let $\mathbf{0}_c$ be the element of $\Xi$ with level vector $\mathbf{0}$ and color vector $c$. Let $B$ denote the rectangle
\begin{equation}
\set{(0,b_2,\dots,b_d)\in \Gamma: 0\leq b_j\leq 3(d-j+1)\text{ for all }2\leq j\leq d}.
\end{equation}

For any $b\in B$, let $\Gamma_b$ be the subsemigroup of $\Gamma$ generated by $\gamma\sbp 1$ and all $\gamma\sbp j$ with $b_j=3(d-j)$. Then we have a stable orbit decomposition
\begin{equation}
\Xi=\bigsqcup_{c\in \set{J,K}^d} \bigsqcup_{b\in B} \Gamma_b (b\cdot \mathbf{0}_c).
\end{equation}

In particular, 
\begin{align}
H_d(t)&=\sum_{c\in \set{J,K}^d} \sum_{b\in B}  \frac{1}{\prod_{j: \gamma\sbp j\in \Gamma_b} (1-q^{j-1}t)} \Cont(b\cdot \mathbf{0}_c)\\
&=\sum_{c\in \set{J,K}^d} \sum_{b\in B}  \frac{\abs{\Hilb(b\cdot \mathbf{0}_c)}}{\prod_{j: \gamma\sbp j\in \Gamma_b} (1-q^{j-1}t)} t^{b_2+\cdots+b_d+\#_J(c)},
\end{align}
where $\#_J(c)$ is the number of $J$'s in the color vector $c$.
\end{theorem}
\begin{proof}
The only part that requires additional comment is the last line, which says
\begin{equation}
n(b\cdot \mathbf{0}_c) = b_2+\dots+b_d+\#_J(c).
\end{equation}

By \eqref{eq:n_plus_1}, it suffices to prove that $n(\mathbf{0}_c)=\#_J(c)$. Indeed, $\mathbf{0}_c$ consists of $K(0)$'s and $J(1)$'s, and since $n(K(0))=0, n(J(1))=1$, we are done. 
\end{proof}

The last assertion completes the proof of the rationality part of Theorem \ref{thm:cusp_t_rationality}, noting that all the denominators divide $(1-t)(1-qt)\dots (1-q^{d-1}t)$. Moreover, it provides a finite formula to compute $H_d(t)$ as long as we understand $\Hilbnil(\alpha)$ for any $\alpha$ purely of color $K$ of rank at most $d$. We will carry out the computation for $d\leq 3$ in \S \ref{sec:low_d}. 

\section{Rationality question in the $q$ variable}\label{sec:staircase}
Recall $H_d(t)=H_{d,R}(t)$ where $R=\Fq[[T^2,T^3]]$. We wonder whether whether $H_d(t)$ is rational in $q$ as well. The motivic version of this is precisely Conjecture \ref{conj:L_rationality}.

To approach Conjecture \ref{conj:L_rationality}, we formulate the following conjecture about a commuting matrix variety, which is of independent interest.

\begin{conjecture}\label{conj:L_rationality_strata}
Let $\alpha=(K(a_1),\dots,K(a_d))$ be a pure-$K$ leading term datum, and let $k$ be any field. Consider the affine variety $V(\alpha):=\set{(X,Y)}\subeq \Mat_d(k)^2$ defined by the matrix equation
\begin{align}
     \label{eq:L_stair3}&X^2=Y^3,\\
      \label{eq:L_stair4}&[X,Y]=0,\\
     \label{eq:L_stair6}& X_{ij}=0 \text{ if } \mu^0_j\prec T^3\mu^0_i,\\
     \label{eq:L_stair7}& Y_{ij}=0 \text{ if } \mu^0_j\prec T^2\mu^0_i.
\end{align}

Then the motive $[V(\alpha)]$ of $V(\alpha)$ in the Grothendieck ring $\KVar{k}$ of $k$-varieties is a polynomial in $\L$. 
\end{conjecture}

The set of $k$-points of $V(\alpha)$ is precisely $\Hilbnil(\alpha)$. By (the motivic versions of) \eqref{eq:strata_factorization} and Theorem \ref{thm:cusp_formula}, for any $D\geq 0$, Conjecture \ref{conj:L_rationality_strata} for all $0\leq d\leq D$ would imply Conjecture \ref{conj:L_rationality} for $d=D$. In \S \ref{sec:low_d}, we will show that both conjectures hold when $d\leq 3$.

We are able to prove the conjecture for the ``most general'' choices of $\alpha$.

\begin{theorem}\label{thm:L_rationality_stable}
If $\alpha$ is purely of color $K$ and is stable under $\gamma\sbp j$ for all $j\in [d]$ \tuparens{see Definition \ref{def:stability}}, then $[V(\alpha)]$ is a polynomial in $\L$ that does not depend on the specific choice of $\alpha$. Moreover, it can be computed using the inductive formula described in Corollary \ref{cor:staircase_formula}.
\end{theorem}

A typical choice of $\alpha$ is $\alpha=(K(0),K(3),K(6),\dots,K(3(d-1))$. We will see in a moment that Theorem \ref{thm:L_rationality_stable} is about the solution set of a single matrix equation, regardless of $\alpha$. 

We first change the indexing of the matrices as in the proof of Lemma \ref{lem:A_factor}. As a result, 
\begin{equation}\label{eq:general_staircase}
V(\alpha)\cong 
\set*{
(X,Y)\in \Mat_d(k)^2
\ \vrule \
\begin{array}{l}
\text{$X,Y$ are strictly upper triangular}\\
X^2=Y^3, XY=YX\\
X_{ij} = 0\text{ if }\delta\sbp{ij}(\alpha)<3\\
Y_{ij} = 0\text{ if }\delta\sbp{ij}(\alpha)<2
\end{array}
}.
\end{equation}

Therefore, under the hypothesis of Theorem \ref{thm:L_rationality_stable}, the variety $V(\alpha)$ is isomorphic to the following variety that only depends on $d$:
\begin{equation}\label{eq:upper_triangular}
V_d:=
\set*{
(X,Y)\in \Mat_d(k)^2
\ \vrule \
\begin{array}{l}
\text{$X,Y$ are strictly upper triangular}\\
X^2=Y^3, XY=YX
\end{array}
}.
\end{equation}

For general $\alpha$ purely of color $K$, the variety $V(\alpha)$ is the interesection of $V_d$ and several coordinate hyperplanes that prescribe certain entries of $X$ or $Y$ to be zero. It is somewhat disturbing that while $V(\alpha)$ requires fewer variables than $V_d$ in general, we are not able to deduce the $\L$-rationality of $V(\alpha)$ from the below methods that are successful for $V_d$. 

We now devote the rest of the section to prove Theorem \ref{thm:L_rationality_stable} for $V_d$. 

\subsection{Induction setup for the motive of $V_d$}
We use induction on $d$. Let $(X_{d+1},Y_{d+1})\in \Mat_{d+1}(k)^2$ given by
\begin{equation}\label{eq:induction_rule}
X_{d+1} = \begin{bmatrix}
X_d & z \\
0 & 0
\end{bmatrix}, \quad
Y_{d+1} = \begin{bmatrix}
Y_d & w \\
0 & 0
\end{bmatrix},
\end{equation}
where $X_d, Y_d\in \Mat_d(k)$ and $z,w\in \Mat_{d\times 1}(k)$. By inspecting \eqref{eq:upper_triangular}, the pair $(X_{d+1},Y_{d+1})$ is in $V_{d+1}$ if and only if
\begin{equation}\label{eq:induction}
\begin{aligned}
(X_d,Y_d)&\in V_d; \\
X_d \,z &= Y_d ^2 \,w; \\
X_d \,w &= Y_d \,z.
\end{aligned}
\end{equation}

Thus, we have a morphism $\Phi_d: V_{d+1}\to V_d$ by sending $(X_{d+1},Y_{d+1})$ to $(X_d,Y_d)$, whose fibers are solution spaces of linear equations. Define the constructible function $a: V_d\to \N$ given by $a(X_d,Y_d):=\dim \Phi_d^{-1}(X_d,Y_d)$. In order to compute the motive of $V_{d+1}$, we need to compute the motive of the stratum $a^{-1}(i)$ for each $i\in \N$. Furthermore, in order to carry out the induction, we also need to understand how $a(X_{d+1},Y_{d+1})$ depends on the choice of $(X_d,Y_d)$ and $(z,w)$. This is a challenging task in general, but it turns out that we are able to reinterpret \eqref{eq:induction} as equations on modules over a ring, so that we can exploit tools in homological algebra. 

\subsection{Some homological algebra}
Let $R=k[[x,y]]/(x^2-y^3)$\footnote{It happens that the ring $k[[x,y]]/(x^2-y^3)$ arising from the nature of \eqref{eq:upper_triangular} is isomorphic to the singular ring $k[[T^2,T^3]]$ in the setup.}, $\m=(x,y)R$, and consider the $R$-module $M_d$ whose underlying vector space is $k^d$ and the multiplications by $x$ and $y$ are given by the matrices $X_d$ and $Y_d$, respectively. The condition $(X_d,Y_d)\in V_d$ ensures that the module structure of $M_d$ is well-defined. The module $M_d$ will be where $z$ and $w$ live. Consider the matrix
\begin{equation}
A=\begin{bmatrix}
x & -y^2\\
-y & x
\end{bmatrix}\in \Mat_2(R).
\end{equation}

Then \eqref{eq:induction} can be restated as $(z,w)\in \ker A_{M_d}$, where $A_{M_d}\in \End_R(M_d\oplus M_d)$ is the tensor product of $A:R^2\to R^2$ with $M_d$. We write $a(M_d):=a(X_d,Y_d)=\dim_k \ker A_{M_d}$. We also write $b(M_d)=\dim_k \im A_{M_d}$. By the rank-nullity theorem for $A_{M_d}$, we have $a(M_d)+b(M_d)=2d$ and $a(M_d)=\dim_k \coker A_{M_d}$. 

Recall that we need to understand how $a(M_d)$ varies as $M_d=(X_d,Y_d)\in V_d$ varies. General principle suggests that we shall fix a minimal resolution of $\coker A$:
\begin{equation}\label{eq:resolution}
\cdots \map[A] R^2 \map[A'] R^2  \map[A] R^2 \map \coker A \map 0,
\end{equation}
where
\begin{equation}
A'=\begin{bmatrix}
x & y^2\\
y & x
\end{bmatrix}.
\end{equation}

One can immediately verify $AA'=A'A=0$, so the above is indeed a chain complex. 

\begin{remark}
Incidentally, $\coker A$ is isomorphic to $\m$, so $\coker A_{M_d} \cong \m \otimes_R M_d$, so that $a(M_d)=\dim_k \parens*{\m\otimes_R M_d}$. However, we are not able to make use of this observation. In fact, the rest of the proof below does not even use the fact that \eqref{eq:resolution} is exact, but only that it is a chain complex. 
\end{remark}

Consider the $2$-periodic chain complex
\begin{equation}
C^\bullet: \cdots \map[A] R^2 \map[A'] R^2 \map[A] R^2 \map[A'] \cdots,
\end{equation}
where the degree convention is made such that a part of $C^\bullet$ reads $C^{-1} \map[A'] C^0 \map[A] C^1$. We explore some properties that hold \emph{universally} for $C^\bullet$, i.e., hold for the chain complex $C_M^\bullet:=C^\bullet \otimes_R M$ for any $R$-module $M$. Define the matrices
\begin{equation}
T=\begin{bmatrix}
0 & y \\
1 & 0
\end{bmatrix}, \quad
H=\begin{bmatrix}
0 & 1 \\
0 & 0
\end{bmatrix}, \quad H'=-H
\end{equation}
in $\Mat_2(R)$ and the matrix
\begin{equation}
K=\begin{bmatrix}
0 & 1 & 0 & -y \\
0 & 0 & 1 & 0
\end{bmatrix}
\end{equation}
in $\Mat_{2\times 4}(R)$.

The following matrix identities on $\Mat_2(R)$ or $\Mat_{2\times 4}(R)$ can be verified directly using the relation $x^3=y^2$ on $R$.
\begin{lemma}\label{lem:auxiliary_identities}
~
\begin{enumerate}
\item $[A,T]=[A',T]=0$.
\item $H'A+A'H=HA'+AH'=T^2$.
\item $
AK=H\begin{bmatrix}
A & O_{2\times 2}
\end{bmatrix}-T\begin{bmatrix}
T & -A'
\end{bmatrix}
$.
\item $
\begin{bmatrix}
 I_{2\times 2} & O_{2\times 2}
\end{bmatrix}
-TK = H\begin{bmatrix}
T & -A'
\end{bmatrix} + A'\begin{bmatrix}
O_{2\times 2} & H
\end{bmatrix}
$.
\item Let $g=\mathrm{diag}(1,-1)\in \Mat_2(R)$, and note that $g=g^{-1}$. Then $gAg=A'$ and $gTg=-T$.
\end{enumerate}
\end{lemma}

From (a), the matrix $T$ induces a chain map $T:C^\bullet\to C^\bullet$. In particular, for any $R$-module $M$, we have the induced chain map $T_M:C_M^\bullet\to C_M^\bullet$ and the induced maps on the cohomology of $C_M^\bullet$:
\begin{equation}
T^0_M: H^0(C_M^\bullet)\to H^0(C_M^\bullet), \quad T^1_M: H^1(C_M^\bullet)\to H^1(C_M^\bullet).
\end{equation}

Lemma \ref{lem:auxiliary_identities}(e) implies that the chain $C_M^\bullet$ together with the chain map $T_M$ ``looks the same'' in every degree:

\begin{lemma}\label{lem:iso_to_shift}
The matrix $g:C^i=R^2\to C^{i+1}=R^2$ induces a chain isomorphism $g: C^\bullet \to C^\bullet[1]$, where $(C^\bullet[1])^i=C^{i+1}$. Moreover, we have a commutative diagram of chain complexes:
\begin{equation}
\begin{tikzcd}
	{C^\bullet} & {C^\bullet[1]} \\
	{C^\bullet} & {C^\bullet[1]}
	\arrow["{T}", from=1-1, to=2-1]
	\arrow["{-T[1]}", from=1-2, to=2-2]
	\arrow["{g}", from=1-1, to=1-2]
	\arrow["{g}", from=2-1, to=2-2]
\end{tikzcd}
\end{equation}
\end{lemma}
\begin{proof}
The fact that $g$ is a chain map and that the diagram commutes is just a restatement of Lemma \ref{lem:auxiliary_identities}(e). The chain map $g$ is a chain isomorphism because $g^2:C^\bullet\to C^\bullet[2]=C^\bullet$ is the identity map, where the equal sign reflects the $2$-periodicity of the chain complex $C^\bullet$.
\end{proof}

Lemma \ref{lem:auxiliary_identities}(b)(c)(d) implies a property of the chain map $T$ that holds universally:

\begin{lemma}\label{lem:exact_on_cohomology}
For any $R$-module $M$, the sequence
\begin{equation}
\dots \map[T^i_M]H^i(C_M^\bullet)\map[T^i_M]H^i(C_M^\bullet)\map[T^i_M] \dots
\end{equation}
is exact for $i=0,1$. 
\end{lemma}
\begin{proof}
Lemma \ref{lem:auxiliary_identities}(b) implies that $T^2:C^\bullet\to C^\bullet$ is chain homotopic to zero, with the chain homotopy map given by $H: C^0\to C^1$ from even degrees and $H':C^{-1}\to C^0$ from odd degrees. Therefore, the induced map on the cohomology $(T^i_M)^2: H^i(C_M^\bullet)\to H^i(C_M^\bullet)$ is zero.

To show the exactness, it suffices to work with $i=0$ thanks to Lemma \ref{lem:iso_to_shift} tensored with $M$. It remains to prove $\ker T^0_M \subeq \im T^0_M$, where we recall that $T^0_M: \ker A_M/\im A'_M\to \ker A_M/\im A'_M$. Let $\bbar u\in \ker(T^0_M)\subeq \ker A_M/\im A'_M$. Take a lift $u\in M^2$ with $Au=0$, then $\bbar u\in \ker(T^0_M)$ implies $Tu=A'v$ for some $v\in M^2$. Note that
\begin{equation}
\begin{bmatrix}
A & O_{2\times 2}
\end{bmatrix}
\begin{bmatrix}
u\\
v
\end{bmatrix}
=
\begin{bmatrix}
T & -A'
\end{bmatrix}
\begin{bmatrix}
u\\
v
\end{bmatrix}=0 \in M^2.
\end{equation}

Consider
\begin{equation}
s=K\begin{bmatrix}
u\\
v
\end{bmatrix}.
\end{equation}

Then by Lemma \ref{lem:auxiliary_identities}(c),
\begin{equation}
As=AK\begin{bmatrix}
u\\
v
\end{bmatrix}=H \begin{bmatrix}
A & O_{2\times 2}
\end{bmatrix}
\begin{bmatrix}
u\\
v
\end{bmatrix}-T
\begin{bmatrix}
T & -A'
\end{bmatrix}
\begin{bmatrix}
u\\
v
\end{bmatrix}=0\in M^2.
\end{equation}

By Lemma \ref{lem:auxiliary_identities}(d),
\begin{equation}
\begin{aligned}
u-Ts&=\parens*{\begin{bmatrix}
 I_{2\times 2} & O_{2\times 2}
\end{bmatrix}
-TK} \begin{bmatrix}
u\\
v
\end{bmatrix}\\
&=
H\begin{bmatrix}
T & -A'
\end{bmatrix}\begin{bmatrix}
u\\
v
\end{bmatrix}
+A'\begin{bmatrix}
O_{2\times 2} & H
\end{bmatrix}\begin{bmatrix}
u\\
v
\end{bmatrix}\\
&=
A'\begin{bmatrix}
O_{2\times 2} & H
\end{bmatrix}\begin{bmatrix}
u\\
v
\end{bmatrix}\in \im(A'_M).
\end{aligned}
\end{equation}

Hence $s\in \ker(A_M)$ and $T^0_M(\bbar s)=\bbar u$, so $\bbar u\in \im(T^0_M)$.
\end{proof}

In particular, for any $R$-module $M$, there is a distinguished subspace $\bbar {W_M}\subeq H^0(C_M^\bullet)$ with $\bbar W_M=\ker(T^0_M)=\im(T^0_M)$. It is instructive to think about the filtrations associated to the nilpotent operator $T_M^0$ on $H^0(C^\bullet_M)$:
\begin{equation}\label{eq:nilponent_filtration}
0\subeq \bbar W_M \subeq H^0(C_M^\bullet), \qquad H^0(C_M^\bullet)\map[T^0_M] \bbar W_M \map[T^0_M] 0.
\end{equation}

When $M$ is finite-dimensional over $k$, the rank-nullity theorem implies
\begin{equation}\label{eq:middle_dimensionality}
\dim_k H^0(C_M^\bullet) = 2\dim_k \bbar W_M.
\end{equation}

Taking the preimage of \eqref{eq:nilponent_filtration} with respect to the quotient map $\ker A_M \onto \ker A_M/\im A'_M$, we define
\begin{equation}
W^0(M):=\im A'_M \subeq W^1(M):=T_M^{-1}(\im A'_M) \subeq W^2(M):=\ker A_M.
\end{equation}

Then \eqref{eq:middle_dimensionality} implies $\dim_k W^1(M) = (\dim_k W^0(M)+\dim_k W^2(M))/2$. This middle-dimensionality is the only reason why we need Lemma \ref{lem:exact_on_cohomology}.

We summarize the relations among relevant dimensions associated to $C_M^\bullet$ and $T$ for a finite-dimensional $R$-module $M$. Recall $a(M):=\dim_k \ker A_M$ and $b(M):=\dim_k \im A_M$, and we have $a(M)+b(M)=2\dim_k M$ by the rank-nullity theorem.

\begin{lemma}\label{lem:numerical}
For any $R$-module $M$ with $\dim_k M<\infty$, we have
\begin{enumerate}
\item $\dim_k \ker A'_M = \dim_k \ker A_M = \dim_k W^0(M) = a(M)$;
\item $\dim_k \im A'_M = \dim_k \im A_M = \dim_k W^2(M) = b(M)$.
\item $\dim_k H^0(C_M^\bullet) = \dim_k H^1(C_M^\bullet) = a(M) - b(M)$. 
\item $\dim_k W^1(M) = \dfrac{a(M)+b(M)}{2}$.
\end{enumerate}
\end{lemma}
\begin{proof}
Identities (a)(b)(c) are immediate from tensoring Lemma \ref{lem:iso_to_shift} with $M$. The identity (d) is restating the middle-dimensionality observation above. Note that $\dim_k H^0(C_M^\bullet)$ is \emph{a priori} $\dim_k \ker A_M - \dim_k \im A'_M$, so it is necessary to observe (b) to express $\dim_k H^0(C_M^\bullet)$ in terms of the rank and the nullity of $A_M$. 
\end{proof}

\subsection{A stratification on $V_d$}
In light of Lemma \ref{lem:numerical}, we have a well-defined stratification on $V_d$ by
\begin{equation}\label{eq:stratification_staircase}
V_d=\bigsqcup_{\substack{a+b=2d\\a\geq b\geq 0}} V\sbp{a,b}, \quad V\sbp{a,b}:=\set{M_d\in V_d: a(M_d)=a, b(M_d)=b}.
\end{equation}

Recall that a point $(X_{d+1},Y_{d+1})$ of $V_{d+1}$ is determined by a point $(X_d,Y_d)\in V_d$ and a vector $u:=\begin{bmatrix}
z\\
w
\end{bmatrix}\in \ker A_{M_d}$, according to the rule \eqref{eq:induction_rule}. We then determine the stratum that $(X_{d+1},Y_{d+1})$ belongs to in terms of $M_d$ and $u$.

\begin{lemma}\label{lem:staircase_strata_induction}
In the notation above, suppose $M_d\in V\sbp{a,b}$, then
\begin{equation}
(X_{d+1},Y_{d+1}) \in \begin{cases}
V\sbp{a+2,b}, & \text{if }u\in W^0(M_d);\\
V\sbp{a+1,b+1}, & \text{if }u\in W^1(M_d)\setminus W^0(M_d);\\
V\sbp{a,b+2}, & \text{if }u\in W^2(M_d)\setminus W^1(M_d).
\end{cases}
\end{equation}
\end{lemma}

Noting that the dimensions of $W^0(M_d),W^1(M_d), W^2(M_d)$ are $b, (a+b)/2, a$ respectively, Lemma \ref{lem:staircase_strata_induction} immediately implies the following inductive formula for $V\sbp{a,b}$, which concludes the proof of Theorem \ref{thm:L_rationality_stable}.

\begin{corollary}\label{cor:staircase_formula}
The motive of $V\sbp{a,b}$ can be computed inductively by
\begin{equation}
\begin{gathered}
[V\sbp{0,0}]=1, [V\sbp{a,b}]=0\text{ unless }a\geq b\geq 0, a\equiv b\pmod 2;\\
[V\sbp{a,b}]=\L^b [V\sbp{a-2,b}]+ (\L^{\frac{a+b-2}{2}}-\L^{b-1}) [V\sbp{a-1,b-1}] + (\L^a - \L^{\frac{a+b-2}{2}}) [V\sbp{a,b-2}].
\end{gathered}
\end{equation}

The motive of $V_d$ for $d\geq 0$ is given by
\begin{equation}
[V_d]=\sum_{b=0}^d [V_{2d-b,b}].
\end{equation}
\end{corollary}

\begin{proof}[Proof of Lemma \ref{lem:staircase_strata_induction}]
It suffices to determine $b(M_{d+1})$, and moreover, we shall use the definition $b(M_{d+1})=\dim_k \im A'_{M_{d+1}}$. The matrix for $A'_{M_{d+1}}$ on $M_{d+1}\oplus M_{d+1}=M_d\oplus k \oplus M_d \oplus k$ is
\begin{equation}
A'_{M_{d+1}} = \begin{bmatrix}
X_{d+1} & Y_{d+1}^2 \\
Y_{d+1} & X_{d+1}
\end{bmatrix}
=
\begin{bmatrix}
X_d & z & Y_d^2 & Y_d w\\
0 & 0 & 0 & 0\\
Y_d & w & X_d & z\\
0 & 0 & 0 & 0
\end{bmatrix}.
\end{equation}

Therefore, the image of $A'_{M_{d+1}}$ actually lies in $M_d\oplus M_d$, and is equal to
\begin{equation}
\begin{aligned}
\im A'_{M_{d+1}} &= \im \begin{bmatrix}
X_d & Y_d^2 \\
Y_d & X_d
\end{bmatrix}
+
\Span_k \set*{\begin{bmatrix}
z\\w
\end{bmatrix},\begin{bmatrix}
Y_d w\\z
\end{bmatrix}} 
\\
&= \im A'_{M_d} + \Span_k \set{u, T_{M_d} u}\in M_d\subeq M_d.
\end{aligned}
\end{equation}

We now work mod $\im A'_{M_d}$. Since $u\in \ker A_{M_d}$, we are working in $\ker A_{M_d}/\im A'_{M_d} = H^0(C_{M_d}^\bullet)$. Hence, the increase in the $b$-parameter is given by
\begin{equation}
b(M_{d+1})-b(M_d)=\dim_k \Span_k \set{\bbar u, T^0_{M_d} \bbar u}\subeq H^0(C_{M_d}^\bullet).
\end{equation}

We recall that $T^0_{M_d}$ is a square-zero linear map on $H^0(C_{M_d}^\bullet)$. By a standard fact about nilpotent linear maps (Lemma \ref{lem:cyclic_subspace}), the dimension of $\dim_k \Span_k \set{\bbar u, T^0_{M_d} \bbar u}$ is the number of nonzero elements in $\set{\bbar u, T^0_{M_d} \bbar u}$. This completes the proof.
\end{proof}

\begin{lemma}\label{lem:cyclic_subspace}
Let $V$ be a finite-dimensional vector space over a field $k$, and $T:V\to V$ be a nilpotent linear map. For any vector $v\in V$, consider the sequence $v,Tv,T^2 v,\dots$, which is eventually zero. Then the nonzero elements in the sequence are linearly independent.
\end{lemma}
\begin{proof}
Assume $T^{n-1}v\neq 0$ and $T^nv=0$. Let $W=\Span_k \set{v,Tv,\dots, T^{n-1}v}$, so $T$ is a nilpotent endomorphism of $W$. We note that $T^{n-1}$ is not the zero map on $W$ because $T^{n-1}v\neq 0$. Hence $\dim W > n-1$, so that $v,Tv,\dots, T^{n-1}v$ must be linearly independent.
\end{proof}

\begin{example}
Using Corollary \ref{cor:staircase_formula}, we compute the motive of $V_d$ for $d\leq 8$ in Table \ref{table:staircase}. The cases with $d\leq 3$ can be directly verified from the definition \eqref{eq:upper_triangular}; see \S \ref{subsec:pure_k}.

\begin{center}
\begin{table}[h]
\begin{tabular}{|r|l|}
\hline
$d$ & $[V_d]$ \\
\hline
0 & 1\\
1 & 1\\
2 & $\L^2$\\
3 & $3\L^4-2\L^3$\\
4 & $2 \L^8+3 \L^7-5 \L^6+\L^5$\\
5 & $10 \L^{12}-5 \L^{11}-9 \L^{10}+5 \L^9$\\
6 & $5 \L^{18}+21 \L^{17}-30 \L^{16}-9 \L^{15}+15 \L^{14}-\L^{12} $\\
7 & $35 \L^{24}+7 \L^{23}-84 \L^{22}+15 \L^{21}+35 \L^{20}-7 \L^{18}$\\
8 & $14 \L^{32}+112 \L^{31}-112 \L^{30}-162 \L^{29}+113 \L^{28}+70 \L^{27}-7 \L^{26}-28 \L^{25}+\L^{22}$\\
\hline
\end{tabular}

\caption{The motive of $[V_d]$ for $d\leq 8$}
\label{table:staircase}
\end{table}
\end{center}

We observe from the table that as a polynomial in $\L$, the total coefficient of $[V_d]$ is $1$. Indeed, if we substitute $\L\mapsto 1$ in Corollary \ref{cor:staircase_formula}, we have $\ev{[V_{2d,0}]}{\L\mapsto 1}=1$ and $\ev{[V_{a,b}]}{\L\mapsto 1}=0$ if $(a,b)\neq (2d,0)$, so $\ev{[V_d]}{\L\mapsto 1}=1$. Geometrically speaking, this means that when $k=\C$, the Euler characteristic of $V_d$ in the analytic topology is $1$, which is expected because the obvious $\C^\times$-action $t\cdot (X,Y)=(tX,tY)$ on $V_d$ has only one fixed point $(X,Y)=(0,0)$. 
\end{example}

\section{Explicit computations in $d\leq 3$}\label{sec:low_d}

We assume again that $R=k[[T^2,T^3]]$ and $H_d(t):=H_{d,R}(t)$. We compute $H_d(t)$ explicitly for $d\leq 3$ and prove the second part of Theorem \ref{thm:cusp_t_rationality} using Theorem \ref{thm:cusp_formula}. As before, we present our proof in terms of point counting over $k=\Fq$. The same proof implies the motivic formulas in Theorem \ref{thm:cusp_t_rationality} since the motivic version of Theorem \ref{thm:cusp_formula} holds according to Appendix \ref{appendix:motivic}. 

\subsection{Pure-$K$ strata}\label{subsec:pure_k}
When $d\leq 3$ and $\alpha$ is a pure-$K$ leading term datum, it is not hard to describe the set $\Hilbnil(\alpha)$ (in fact the reduced structure of the variety $V(\alpha)$ in \eqref{eq:general_staircase}) explicitly. We give one typical example.

\begin{example}
Let $\alpha=(K(0),K(2),K(9))$. Then $V(\alpha)$ consists of $(X,Y)$ such that
\begin{equation}
X=\begin{bmatrix}
0 & X_{12} & X_{13} \\
0 & 0 & X_{23} \\
0 & 0 & 0
\end{bmatrix},\quad
Y=\begin{bmatrix}
0 & Y_{12} & Y_{13} \\
0 & 0 & Y_{23} \\
0 & 0 & 0
\end{bmatrix},\quad X_{12}=0
\end{equation}
and the matrix equations $X^2=Y^3, XY=YX$ are equivalent to
\begin{equation}
X_{12}X_{23}=0, \quad X_{12}Y_{23}=Y_{12}X_{23}.
\end{equation}

Since $X_{12}=0$, the only requirement is $Y_{12}X_{23}=0$. Thus $[V(\alpha)]=(2\L-1)\L^3$, where $\L^3$ results from the three freely chosen variables $X_{13},Y_{13},Y_{23}$. 
\end{example}

According to \eqref{eq:general_staircase}, the variety $V(\alpha)$ only depends on the distance matrix $\delta\sbp{ij}(\alpha)$ of $\alpha$. Moreover, for each $i<j$, the exact distance $\delta\sbp{ij}(\alpha)$ does not matter; all that matters is whether it is at most $1$ (labeled $1^-$), or $2$, or at least $3$ (labeled $3^+$). For example, the above example corresponds to the case $(\delta\sbp{12}(\alpha),\delta\sbp{23}(\alpha),\delta\sbp{13}(\alpha))=(2,3^+,3^+)$. We provide the motive of $V(\alpha)$ for every possible distance matrix in $d\leq 3$. If $d=0,1$, then $V(\alpha)$ is just a point, so $[V(\alpha)]=1$. If $d=2,3$, see Tables \ref{table:pure_K_2} and \ref{table:pure_K_3}. Note that the last rows of Tables \ref{table:pure_K_2} and \ref{table:pure_K_3} verify the $d=2,3$ entries of Table \ref{table:staircase}.

\begin{table}[h]
\begin{tabular}{|l|l|}
\hline
$\delta\sbp{12}(\alpha)$ & $[V(\alpha)]$\\
\hline
$1^-$ & $1$\\
2 & $\L$\\
$3^+$ & $\L^2$\\
\hline
\end{tabular}
\caption{Pure-$K$ strata in terms of distances, $d=2$}
\label{table:pure_K_2}
\end{table}

\begin{table}[h]
\begin{tabular}{|l|l|}
\hline
$(\delta\sbp{12}(\alpha),\delta\sbp{23}(\alpha),\delta\sbp{13}(\alpha))$ & $[V(\alpha)]$\\
\hline
$(1^-,1^-,1^-)$ & $1$\\
$(1^-,1^-,2)$ & $\L$\\
$(1^-,1^-,3+)$ & $\L^2$\\
$(1^-,2,2)$ & $\L^2$\\
$(1^-,2,3+)$ & $\L^3$\\
$(1^-,3+,3+)$ & $\L^4$\\
$(2,1^-,2)$ & $\L^2$\\
$(2,1^-,3^+)$ & $\L^3$\\
$(2,2,3^+)$ & $\L^4$\\
$(2,3^+,3^+)$ & $2\L^4-\L^3$\\
$(3^+,1^-,3^+)$ & $\L^4$\\
$(3^+,2,3^+)$ & $2\L^4-\L^3$\\
$(3^+,3^+,3^+)$ & $3\L^4-2\L^3$\\
\hline
\end{tabular}
\caption{Pure-$K$ strata in terms of distances, $d=3$}
\label{table:pure_K_3}
\end{table}

\subsection{Computation of $H_d(t)$ for $d\leq 3$}\label{subsec:low_d_computation}
\subsubsection*{Case $d=0$}
Clearly $H_d(t)=Q_d(t)=1$ if $d=0$. 

\

For $d=1,2,3$, we shall compute $H_d(t)$ by first decomposing the set $\Xi$ of leading term data into a disjoint union of stable orbits, and then computing the total contribution of each stable orbit using Lemma \ref{lem:orbit_content}. The existence of a decomposition is guaranteed by Theorem \ref{thm:cusp_formula}, but a more efficient decomposition is often available. We will need to compute the contents of finitely many leading term data along the way. Using \eqref{eq:strata_factorization}, this is always doable as long as we know $\Hilbnil(\alpha)$ for any pure-$K$ leading term data up to rank $3$. From now on, we will provide the quantities $\abs{\Hilb(\alpha)}, A(\alpha),B(\alpha),D(\alpha),\Cont(\alpha)$ as needed without proof.

All formulas in this section hold in the motivic sense as well if we replace $q$ by $\L$, because the previous discussions show that $[V(\alpha)]$ is a polynomial in $\L$ if $\alpha$ is pure-$K$ of rank at most $3$.

\subsubsection*{Case $d=1$} The set $\Xi$ of leading term data is a disjoint union of two stable orbits: $(K(0))$ stable under $\gamma\sbp 1$, and $(J(1))$ stable under $\gamma\sbp 1$. We have $\Cont(K(0))=1$ and $\Cont(J(1))=qt$. Hence
\begin{equation}
H_1(t)=\frac{1}{1-t}+\frac{qt}{1-t}=\frac{1+qt}{1-t}.
\end{equation}

By \eqref{eq:quot_in_hilb}, we have
\begin{equation}
Q_1(t)=1+tH_1(t)=\frac{1+qt^2}{1-t}.
\end{equation}

This matches the well-known formula for the local Hilbert zeta function for the cusp singularity; see for instance \cite[Example 19, $A_2$]{goettscheshende2014refined}.

\subsubsection*{Case $d=2$} We exhaust the elements of $\Xi$ using the four grids in Example \ref{eg:grid_2}. We say an arrow is stable (marked as solid) if the leading term datum at the source is stable under the spiral raising operator the arrow represents, and unstable (marked as dashed) otherwise; see Figure \ref{fig:kk_full_grid}. We introduce the notation $\Xi_c$ to refer to the set of all $\alpha\in\Xi$ with color vector $c(\alpha)=c$. From this grid, we can read that $\Xi\sbp{K,K}$ is a disjoint union of four stable orbits: $(K(0),K(0))$ under $\gamma\sbp 1$, $(K(0),K(1))$ under $\gamma\sbp 1$, $(K(0),K(2))$ under $\gamma\sbp 1$, and $(K(0),K(3))$ under $\gamma\sbp 1$, $\gamma\sbp 2$.

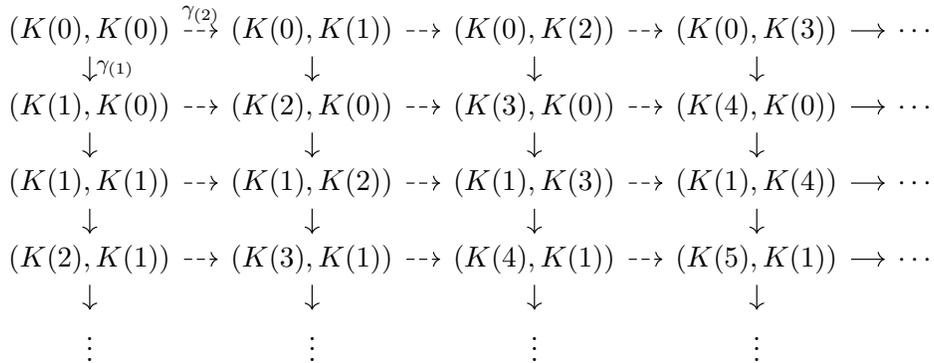
\begin{figure}[h]
\begin{equation*}
\begin{tikzcd}[sep=small]
	{(K(0),K(0))} & {(K(0),K(1))} & {(K(0),K(2))} & {(K(0),K(3))} & \cdots \\
	{(K(1),K(0))} & {(K(2),K(0))} & {(K(3),K(0))} & {(K(4),K(0))} & \cdots \\
	{(K(1),K(1))} & {(K(1),K(2))} & {(K(1),K(3))} & {(K(1),K(4))} & \cdots \\
	{(K(2),K(1))} & {(K(3),K(1))} & {(K(4),K(1))} & {(K(5),K(1))} & \cdots \\
	\vdots & \vdots & \vdots & \vdots
	\arrow["\gamma\sbp 2",dashed,from=1-1, to=1-2]
	\arrow["\gamma\sbp 1",from=1-1, to=2-1]
	\arrow[from=2-1, to=3-1]
	\arrow[from=3-1, to=4-1]
	\arrow[dashed,from=1-2, to=1-3]
	\arrow[from=1-2, to=2-2]
	\arrow[dashed,from=2-1, to=2-2]
	\arrow[dashed,from=3-1, to=3-2]
	\arrow[dashed,from=4-1, to=4-2]
	\arrow[from=2-2, to=3-2]
	\arrow[from=3-2, to=4-2]
	\arrow[dashed,from=1-3, to=1-4]
	\arrow[dashed,from=2-2, to=2-3]
	\arrow[from=1-3, to=2-3]
	\arrow[from=1-4, to=2-4]
	\arrow[dashed,from=2-3, to=2-4]
	\arrow[dashed,from=3-2, to=3-3]
	\arrow[from=2-3, to=3-3]
	\arrow[from=2-4, to=3-4]
	\arrow[dashed,from=3-3, to=3-4]
	\arrow[from=3-3, to=4-3]
	\arrow[dashed,from=4-2, to=4-3]
	\arrow[dashed,from=4-3, to=4-4]
	\arrow[from=3-4, to=4-4]
	\arrow[from=1-4, to=1-5]
	\arrow[from=2-4, to=2-5]
	\arrow[from=3-4, to=3-5]
	\arrow[from=4-4, to=4-5]
	\arrow[from=4-1, to=5-1]
	\arrow[from=4-2, to=5-2]
	\arrow[from=4-3, to=5-3]
	\arrow[from=4-4, to=5-4]
\end{tikzcd}
\end{equation*}
\caption{Stability in color vector $(K,K)$}
\label{fig:kk_full_grid}
\end{figure}

We record the information in a simpler diagram Figure \ref{fig:kk}, where we only keep the starting points of the orbits. We also omit the arrows $\gamma\sbp 1$, with the understanding that every element of $\Xi$ is stable under $\gamma\sbp 1$ (cf. Lemma \ref{lem:is_stable}). The stable arrows can thus be read from the diagram as solid arrows plus $\gamma\sbp 1$. We also label the content of each node. 

\begin{figure}[h]
\[\begin{tikzcd}[sep=small]
	{\begin{matrix}(K(0),K(0))\\1\end{matrix}} & {\begin{matrix}(K(0),K(1))\\qt\end{matrix}} & {\begin{matrix}(K(0),K(2))\\q^3 t^2\end{matrix}} & {\begin{matrix}(K(0),K(3))\\q^5t^3\end{matrix}} & \cdots
	\arrow["{\gamma\sbp 2}", dashed, from=1-1, to=1-2]
	\arrow[dashed, from=1-2, to=1-3]
	\arrow[dashed, from=1-3, to=1-4]
	\arrow[from=1-4, to=1-5]
\end{tikzcd}\]
\caption{Stability and contents in $(K,K)$}
\label{fig:kk}
\end{figure}

From this diagram, we read
\begin{equation}
\begin{aligned}
\sum_{\alpha\in \Xi\sbp{K,K}}\Cont(\alpha)&=\frac{1}{1-t}+\frac{qt}{1-t}+\frac{q^3t^2}{1-t}+\frac{q^5 t^3}{(1-t)(1-qt)}\\
&=\frac{1-q^2t^2+q^3t^2-q^4t^3+q^5t^3}{(1-t)(1-qt)}.
\end{aligned}
\end{equation}

Similar diagrams for other three color vectors are given in Figures \ref{fig:kj}, \ref{fig:jk} and \ref{fig:jj}.

\begin{figure}[h]
\[\begin{tikzcd}[sep=small]
\begin{matrix}
(K(0),J(1))\\q^3 t
\end{matrix} & \cdots
\arrow["\gamma\sbp 2", from=1-1,to=1-2]
\end{tikzcd}
\]
\[\sum_{\alpha\in\Xi\sbp{K,J}} \Cont(\alpha)=\frac{q^3 t}{(1-t)(1-qt)}.\]
\caption{Stability and contents in $(K,J)$}
\label{fig:kj}
\end{figure}

\begin{figure}[h]
\[\begin{tikzcd}[sep=small]
\begin{matrix}
(J(1),K(0))\\q^2 t
\end{matrix} & \begin{matrix}
(J(1),K(1))\\q^2 t^2
\end{matrix}& \cdots
\arrow[dashed, "\gamma\sbp 2", from=1-1,to=1-2]
\arrow[from=1-2, to=1-3]
\end{tikzcd}
\]
\[\sum_{\alpha\in\Xi\sbp{J,K}} \Cont(\alpha)=\frac{q^2 t}{1-t}+ \frac{q^2 t^2}{(1-t)(1-qt)}.\]
\caption{Stability and contents in $(J,K)$}
\label{fig:jk}
\end{figure}

\begin{figure}[h]
\[\begin{tikzcd}[sep=small]
\begin{matrix}
(J(1),J(1))\\q^4 t^2
\end{matrix} & \begin{matrix}
(J(1),J(2))\\q^4 t^3
\end{matrix}& \cdots
\arrow[dashed, "\gamma\sbp 2", from=1-1,to=1-2]
\arrow[from=1-2, to=1-3]
\end{tikzcd}
\]
\[\sum_{\alpha\in\Xi\sbp{J,J}} \Cont(\alpha)=\frac{q^4 t^2}{1-t}+ \frac{q^4 t^3}{(1-t)(1-qt)}.\]
\caption{Stability and contents in $(J,J)$}
\label{fig:jj}
\end{figure}

Summing up, we get
\begin{equation}
H_2(t)=\sum_{\alpha\in \Xi} \Cont(\alpha) = \frac{1+q^2 t+q^3 t+q^4 t^2}{(1-t) (1-qt)},
\end{equation}
where we note that all terms involving $t^3$ are cancelled. 

By \eqref{eq:quot_in_hilb}, we have
\begin{equation}
Q_2(t)=1+(1+q)tH_1(qt)+t^2 H_2(t) = \frac{1+q^2 t^2+q^3 t^2+q^4 t^4}{(1-t) (1-qt)},
\end{equation}
where we note unexpectedly that the numerator of $Q_2(t)$ is the numerator of $H_2(t)$ evaluated at $t\mapsto t^2$.

\subsubsection*{Case $d=3$} We now compute $H_3(t)$ by constructing similar diagrams for all eight color vectors, see Figures \ref{fig:jjj} through \ref{fig:kkk}. We put $\gamma\sbp 2$ on verticle arrows, $\gamma\sbp 3$ on horizontal arrows, and omit $\gamma\sbp 1$ with the understanding that every node is stable under $\gamma\sbp 1$. To save space, we also omit the first component in the grid, which is always $K(0)$ when the first component of the color vector is $K$, and $J(1)$ otherwise. For example, in color $(J,K,J)$, the node $(J(1),J(2),K(0))$ is denoted by $(J(2),K(0))$. For ease of verification, we label the content of each node in the form
\begin{equation}
\Cont(\alpha)=A(\alpha)\cdot B(\alpha)\cdot D(\alpha)t^{n(\alpha)}.
\end{equation}

We note that $A(\alpha)$ is generically unchanged by $\gamma\sbp j$, unless the distance between some $K$ components are stretched from $1$ to $2$ or from $2$ to $3$. The factor $B(\alpha)$ stays the same on each grid, and it only depends on the color vector (Lemma \ref{lem:B_factor}). The factor $D(\alpha)$ is generically multiplied by $q^{j-1}$ under $\gamma\sbp j$, unless some stretches are obstructed (see Lemma \ref{lem:D_factor}). The exponent $n(\alpha)$ is always increased by $1$ under $\gamma\sbp j$. The top-left node $\mathbf{0}_c$ of each color vector $c$ only consists of $K(0)$ and $J(1)$ components, and it has $A(\mathbf{0}_c)=1$, $D(\mathbf{0}_c)=q^{d\#_J(c)}=q^{3\#_J(c)}$, and $n(\mathbf{0}_c)=\#_J(c)$, where $\#_J(c)$ is the number of $J$'s in $c$. One can write down every factor but $A(\alpha)$ by starting from the top-left node and tracing the arrows using the rules above. For $A(\alpha)$, we consult Tables \ref{table:pure_K_2} and \ref{table:pure_K_3}. 

To obtain a rational formula for $H_3(t)$, we decompose each grid into a disjoint union of stable orbits, marked by enclosing frames. Note that our decomposition is different from the one in Theorem \ref{thm:cusp_formula} because we exploit stable arrows not guaranteed by Lemma \ref{lem:is_stable}; we only do so to reduce computation. Note also that the choice of the ``most efficient'' decomposition in Figure \ref{fig:kjk} is noncanonical. The contribution of each orbit to $H_3(t)$ is computed by Lemma \ref{lem:orbit_content}. For example, the top-right box of Figure \ref{fig:jjj} contributes $\dfrac{q^9t^4}{(1-t)(1-q^2t)}$.

\begin{figure}[h ]
\[\begin{tikzcd}[sep=small, every matrix/.append style={name=m}, /tikz/execute at end picture={
  \node[draw,fit=(m-1-1), inner sep=-0.5ex] {};
  \node[draw,fit=(m-1-2)(m-1-3), inner sep=-0.5ex] {};
  \node[draw,fit=(m-2-1),inner sep=-0.5ex] {};
  \node[draw,fit=(m-2-2)(m-2-3),inner sep=-0.5ex] {};
  \node[draw,fit=(m-3-1)(m-4-1),inner sep=-0.5ex] {};
  \node[draw,fit=(m-3-2)(m-3-3)(m-4-3),inner sep=-0.5ex] {};
}]
	{\begin{matrix}(J(1),J(1))\\1\cdot 1\cdot q^9t^3 \end{matrix}} & {\begin{matrix}(J(1),J(2))\\1\cdot 1\cdot q^9t^4 \end{matrix}} & \cdots \\
	{\begin{matrix}(J(2),J(1))\\1\cdot 1\cdot q^9t^4 \end{matrix}} & {\begin{matrix}(J(3),J(1))\\1\cdot 1\cdot q^{10}t^5 \end{matrix}} & \cdots \\
	{\begin{matrix}(J(2),J(2))\\1\cdot 1\cdot q^9t^5 \end{matrix}} & {\begin{matrix}(J(2),J(3))\\1\cdot 1\cdot q^{10}t^6 \end{matrix}} & \cdots \\
	\vdots & \vdots & {}
	\arrow["{\gamma\sbp 3}", dashed, from=1-1, to=1-2]
	\arrow["{\gamma\sbp 2}"', dashed, from=1-1, to=2-1]
	\arrow[dashed, from=2-1, to=2-2]
	\arrow[from=1-2, to=1-3]
	\arrow[dashed, from=1-2, to=2-2]
	\arrow[from=2-2, to=2-3]
	\arrow[dashed, from=2-1, to=3-1]
	\arrow[dashed, from=2-2, to=3-2]
	\arrow[dashed, from=3-1, to=3-2]
	\arrow[from=3-2, to=3-3]
	\arrow[from=3-2, to=4-2]
	\arrow[from=3-1, to=4-1]
\end{tikzcd}\]
\caption{Stability and contents in $(J,J,J)$}
\label{fig:jjj}
\end{figure}
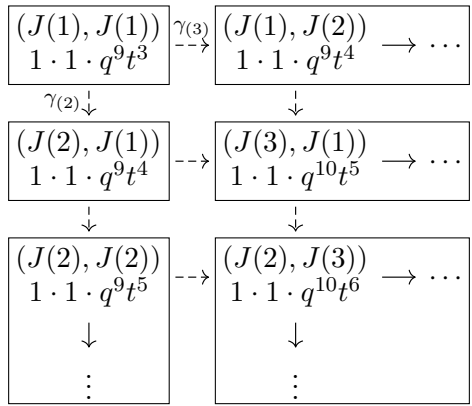

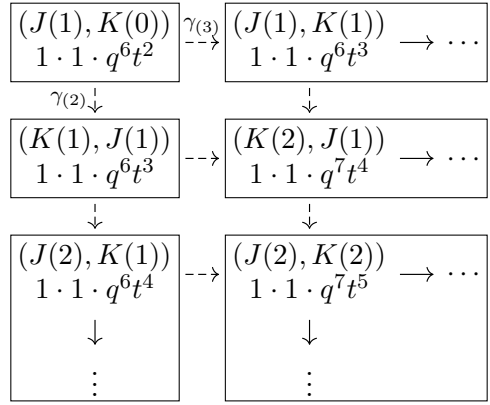
\begin{figure}[h ]
\[\begin{tikzcd}[sep=small, every matrix/.append style={name=m}, /tikz/execute at end picture={
  \node[draw,fit=(m-1-1), inner sep=-0.5ex] {};
  \node[draw,fit=(m-1-2)(m-1-3), inner sep=-0.5ex] {};
  \node[draw,fit=(m-2-1),inner sep=-0.5ex] {};
  \node[draw,fit=(m-2-2)(m-2-3),inner sep=-0.5ex] {};
  \node[draw,fit=(m-3-1)(m-4-1),inner sep=-0.5ex] {};
  \node[draw,fit=(m-3-2)(m-3-3)(m-4-3),inner sep=-0.5ex] {};
}]
	{\begin{matrix}(J(1),K(0))\\1\cdot 1\cdot q^6t^2 \end{matrix}} & {\begin{matrix}(J(1),K(1))\\1\cdot 1\cdot q^6t^3 \end{matrix}} & \cdots \\
	{\begin{matrix}(K(1),J(1))\\1\cdot 1\cdot q^6t^3 \end{matrix}} & {\begin{matrix}(K(2),J(1))\\1\cdot 1\cdot q^{7}t^4 \end{matrix}} & \cdots \\
	{\begin{matrix}(J(2),K(1))\\1\cdot 1\cdot q^6t^4 \end{matrix}} & {\begin{matrix}(J(2),K(2))\\1\cdot 1\cdot q^{7}t^5 \end{matrix}} & \cdots \\
	\vdots & \vdots & {}
	\arrow["{\gamma\sbp 3}", dashed, from=1-1, to=1-2]
	\arrow["{\gamma\sbp 2}"', dashed, from=1-1, to=2-1]
	\arrow[dashed, from=2-1, to=2-2]
	\arrow[from=1-2, to=1-3]
	\arrow[dashed, from=1-2, to=2-2]
	\arrow[from=2-2, to=2-3]
	\arrow[dashed, from=2-1, to=3-1]
	\arrow[dashed, from=2-2, to=3-2]
	\arrow[dashed, from=3-1, to=3-2]
	\arrow[from=3-2, to=3-3]
	\arrow[from=3-2, to=4-2]
	\arrow[from=3-1, to=4-1]
\end{tikzcd}\]
\caption{Stability and contents in $(J,J,K)$}
\label{fig:jjk}
\end{figure}

\begin{figure}[h ]
\[\begin{tikzcd}[sep=small, every matrix/.append style={name=m}, /tikz/execute at end picture={
  \node[draw,fit=(m-1-1), inner sep=-0.5ex] {};
  \node[draw,fit=(m-1-2)(m-1-3), inner sep=-0.5ex] {};
  \node[draw,fit=(m-2-1)(m-2-3),inner sep=-0.5ex] {};
  \node[draw,fit=(m-3-1)(m-3-3)(m-4-1),inner sep=-0.5ex] {};
}]
	{\begin{matrix}(K(0),J(1))\\1\cdot q\cdot q^6t^2 \end{matrix}} & {\begin{matrix}(K(0),J(2))\\1\cdot q\cdot q^7t^3 \end{matrix}} & \cdots \\
	{\begin{matrix}(J(2),K(0))\\1\cdot q\cdot q^6t^3 \end{matrix}} & {\begin{matrix}(J(3),K(0))\\1\cdot q\cdot q^{8}t^4 \end{matrix}} & \cdots \\
	{\begin{matrix}(K(1),J(2))\\1\cdot q\cdot q^6t^4 \end{matrix}} & {\begin{matrix}(K(1),J(3))\\1\cdot q\cdot q^{8}t^5 \end{matrix}} & \cdots \\
	\vdots & \vdots & {}
	\arrow["{\gamma\sbp 3}", dashed, from=1-1, to=1-2]
	\arrow["{\gamma\sbp 2}"', dashed, from=1-1, to=2-1]
	\arrow[from=2-1, to=2-2]
	\arrow[from=1-2, to=1-3]
	\arrow[dashed, from=1-2, to=2-2]
	\arrow[from=2-2, to=2-3]
	\arrow[dashed, from=2-1, to=3-1]
	\arrow[dashed, from=2-2, to=3-2]
	\arrow[from=3-1, to=3-2]
	\arrow[from=3-2, to=3-3]
	\arrow[from=3-2, to=4-2]
	\arrow[from=3-1, to=4-1]
\end{tikzcd}\]
\caption{Stability and contents in $(J,K,J)$}
\label{fig:jkj}
\end{figure}
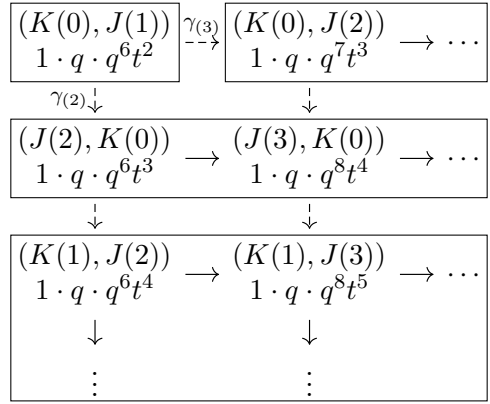

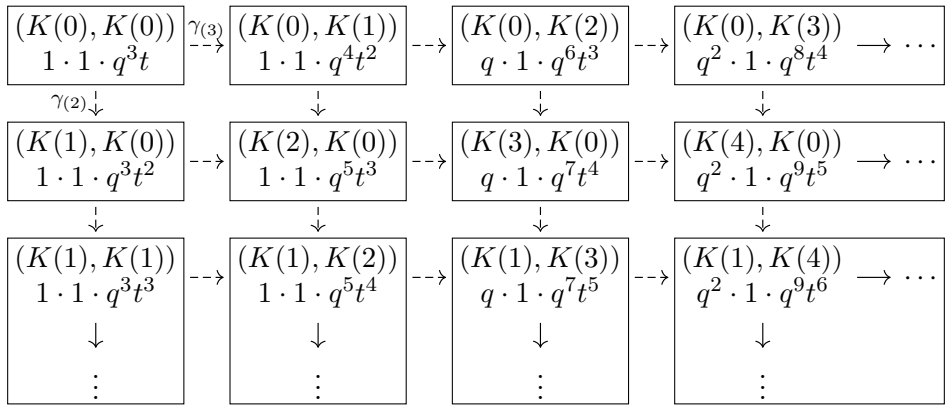
\begin{figure}[h ]
\[\begin{tikzcd}[sep=small, every matrix/.append style={name=m}, /tikz/execute at end picture={
  \node[draw,fit=(m-1-1), inner sep=-0.5ex] {};
  \node[draw,fit=(m-1-2), inner sep=-0.5ex] {};
  \node[draw,fit=(m-1-3), inner sep=-0.5ex] {};
  \node[draw,fit=(m-1-4)(m-1-5), inner sep=-0.5ex] {};
  \node[draw,fit=(m-2-1),inner sep=-0.5ex] {};
  \node[draw,fit=(m-2-2),inner sep=-0.5ex] {};
  \node[draw,fit=(m-2-3),inner sep=-0.5ex] {};
  \node[draw,fit=(m-2-4)(m-2-5),inner sep=-0.5ex] {};
  \node[draw,fit=(m-3-1)(m-4-1),inner sep=-0.5ex] {};
  \node[draw,fit=(m-3-2)(m-4-2),inner sep=-0.5ex] {};
  \node[draw,fit=(m-3-3)(m-4-3),inner sep=-0.5ex] {};
  \node[draw,fit=(m-3-4)(m-3-5)(m-4-4),inner sep=-0.5ex] {};
}]
	{\begin{matrix}(K(0),K(0))\\1\cdot 1\cdot q^3t \end{matrix}} & {\begin{matrix}(K(0),K(1))\\1\cdot 1\cdot q^4t^2 \end{matrix}} & {\begin{matrix}(K(0),K(2))\\q\cdot 1\cdot q^6t^3 \end{matrix}} & {\begin{matrix}(K(0),K(3))\\q^2\cdot 1\cdot q^8t^4 \end{matrix}} & \cdots \\
	{\begin{matrix}(K(1),K(0))\\1\cdot 1\cdot q^3t^2 \end{matrix}} & {\begin{matrix}(K(2),K(0))\\1\cdot 1\cdot q^{5}t^3 \end{matrix}} & {\begin{matrix}(K(3),K(0))\\q\cdot 1\cdot q^7t^4 \end{matrix}} & {\begin{matrix}(K(4),K(0))\\q^2\cdot 1\cdot q^9t^5 \end{matrix}} & \cdots \\
	{\begin{matrix}(K(1),K(1))\\1\cdot 1\cdot q^3t^3 \end{matrix}} & {\begin{matrix}(K(1),K(2))\\1\cdot 1\cdot q^{5}t^4 \end{matrix}} & {\begin{matrix}(K(1),K(3))\\q\cdot 1\cdot q^7t^5 \end{matrix}} & {\begin{matrix}(K(1),K(4))\\q^2\cdot 1\cdot q^9t^6 \end{matrix}} & \cdots \\
	\vdots & \vdots & \vdots & \vdots & {}
	\arrow["{\gamma\sbp 3}", dashed, from=1-1, to=1-2]
	\arrow["{\gamma\sbp 2}"', dashed, from=1-1, to=2-1]
	\arrow[dashed, from=2-1, to=2-2]
	\arrow[dashed, from=1-2, to=1-3]
	\arrow[dashed, from=1-2, to=2-2]
	\arrow[dashed, from=2-2, to=2-3]
	\arrow[dashed, from=2-1, to=3-1]
	\arrow[dashed, from=2-2, to=3-2]
	\arrow[dashed, from=3-1, to=3-2]
	\arrow[dashed, from=3-2, to=3-3]
	\arrow[from=3-2, to=4-2]
	\arrow[from=3-1, to=4-1]
	\arrow[dashed, from=1-3, to=1-4]
	\arrow[from=1-4, to=1-5]
	\arrow[from=2-4, to=2-5]
	\arrow[from=3-4, to=3-5]
	\arrow[from=3-4, to=4-4]
	\arrow[from=3-3, to=4-3]
	\arrow[dashed, from=3-3, to=3-4]
	\arrow[dashed, from=2-3, to=2-4]
	\arrow[dashed, from=2-3, to=3-3]
	\arrow[dashed, from=1-3, to=2-3]
	\arrow[dashed, from=1-4, to=2-4]
	\arrow[dashed, from=2-4, to=3-4]
\end{tikzcd}\]
\caption{Stability and contents in $(J,K,K)$}
\label{fig:jkk}
\end{figure}

\begin{figure}[h ]
\[\begin{tikzcd}[sep=small, every matrix/.append style={name=m}, /tikz/execute at end picture={
  \node[draw,fit=(m-1-1)(m-2-1), inner sep=-0.5ex] {};
  \node[draw,fit=(m-1-2)(m-1-3)(m-2-2),inner sep=-0.5ex] {};
}]
	{\begin{matrix}(J(1),J(1))\\1\cdot q^2\cdot q^6t^2 \end{matrix}} & {\begin{matrix}(J(1),J(2))\\1\cdot q^2\cdot q^7t^3 \end{matrix}} & \cdots \\
	\vdots & \vdots & {}
	\arrow["{\gamma\sbp 3}", dashed, from=1-1, to=1-2]
	\arrow[from=1-2, to=1-3]
	\arrow[from=1-2, to=2-2]
	\arrow["{\gamma\sbp 2}"', from=1-1, to=2-1]
\end{tikzcd}\]
\caption{Stability and contents in $(K,J,J)$}
\label{fig:kjj}
\end{figure}
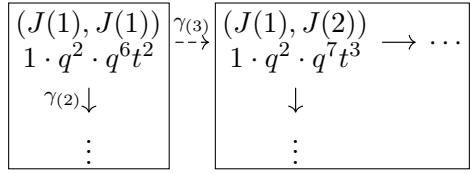

\begin{figure}[h ]
\[\begin{tikzcd}[sep=small, every matrix/.append style={name=m}, /tikz/execute at end picture={
  \node[draw,fit=(m-1-1), inner sep=-0.5ex] {};
  \node[draw,fit=(m-1-2), inner sep=-0.5ex] {};
  \node[draw,fit=(m-1-3),inner sep=-0.5ex] {};
  \node[draw,fit=(m-1-4)(m-1-5)(m-7-4),inner sep=-0.5ex] {};
  \node[draw,fit=(m-2-1),inner sep=-0.5ex] {};
  \node[draw,fit=(m-2-2),inner sep=-0.5ex] {};
  \node[draw,fit=(m-2-3)(m-7-3),inner sep=-0.5ex] {};
  \node[draw,fit=(m-3-1),inner sep=-0.5ex] {};
  \node[draw,fit=(m-3-2),inner sep=-0.5ex] {};
  \node[draw,fit=(m-4-1),inner sep=-0.5ex] {};
  \node[draw,fit=(m-4-2)(m-7-2),inner sep=-0.5ex] {};
  \node[draw,fit=(m-5-1),inner sep=-0.5ex] {};
  \node[draw,fit=(m-6-1)(m-7-1),inner sep=-0.5ex] {};
}]
	{\begin{matrix}(J(1),K(0))\\1\cdot q\cdot q^3t^1 \end{matrix}} & {\begin{matrix}(J(1),K(1))\\1\cdot q\cdot q^4t^2 \end{matrix}} & {\begin{matrix}(J(1),K(2))\\q\cdot q\cdot q^6t^3 \end{matrix}} & {\begin{matrix}(J(1),K(3))\\q^2\cdot q\cdot q^8t^4 \end{matrix}} & \cdots \\
	{\begin{matrix}(K(1),J(1))\\1\cdot q\cdot q^4t^2 \end{matrix}} & {\begin{matrix}(K(2),J(1))\\q\cdot q\cdot q^{5}t^3 \end{matrix}} & {\begin{matrix}(K(3),J(1))\\q^2\cdot q\cdot q^7t^4 \end{matrix}} & {\begin{matrix}(K(4),J(1))\\q^2\cdot q\cdot q^9t^5 \end{matrix}} & \cdots \\
	{\begin{matrix}(J(2),K(1))\\1\cdot q\cdot q^5t^3 \end{matrix}} & {\begin{matrix}(J(2),K(2))\\q\cdot q\cdot q^{6}t^4 \end{matrix}} & {\begin{matrix}(J(2),K(3))\\q^2\cdot q\cdot q^8t^5 \end{matrix}} & {\begin{matrix}(J(2),K(4))\\q^2\cdot q\cdot q^{10}t^6 \end{matrix}} & \cdots \\
	{\begin{matrix}(K(2),J(2))\\q\cdot q\cdot q^6t^4 \end{matrix}} & {\begin{matrix}(K(3),J(2))\\q^2\cdot q\cdot q^7t^5 \end{matrix}} & {\begin{matrix}(K(4),J(2))\\q^2\cdot q\cdot q^9t^6 \end{matrix}} & {\begin{matrix}(K(5),J(2))\\q^2\cdot q\cdot q^{11}t^7 \end{matrix}} & \cdots \\
	{\begin{matrix}(J(3),K(2))\\q\cdot q\cdot q^7t^5 \end{matrix}} & {\begin{matrix}(J(3),K(3))\\q^2\cdot q\cdot q^8t^6 \end{matrix}} & {\begin{matrix}(J(3),K(4))\\q^2\cdot q\cdot q^{10}t^7 \end{matrix}} & {\begin{matrix}(J(3),K(5))\\q^2\cdot q\cdot q^{12}t^8 \end{matrix}} & \cdots \\
	{\begin{matrix}(K(3),J(3))\\q^2\cdot q\cdot q^8t^6 \end{matrix}} & {\begin{matrix}(K(4),J(3))\\q^2\cdot q\cdot q^9t^7 \end{matrix}} & {\begin{matrix}(K(5),J(3))\\q^2\cdot q\cdot q^{11}t^8 \end{matrix}} & {\begin{matrix}(K(6),J(3))\\q^2\cdot q\cdot q^{13}t^9 \end{matrix}} & \cdots \\
	\vdots & \vdots & \vdots & \vdots & {}
	\arrow["{\gamma\sbp 3}", dashed, from=1-1, to=1-2]
	\arrow["{\gamma\sbp 2}"', dashed, from=1-1, to=2-1]
	\arrow[dashed, from=2-1, to=2-2]
	\arrow[dashed, from=1-2, to=1-3]
	\arrow[dashed, from=1-2, to=2-2]
	\arrow[dashed, from=2-2, to=2-3]
	\arrow[dashed, from=2-1, to=3-1]
	\arrow[dashed, from=2-2, to=3-2]
	\arrow[dashed, from=3-1, to=3-2]
	\arrow[dashed, from=3-2, to=3-3]
	\arrow[dashed, from=3-2, to=4-2]
	\arrow[dashed, from=3-1, to=4-1]
	\arrow[dashed, from=1-3, to=1-4]
	\arrow[from=1-4, to=1-5]
	\arrow[from=2-3, to=2-4]
	\arrow[dashed, from=1-3, to=2-3]
	\arrow[from=1-4, to=2-4]
	\arrow[from=2-4, to=2-5]
	\arrow[from=3-4, to=3-5]
	\arrow[from=2-3, to=3-3]
	\arrow[from=2-4, to=3-4]
	\arrow[from=3-3, to=3-4]
	\arrow[from=4-2, to=4-3]
	\arrow[from=3-3, to=4-3]
	\arrow[from=4-3, to=4-4]
	\arrow[from=3-4, to=4-4]
	\arrow[from=4-4, to=4-5]
	\arrow[from=4-4, to=5-4]
	\arrow[from=5-4, to=5-5]
	\arrow[from=5-4, to=6-4]
	\arrow[from=6-4, to=6-5]
	\arrow[from=6-4, to=7-4]
	\arrow[from=6-3, to=7-3]
	\arrow[from=6-2, to=7-2]
	\arrow[from=5-2, to=5-3]
	\arrow[from=4-3, to=5-3]
	\arrow[from=5-3, to=5-4]
	\arrow[from=5-3, to=6-3]
	\arrow[from=6-3, to=6-4]
	\arrow[from=4-2, to=5-2]
	\arrow[from=5-2, to=6-2]
	\arrow[from=6-2, to=6-3]
	\arrow[dashed, from=4-1, to=4-2]
	\arrow[dashed, from=4-1, to=5-1]
	\arrow[dashed, from=5-1, to=5-2]
	\arrow[dashed, from=5-1, to=6-1]
	\arrow[from=6-1, to=7-1]
	\arrow[dashed, from=6-1, to=6-2]
\end{tikzcd}\]
\caption{Stability and contents in $(K,J,K)$}
\label{fig:kjk}
\end{figure}
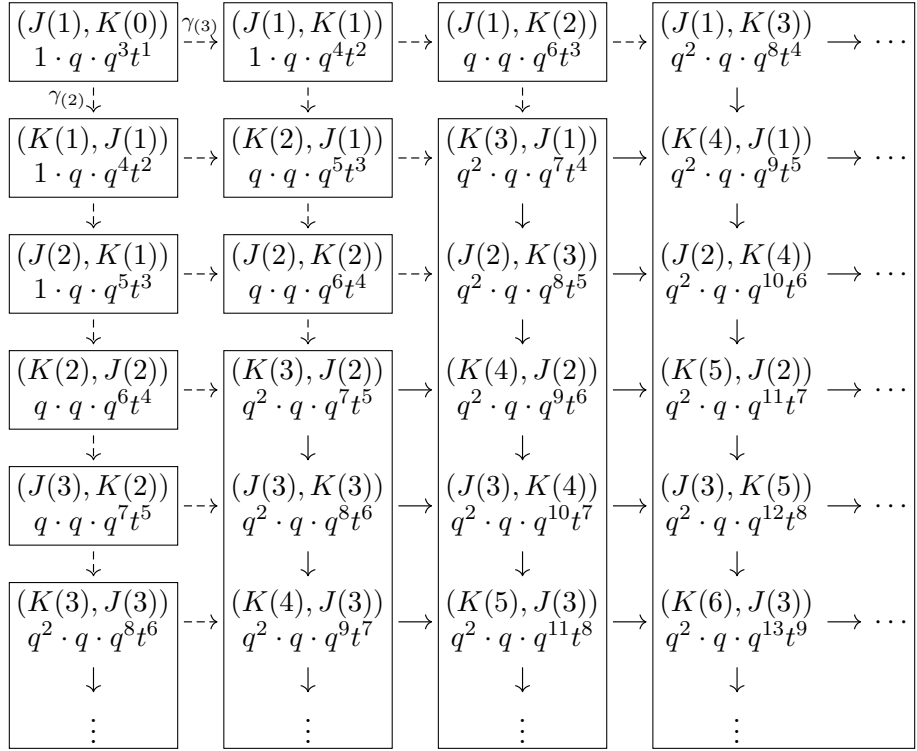

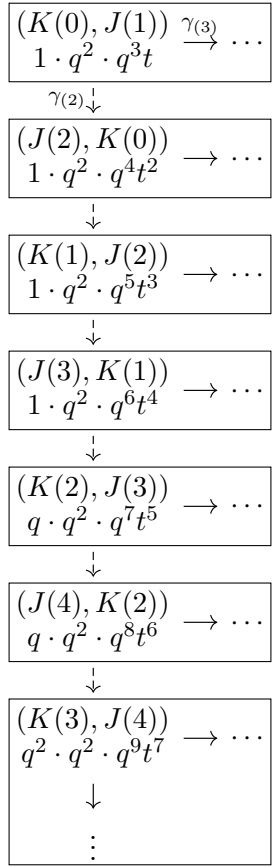
\begin{figure}[h ]
\[\begin{tikzcd}[sep=small, every matrix/.append style={name=m}, /tikz/execute at end picture={
  \node[draw,fit=(m-1-1)(m-1-2), inner sep=-0.5ex] {};
  \node[draw,fit=(m-2-1)(m-2-2), inner sep=-0.5ex] {};
  \node[draw,fit=(m-3-1)(m-3-2), inner sep=-0.5ex] {};
  \node[draw,fit=(m-4-1)(m-4-2), inner sep=-0.5ex] {};
  \node[draw,fit=(m-5-1)(m-5-2), inner sep=-0.5ex] {};
  \node[draw,fit=(m-6-1)(m-6-2), inner sep=-0.5ex] {};
  \node[draw,fit=(m-7-1)(m-7-2)(m-8-1), inner sep=-0.5ex] {};
}]
	{\begin{matrix}(K(0),J(1))\\1\cdot q^2\cdot q^3t \end{matrix}} & \cdots \\
	{\begin{matrix}(J(2),K(0))\\1\cdot q^2\cdot q^4t^2 \end{matrix}} & \cdots \\
	{\begin{matrix}(K(1),J(2))\\1\cdot q^2\cdot q^5t^3 \end{matrix}} & \cdots \\
	{\begin{matrix}(J(3),K(1))\\1\cdot q^2\cdot q^6t^4 \end{matrix}} & \cdots \\
	{\begin{matrix}(K(2),J(3))\\q\cdot q^2\cdot q^7t^5 \end{matrix}} & \cdots \\
	{\begin{matrix}(J(4),K(2))\\q\cdot q^2\cdot q^8t^6 \end{matrix}} & \cdots \\
	{\begin{matrix}(K(3),J(4))\\q^2\cdot q^2\cdot q^9t^7 \end{matrix}} & \cdots \\
	\vdots
	\arrow["{\gamma\sbp 3}", from=1-1, to=1-2]
	\arrow["{\gamma\sbp 2}"', dashed, from=1-1, to=2-1]
	\arrow[from=2-1, to=2-2]
	\arrow[dashed, from=2-1, to=3-1]
	\arrow[from=3-1, to=3-2]
	\arrow[dashed, from=3-1, to=4-1]
	\arrow[dashed, from=4-1, to=5-1]
	\arrow[from=4-1, to=4-2]
	\arrow[from=5-1, to=5-2]
	\arrow[dashed, from=5-1, to=6-1]
	\arrow[from=6-1, to=6-2]
	\arrow[dashed, from=6-1, to=7-1]
	\arrow[from=7-1, to=7-2]
	\arrow[from=7-1, to=8-1]
\end{tikzcd}\]
\caption{Stability and contents in $(K,K,J)$}
\label{fig:kkj}
\end{figure}

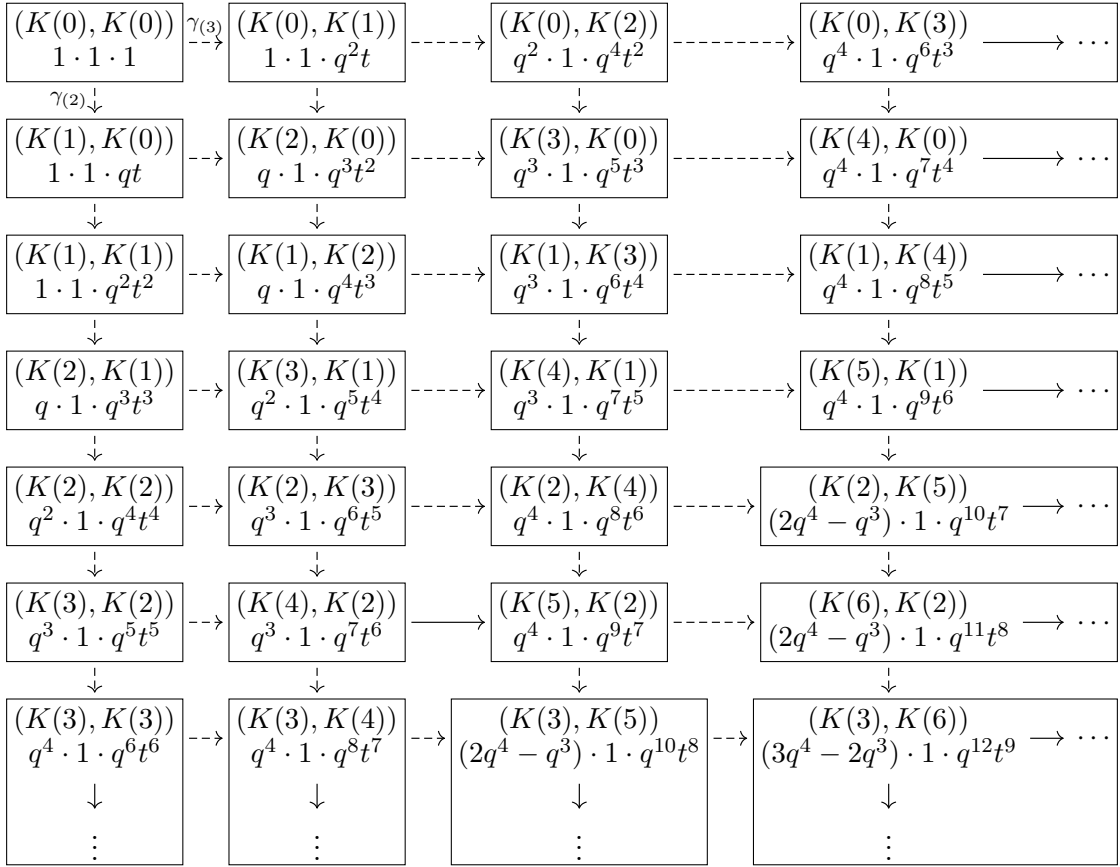
\begin{figure}[h ]
\[\begin{tikzcd}[sep=small, every matrix/.append style={name=m}, /tikz/execute at end picture={
  \node[draw,fit=(m-1-1), inner sep=-0.5ex] {};
  \node[draw,fit=(m-1-2), inner sep=-0.5ex] {};
  \node[draw,fit=(m-1-3), inner sep=-0.5ex] {};
  \node[draw,fit=(m-2-1), inner sep=-0.5ex] {};
  \node[draw,fit=(m-2-2), inner sep=-0.5ex] {};
  \node[draw,fit=(m-2-3), inner sep=-0.5ex] {};
  \node[draw,fit=(m-3-1), inner sep=-0.5ex] {};
  \node[draw,fit=(m-3-2), inner sep=-0.5ex] {};
  \node[draw,fit=(m-3-3), inner sep=-0.5ex] {};
  \node[draw,fit=(m-4-1), inner sep=-0.5ex] {};
  \node[draw,fit=(m-4-2), inner sep=-0.5ex] {};
  \node[draw,fit=(m-4-3), inner sep=-0.5ex] {};
  \node[draw,fit=(m-5-1), inner sep=-0.5ex] {};
  \node[draw,fit=(m-5-2), inner sep=-0.5ex] {};
  \node[draw,fit=(m-5-3), inner sep=-0.5ex] {};
  \node[draw,fit=(m-6-1), inner sep=-0.5ex] {};
  \node[draw,fit=(m-6-2), inner sep=-0.5ex] {};
  \node[draw,fit=(m-6-3), inner sep=-0.5ex] {};
  \node[draw,fit=(m-1-4)(m-1-5), inner sep=-0.5ex] {};
  \node[draw,fit=(m-2-4)(m-2-5), inner sep=-0.5ex] {};
  \node[draw,fit=(m-3-4)(m-3-5), inner sep=-0.5ex] {};
  \node[draw,fit=(m-4-4)(m-4-5), inner sep=-0.5ex] {};
  \node[draw,fit=(m-5-4)(m-5-5), inner sep=-0.5ex] {};
  \node[draw,fit=(m-6-4)(m-6-5), inner sep=-0.5ex] {};
  \node[draw,fit=(m-7-1)(m-8-1), inner sep=-0.5ex] {};
  \node[draw,fit=(m-7-2)(m-8-2), inner sep=-0.5ex] {};
  \node[draw,fit=(m-7-3)(m-8-3), inner sep=-0.5ex] {};
  \node[draw,fit=(m-7-4)(m-7-5)(m-8-4), inner sep=-0.5ex] {};
}]
	{\begin{matrix}(K(0),K(0))\\1\cdot 1\cdot 1 \end{matrix}} & {\begin{matrix}(K(0),K(1))\\1\cdot 1\cdot q^2t \end{matrix}} & {\begin{matrix}(K(0),K(2))\\q^2\cdot 1\cdot q^4t^2 \end{matrix}} & {\begin{matrix}(K(0),K(3))\\q^4\cdot 1\cdot q^6t^3 \end{matrix}} & \cdots \\
	{\begin{matrix}(K(1),K(0))\\1\cdot 1\cdot qt \end{matrix}} & {\begin{matrix}(K(2),K(0))\\q\cdot 1\cdot q^3t^2 \end{matrix}} & {\begin{matrix}(K(3),K(0))\\q^3\cdot 1\cdot q^5t^3 \end{matrix}} & {\begin{matrix}(K(4),K(0))\\q^4\cdot 1\cdot q^7t^4 \end{matrix}} & \cdots \\
	{\begin{matrix}(K(1),K(1))\\1\cdot 1\cdot q^2t^2 \end{matrix}} & {\begin{matrix}(K(1),K(2))\\q\cdot 1\cdot q^4t^3 \end{matrix}} & {\begin{matrix}(K(1),K(3))\\q^3\cdot 1\cdot q^6t^4 \end{matrix}} & {\begin{matrix}(K(1),K(4))\\q^4\cdot 1\cdot q^8t^5 \end{matrix}} & \cdots \\
	{\begin{matrix}(K(2),K(1))\\q\cdot 1\cdot q^3t^3 \end{matrix}} & {\begin{matrix}(K(3),K(1))\\q^2\cdot 1\cdot q^5t^4 \end{matrix}} & {\begin{matrix}(K(4),K(1))\\q^3\cdot 1\cdot q^7t^5 \end{matrix}} & {\begin{matrix}(K(5),K(1))\\q^4\cdot 1\cdot q^9t^6 \end{matrix}} & \cdots \\
	{\begin{matrix}(K(2),K(2))\\q^2\cdot 1\cdot q^4t^4 \end{matrix}} & {\begin{matrix}(K(2),K(3))\\q^3\cdot 1\cdot q^6t^5 \end{matrix}} & {\begin{matrix}(K(2),K(4))\\q^4\cdot 1\cdot q^8t^6 \end{matrix}} & {\begin{matrix}(K(2),K(5))\\(2q^4-q^3)\cdot 1\cdot q^{10}t^7 \end{matrix}} & \cdots \\
	{\begin{matrix}(K(3),K(2))\\q^3\cdot 1\cdot q^5t^5 \end{matrix}} & {\begin{matrix}(K(4),K(2))\\q^3\cdot 1\cdot q^7t^6 \end{matrix}} & {\begin{matrix}(K(5),K(2))\\q^4\cdot 1\cdot q^9t^7 \end{matrix}} & {\begin{matrix}(K(6),K(2))\\(2q^4-q^3)\cdot 1\cdot q^{11}t^8 \end{matrix}} & \cdots \\
	{\begin{matrix}(K(3),K(3))\\q^4\cdot 1\cdot q^6t^6 \end{matrix}} & {\begin{matrix}(K(3),K(4))\\q^4\cdot 1\cdot q^8t^7 \end{matrix}} & {\begin{matrix}(K(3),K(5))\\(2q^4-q^3)\cdot 1\cdot q^{10}t^8 \end{matrix}} & {\begin{matrix}(K(3),K(6))\\(3q^4-2q^3)\cdot 1\cdot q^{12}t^9 \end{matrix}} & \cdots \\
	\vdots & \vdots & \vdots & \vdots & {}
	\arrow["{\gamma\sbp 3}", dashed, from=1-1, to=1-2]
	\arrow["{\gamma\sbp 2}"', dashed, from=1-1, to=2-1]
	\arrow[dashed, from=2-1, to=2-2]
	\arrow[dashed, from=1-2, to=1-3]
	\arrow[dashed, from=1-2, to=2-2]
	\arrow[dashed, from=2-2, to=2-3]
	\arrow[dashed, from=2-1, to=3-1]
	\arrow[dashed, from=2-2, to=3-2]
	\arrow[dashed, from=3-1, to=3-2]
	\arrow[dashed, from=3-2, to=3-3]
	\arrow[dashed, from=3-2, to=4-2]
	\arrow[dashed, from=3-1, to=4-1]
	\arrow[dashed, from=1-3, to=1-4]
	\arrow[from=1-4, to=1-5]
	\arrow[dashed, from=2-3, to=2-4]
	\arrow[dashed, from=1-3, to=2-3]
	\arrow[dashed, from=1-4, to=2-4]
	\arrow[from=2-4, to=2-5]
	\arrow[from=3-4, to=3-5]
	\arrow[dashed, from=2-3, to=3-3]
	\arrow[dashed, from=2-4, to=3-4]
	\arrow[dashed, from=3-3, to=3-4]
	\arrow[dashed, from=4-2, to=4-3]
	\arrow[dashed, from=3-3, to=4-3]
	\arrow[dashed, from=4-3, to=4-4]
	\arrow[dashed, from=3-4, to=4-4]
	\arrow[from=4-4, to=4-5]
	\arrow[dashed, from=4-4, to=5-4]
	\arrow[from=5-4, to=5-5]
	\arrow[dashed, from=5-4, to=6-4]
	\arrow[from=6-4, to=6-5]
	\arrow[dashed, from=6-4, to=7-4]
	\arrow[dashed, from=6-3, to=7-3]
	\arrow[dashed, from=6-2, to=7-2]
	\arrow[dashed, from=5-2, to=5-3]
	\arrow[dashed, from=4-3, to=5-3]
	\arrow[dashed, from=5-3, to=5-4]
	\arrow[dashed, from=5-3, to=6-3]
	\arrow[dashed, from=6-3, to=6-4]
	\arrow[dashed, from=4-2, to=5-2]
	\arrow[dashed, from=5-2, to=6-2]
	\arrow[from=6-2, to=6-3]
	\arrow[dashed, from=4-1, to=4-2]
	\arrow[dashed, from=4-1, to=5-1]
	\arrow[dashed, from=5-1, to=5-2]
	\arrow[dashed, from=5-1, to=6-1]
	\arrow[dashed, from=6-1, to=7-1]
	\arrow[dashed, from=6-1, to=6-2]
	\arrow[from=7-1, to=8-1]
	\arrow[from=7-2, to=8-2]
	\arrow[from=7-3, to=8-3]
	\arrow[dashed, from=7-1, to=7-2]
	\arrow[dashed, from=7-2, to=7-3]
	\arrow[dashed, from=7-3, to=7-4]
	\arrow[from=7-4, to=8-4]
	\arrow[from=7-4, to=7-5]
\end{tikzcd}\]
\caption{Stability and contents in $(K,K,K)$}
\label{fig:kkk}
\end{figure}

Summing up the contributions of 78 orbits in the eight diagrams, we obtain
\begin{equation}
H_3(t)=\frac{1+q^3t+q^4t+q^5t+ q^6t^2+q^7t^2+q^8t^2+q^9t^3}{(1-t)(1-qt)(1-q^2t)}
\end{equation}
and \eqref{eq:quot_in_hilb} gives
\begin{equation}
Q_3(t)=\frac{1+q^3t^2+q^4t^2 +q^5t^2+ q^6t^4+ q^7t^4+q^8t^4+q^9t^6}{(1-t)(1-qt)(1-q^2t)},
\end{equation}
where we note again that the numerator of $Q_3(t)$ is the numerator of $H_3(t)$ evaluated at $t\mapsto t^2$. 

The simplicity of the formula for $H_3(t)$ is striking. In terms of the $t$-degrees of their contributions to the numerator of $H_3(t)$, individual orbits have contributions up to $t^9$, and the total contribution of a color vector has contribution up to $t^7$, but the $t$-degree of the numerator of $H_3(t)$ is only $3$. See Table \ref{table:color_3} for a breakdown of $H_3(t)$ in each color, where we also mark the number of orbits used in our computation. We multiply the content by $(t;q)_3=(1-t)(1-qt)(1-q^2t)$ to extract the contribution to the numerator of $H_3(t)$.

\begin{table}[h]
{\renewcommand{\arraystretch}{1.25}%
\begin{tabular}{|c|c|p{0.7\textwidth}|}
\hline
$c$ & \# orbits & $(t;q)_3\sum_{\alpha\in \Xi_c}\Cont(\alpha)$ \\
\hline\hline
$(J,J,J)$ & 6 & $q^6 t^2+2 q^6 t^3-q^7 t^3-q^8 t^3+q^6 t^4-q^7 t^4-q^8 t^4+q^9 t^4+q^7 t^5-2 q^8 t^5+q^9 t^5$\\
\hline
$(J,J,K)$ & 6 & $q^9 t^3+2 q^9 t^4-q^{10} t^4-q^{11} t^4+q^9 t^5-q^{10} t^5-q^{11} t^5+q^{12} t^5+q^{10} t^6-2 q^{11} t^6+q^{12} t^6$\\
\hline
$(J,K,J)$ & 4 & $q^7 t^2+q^7 t^3-q^9 t^3+q^7 t^4-q^8 t^4-q^9 t^4+q^{10} t^4$ \\
\hline
$(J,K,K)$ & 12 & $q^3 t+q^3 t^2-q^5 t^2+q^3 t^3-q^4 t^3-q^5 t^3+q^7 t^3-q^9 t^4+q^{10} t^4-q^7 t^5+2 q^8 t^5-q^9 t^5-q^{10} t^6+2 q^{11} t^6-q^{12} t^6$\\
\hline
$(K,J,J)$ & 2 & $q^8 t^2+q^9 t^3-q^{10} t^3$\\
\hline
$(K,J,K)$ & 13 & $q^4 t+q^5 t^2-q^6 t^2-q^6 t^3+q^8 t^3-q^7 t^4+2 q^8 t^4-2 q^9 t^4+q^{11} t^4+q^{11} t^5-q^{12} t^5-q^{10} t^6+2 q^{11} t^6-q^{12} t^6+q^{12} t^7-q^{13} t^7$ \\
\hline
$(K,K,J)$ & 7 & $q^5 t-q^9 t^5+q^{10} t^5-q^{12} t^7+q^{13} t^7$ \\
\hline
$(K,K,K)$ & 28 & $1-q^3 t^2+q^6 t^2-q^3 t^3+q^4 t^3+q^5 t^3-q^6 t^3-q^7 t^3+q^{10} t^3-q^6 t^4+q^7 t^4+q^9 t^4-q^{10} t^4+q^{10} t^6-2 q^{11} t^6+q^{12} t^6$\\
\hline \hline
Total & 78 & $1+q^3t+q^4t+q^5t+ q^6t^2+q^7t^2+q^8t^2+q^9t^3$\\
\hline
\end{tabular}
}
\caption{$H_3(t)$ from each color}
\label{table:color_3}
\end{table}

\subsection{Conjectural formulas for general $d$}\label{sec:conj}

Following the notation of Theorem \ref{thm:cusp_t_rationality}, let $\NH_d(t)=(t;q)_d H_d(t)$ and $\NQ_d(t)=(t;q)_d Q_d(t)$. They are now known to be polynomials in $t$ by Theorem \ref{thm:cusp_formula}. 

We translate Conjecture \ref{conj:quot}(b) together with the conjectural \eqref{eq:nq_in_nh} into a version for $\NH_d(t)$ using \eqref{eq:quot_in_hilb}.

\begin{conjecture}\label{conj:hd_pattern}
We have
\begin{enumerate}
\item $\NH_d(t^2)=\sum_{r=0}^d t^r {d \brack r}_q (t;q)_{d-r} \NH_r(tq^{d-r})$;
\item \tuparens{Functional Equation} $q^{d^2}t^d \NH_d(q^{-2d}t^{-1})=\NH_d(t)$.
\end{enumerate}
\end{conjecture}

We have just verified the above conjecture for $d\leq 3$. It turns out Conjecture \ref{conj:hd_pattern} is strong enough to \emph{overdetermine} $\NH_d(t)$ for all $d$. The following example for $d=4$ will demonstrate how to find $\NH_d(t)$ inductively and verify the consitency of Conjecture \ref{conj:hd_pattern} for $d$ up to a given bound. 

Assume Conjecture \ref{conj:hd_pattern}. Conjecture \ref{conj:hd_pattern}(a) then implies that $\NH_d(t)$ has constant term $1$. By Conjecture \ref{conj:hd_pattern}(b), the $t$-leading term of $\NH_d(t)$ is $q^{d^2}t^d$. 

For any commutative ring $A$ with 1, let $A_r[t]$ be the $A$-module of $A$-polynomials of degree at most $r$. Consider the $A$-linear map $\Theta_d:A_d[t]\to A_{2d-1}[t]$ defined by
\begin{equation}
\Theta_d(f(t)):=f(t^2)-t^d f(t).
\end{equation}

Then Conjecture \ref{conj:hd_pattern}(a) can be rewritten as
\begin{equation}\label{eq:nh_induction}
\Theta_d(\NH_d(t)) = \sum_{r=0}^{d-1} t^r {d \brack r}_q (t;q)_{d-r} \NH_r(tq^{d-r}).
\end{equation}

\begin{example}
When $d=4$, based on the known expressions for $d\leq 3$, the equation \eqref{eq:nh_induction} becomes
\begin{equation}\label{eq:nh_induction_4}
\begin{multlined}
\Theta_4(\NH_4(t))=1+(q^4+q^5+q^6+q^7) t^2+(-1+q^8+q^9+2 q^{10}+q^{11}+q^{12}) t^4+\\
(-q^4-q^5-q^6-q^7) t^5+
(-q^8-q^9-2 q^{10}-q^{11}+q^{13}+q^{14}+q^{15}) t^6+(-q^{12}-q^{13}-q^{14}-q^{15}) t^7.
\end{multlined}
\end{equation}

Matching the $t^4$- through $t^7$-coefficients of $\Theta_4(\NH_4(t))$, the only possibility for $\NH_4(t)$ is
\begin{equation}\label{eq:nh_4_formula}
\NH_4(t)=1 + (q^4 + q^5 + q^6 + q^7) t + (q^8 + q^9 + 2 q^{10} + q^{11} + q^{12}) t^2 + (q^{12} + q^{13} + q^{14} + q^{15}) t^3 + q^{16} t^4.
\end{equation}

Here, the $t^4$-coefficient is determined by the functional equation, not by \eqref{eq:nh_induction_4}. One can verify that $\NH_4(t)$ indeed satisfies \eqref{eq:nh_induction_4} and Conjecture \ref{conj:hd_pattern}(b). We thus conclude that \eqref{eq:nh_4_formula} is the unique $\NH_4(t)$ that makes Conjecture \ref{conj:hd_pattern} consistent up to $d\leq 4$.

We also notice that $\NH_4(t)$ is a polynomial in $q,t$, even if we have not assumed so. 
\end{example}

We are able to determine the exact formulas of $\NH_d(t)$ assuming Conjecture \ref{conj:hd_pattern}.

\begin{proposition}\label{prop:q-series}
Assume $A$ is a commutative ring with 1, and $q$ is an element of $A$. Then there are unique polynomials $\NH_d(t)\in A[t]$ for $d\geq 0$ such that
\begin{enumerate}
\item $\NH_0(t)=1$;
\item $\NH_d(t^2)=\sum_{r=0}^d t^r {d \brack r}_q (t;q)_{d-r} \NH_r(tq^{d-r})$;
\item \tuparens{Functional Equation} $\NH_d(t)=\sum_{i=0}^d a_i t^i$, with $a_{d-i}=q^{d(d-2i)} a_i$ for $0\leq i\leq \floor{d/2}$.
\end{enumerate}

Moreover, $\NH_d(t)$ is given by a polynomial $\NH_d(t;q)$ in $t,q$ with nonnegative integer coefficients, defined by the following:
\begin{equation}\label{eq:nh_guess}
\NH_d(t;q):=\sum_{j=0}^d {d \brack j}_{q} (q^d t)^j
\end{equation}

In addition, let $H_d(t;q):=\NH_d(t;q)/(t;q)_d$, and $\Zhat(t;q)=\sum_{d=0}^\infty \dfrac{q^{-d^2}t^d}{(q^{-1};q^{-1})_d} H_{d}(q^{-d}t;q)$ \tuparens{compare \eqref{eq:zhat_in_hilb}}, then
\begin{equation}\label{eq:cohen_lenstra_guess}
\Zhat(t;q)=\frac{1}{(tq^{-1};q^{-1})_\infty} \sum_{n=0}^\infty \frac{q^{-n^2}t^{2n}}{(q^{-1};q^{-1})_n}.
\end{equation}
\end{proposition}

\begin{proof}
The uniqueness part is argued before. To prove the existence, it suffices to define $\NH_d(t;q)$ by \eqref{eq:nh_guess} and show that $\NH_d(t):=\NH_d(t;q)$ satisfies the equations (a)(b)(c) and \eqref{eq:cohen_lenstra_guess}.

The equations (a)(c) are trivial. We now prove (b), starting from the right-hand side. We will use the change of variable $b=r-j$:
\begin{align}
&\sum_{r=0}^d t^r {d \brack r}_q (t;q)_{d-r} \NH_r(tq^{d-r})\\
&=\sum_{0\leq j\leq r\leq d} t^r (q^d t)^j (t;q)_{d-r} {d \brack r}_q {r\brack j}_q \\
&= \sum_{j=0}^d t^{2j}q^{dj} {d \brack j}_q \sum_{b=0}^{d-j} t^b (t;q)_{d-j-b} {d-j \brack b}_q.
\end{align}

Now let $m=d-j$, and it suffices to verify that $\sum_{b=0}^m t^{m-b}(t;q)_b {m \brack b}_q=1$. Letting $a=m-b$, we have
\begin{align}
\sum_{b=0}^m t^{m-b}(t;q)_b {m \brack b}_q &= (q;q)_m \sum_{a+b=m} \frac{t^a}{(q;q)_a}\frac{(t;q)_b}{(q;q)_b}\\
&= (q;q)_m [z^m]\parens*{ \sum_{a=0}^\infty \frac{(tz)^a}{(q;q)_a} \sum_{b=0}^\infty \frac{(t;q)_b}{(q;q)_b}z^b},
\end{align}
where $[z^m]$ denotes extracting the $z^m$-coefficient. 

By identities of Euler \cite[Corollary 2.2]{andrewspartitions} and Cauchy \cite[Theorem 2.1]{andrewspartitions}, we have
\begin{equation}
\sum_{a=0}^\infty \frac{(tz)^a}{(q;q)_a} = \frac{1}{(tz;q)_\infty}
\end{equation}
and
\begin{equation}
\sum_{b=0}^\infty \frac{(t;q)_b}{(q;q)_b}z^b=\frac{(tz;q)_\infty}{(z;q)_\infty}.
\end{equation}

By Euler's identity again, we have
\begin{equation}
[z^m]\parens*{ \sum_{a=0}^\infty \frac{(tz)^a}{(q;q)_a} \sum_{b=0}^\infty \frac{(t;q)_b}{(q;q)_b}z^b}=[z^m]\frac{1}{(z;q)_\infty}=\frac{1}{(q;q)_m},
\end{equation}
so the claim follows, finishing the proof of (b).

Finally, we prove \eqref{eq:cohen_lenstra_guess}, where we will use the change of variable $i=d-j$ and ${d \brack j}_q=q^{ij}{d \brack j}_{q^{-1}}$:

\begin{align}
\Zhat(t;q)&=\sum_{d=0}^\infty \dfrac{q^{-d^2}t^d}{(q^{-1};q^{-1})_d} H_{d}(q^{-d}t;q)\\
&=\sum_{j=0}^\infty \frac{t^{2j}q^{-j^2}}{(q^{-1};q^{-1})_j}\sum_{i=0}^\infty \frac{q^{-i(i+j)}}{(q^{-1};q^{-1})_i(tq^{-1};q^{-1})_{i+j}}.
\end{align}

But 
\begin{equation}
\sum_{i=0}^\infty \frac{q^{-i(i+j)}}{(q^{-1};q^{-1})_i(tq^{-1};q^{-1})_{i+j}} = \frac{1}{(tq^{-1};q^{-1})_\infty}
\end{equation}
using a standard argument involving the $j$-Durfee rectangle of a partition; see for instance \cite{gordonhouton1968ii}. 
\end{proof}

\begin{corollary}
The motivic version of Conjecture \ref{conj:hd_pattern} imply Conjecture \ref{conj:cusp_exact_formula}.
\end{corollary}
\begin{proof}
Apply Proposition \ref{prop:q-series} with $A=\KVar{k}$ and $q=\L$.
\end{proof}

The above lemma suggests a ``conceptual proof'' of Conjecture \ref{conj:cusp_exact_formula} without explicit computations and without knowing in advance whether the relevant motives are polynomials of $\L$. Since Conjecture \ref{conj:cusp_exact_formula} implies Conjecture \ref{conj:L_rationality}, this provides a plan to attack Conjecture \ref{conj:L_rationality} without computing the Gr\"obner strata and proving Conjecture \ref{conj:L_rationality_strata}. Of course, Conjecture \ref{conj:L_rationality} does not imply the stratum-wise polynomiality statement of Conjecture \ref{conj:L_rationality_strata}, so the latter conjecture would not be obsolete even if the former is proved to be true.

\

We make a further remark of the formula \eqref{eq:cohen_lenstra_guess}. Let
\begin{equation}
F(t;q)=\sum_{n=0}^\infty \frac{q^{n^2}t^{n}}{(q;q)_n}
\end{equation}
be a Rogers--Ramanujan series (see \cite[Chapter 7.1]{andrewspartitions}), then $\Zhat_R(t)=\frac{1}{(tq^{-1};q^{-1})_\infty} F(t^2;q^{-1})$, so the Rogers--Ramanujan identities would imply
\begin{align}
\Zhat_R(1)&=\frac{1}{(q^{-1};q^{-1})_\infty (q^{-1};q^{-5})_\infty (q^{-4};q^{-5})_\infty};
\\
\Zhat_R(-1)&=\frac{(q^{-1};q^{-2})_\infty}{(q^{-1};q^{-5})_\infty (q^{-4};q^{-5})_\infty}.
\end{align}

If true, this would provide another evidence for an observation in \cite{huang2023mutually} that $\Zhat_R(\pm 1)$ tend to have special values that admit infinite product formulas, if $R$ is the local ring at a point on a curve. The conjectured formula \eqref{eq:cohen_lenstra_guess} also gives the first guess for the number of matrix pairs $\set{(A,B)\in \Mat_n(\Fq)^2:AB=BA,A^2=B^3}$. Via \eqref{eq:stack_quotient} and \eqref{eq:cohen_lenstra_in_coh}, our guess \eqref{eq:cohen_lenstra_guess} is equivalent to
\begin{equation}\label{eq:matrix_count}
\abs[\big]{\set{(A,B)\in \Mat_n(\Fq)^2:AB=BA,A^2=B^3}}=\sum _{j=0}^{\floor{\frac{n}{2}}} (-1)^j q^{\frac{1}{2} \left(3 j^2-j\right)+n (n-2 j)} \frac{(q;q)_n}{(q;q)_j(q;q)_{n-2j}}.
\end{equation}

Unlike its nodal analog (which now has two proofs \cite{fulmanguralnick2022, huang2023mutually} by counting matrices $AB=BA=0$), the matrix counting problem \eqref{eq:matrix_count} is not known to have a direct approach.

We make another observation about $\NH_d(t;q)$ defined in \eqref{eq:nh_guess}. 

\begin{proposition}
Let $\zeta_r$ be a primitive $r$-th root of unity in $\C$. Then
\begin{equation}
\NH_d(t;\zeta_r)=(1+t^r)^{d/r}=\NH_1(t^r;1)^{d/r}.
\end{equation}
if $r$ divides $d$. Equivalently, for any $i\in \Z$, let $e=\gcd(i,d)$, then we have
\begin{equation}
\NH_d(t;\zeta_d^i)=(1+t^{d/e})^e.
\end{equation}
\end{proposition}

\begin{proof}
We first claim $\NH_d(t;q)$ is equal to a ``shifted Pochhammer":
\begin{equation}
\NH_d(t;q)=\sum_{j=0}^d q^{{j+1\choose 2}+j(d-j)} c_j(q)t^j,
\end{equation}
where $\sum_{j=0}^d c_j(q)t^j=(-t;q)_d$. To verify the claim, we simply note that $c_j(q)=q^{j\choose 2} {d\ brack j}_q$ by the Cauchy binomial theorem $(-t;q)_d=\sum_{j=0}^d q^{j\choose 2} {d\ brack j}_q$.

Assume $r>0$ divides $d$. Putting $q=\zeta_r$ in \eqref{eq:nh_guess}, we get
\begin{equation}\label{eq:nh_rou}
\NH_d(t;\zeta_r)=\sum_{j=0}^d \zeta_r^{{j+1\choose 2}+j(d-j)} c_j(\zeta_r) t^j,
\end{equation}
where $\sum_{j=0}^d c_j(\zeta_r)t^j=(-t;\zeta_r)_d$. Since $r$ divides $d$, we have $\zeta_r^d=1$, so the factor $\zeta_r^{{j+1\choose 2}+j(d-j)}$ is just $\zeta_r^{-{j \choose 2}}$. Using
\begin{equation}
(t;\zeta_r)_d=((1-t)(1-\zeta_r t)\dots (1-\zeta_r^{r-1}t))^{d/r} = (1-t^r)^{d/r}
\end{equation}
and substituting $t\mapsto -t$, we get $(-t;\zeta_r)_d=(1-(-t)^r)^{d/r}$, so that $c_j(\zeta_r)=0$ if $r\nmid j$, and
\begin{equation}
c_{ir}(\zeta_r)={d/r \choose i}(-(-1)^r)^i={d/r \choose i}(-1)^{i(r-1)}
\end{equation}
for $0\leq i\leq d/r$. We have
\begin{align}
\NH_d(t;\zeta_r)&=\sum_{i=0}^{d/r} \zeta_r^{-{ir \choose 2}} {d/r \choose i} (-1)^{i(r-1)} t^{ir}\\
&=\sum_{i=0}^{d/r} \zeta_{2r}^{-2{ir \choose 2}} {d/r \choose i} \zeta_{2r}^{r\cdot i(r-1)} t^{ir} \\
&=\sum_{i=0}^{d/r} \zeta_{2r}^{-ir(ir-1)+ir(r-1)} {d/r \choose i} t^{ir}\\
&=\sum_{i=0}^{d/r} \zeta_{2r}^{-i(i-1)r^2} {d/r \choose i} t^{ir}.
\end{align}

Since $i(i-1)$ is even, $\zeta_{2r}^{-i(i-1)r^2}=1$, so we have
\begin{equation}
\NH_d(t;\zeta_r)=\sum_{i=0}^{d/r} {d/r \choose i} t^{ir} = (1+t^r)^{d/r}.
\end{equation}
\end{proof}

We note a consequence that has potential generalizations outside the cusp case. Let $\NQ_d(t;q)=\NH_d(t^2;q)$ and $Q_d(t;q)=\NQ_d(t;q)/(t;q)_d$. Then $\NQ_d(t;q)$ clearly satisfies an analogous property $\NQ_d(t;\zeta_r)=\NQ_1(t^r;1)^{d/r}$ for $r|d$. Because the $q,t$-polynomial $(t;q)_d$ satisfies a similar property $(t;\zeta_r)_d=(1-t^r)^{d/r}=(t^r;1)_1^{d/r}$ when $r|d$, we have
\begin{equation}\label{eq:cyclic_sieving}
\begin{aligned}
H_d(t;\zeta_r)&=H_1(t^r;1)^{d/r};\\
Q_d(t;\zeta_r)&=Q_1(t^1;1)^{d/r}
\end{aligned}
\end{equation}
if $r$ divides $d$.

The analogous property holds for a smooth rational point on a curve, using $H_{d,\Fq[[T]]}(t)=Q_{d,\Fq[[T]]}(t)=1/(t;q)_d$ and the property for $(t;q)_d$ above. If the formula in Conjecture \ref{conj:cusp_exact_formula} is true, then the $r=1$ case of the proposition above would imply that if $k=\C$ and $X$ is a reduced curve with only cusp singularities, the Euler characteristics of $\Quot_{d,n}(X)$ would satisfy
\begin{equation}\label{eq:euler_char_cyclic_sieving}
\sum_{n=0}^\infty \chi(\Quot_{d,n}(X)) t^n = \parens*{\sum_{n=0}^\infty \chi(\Hilb_n(X)) t^n}^d.
\end{equation}

For a reduced curve $X$ with planar singularities, the right-hand side is known in terms of the HOMFLY polynomial of the associated algebraic link. Therefore, if \eqref{eq:euler_char_cyclic_sieving} is true for $X$, it would give a formula for the Euler characteristics of $\Quot_{d,n}(X)$. The truth of \eqref{eq:euler_char_cyclic_sieving} for the cusp singularity is unknown, and one may attempt it as another weak form of Conjecture \ref{conj:cusp_exact_formula}.

\

We conclude this section with a curious observation about factorizations of $\NH_d(t;q)$ over the rationals, verified for $d\leq 30$ and $t=\pm 1,\pm 2$. In the following statement, we say the constant polynomial $1$ is irreducible as well.

\begin{conjecture}
Define 
\begin{equation}
P_d(t;q):=\begin{cases}
\NH_d(t;q),&\text{ if $d$ is even;}\\
\NH_d(t;q)/(1+q^d t),&\text{ if $d$ is odd.}
\end{cases}
\end{equation}

Then
\begin{enumerate}
\item $P_d(t;q)$ is a polynomial in $\Z[t,q]$ with nonnegative coefficients, and is irreducible \tuparens{or $1$};
\item For any integer $m\neq -1$, the polynomial $P_d(m;q)$ is irreducible in $\Z[q]$;
\item The polynomial $P_d(-1;q)$ is of the form
\begin{equation}
P_d(-1;q)=I_d(q)\prod_{\substack{1\leq r\leq d\\r\text{ odd}}} \Phi_r(q)^{\floor{\frac{d+r-1}{2r}}},
\end{equation}
where $I_d(q)$ is an irreducible polynomial in $\Z[q]$ and $\Phi_r(q)$ is the cyclotomic polynomial of order $r$. 
\end{enumerate}
\end{conjecture}

In other words, when $m$ is an integer different from $-1$, the polynomial $\NH_d(m;q)$ in $q$ only has the factorization dictated by the generic factorization of $\NH_d(t;q)$. If $m=-1$, then $\NH_d(m;q)$ has lots of cyclotomic factors of odd degrees, but nothing else. 

\renewcommand{\thesection}{\Alph{section}}
\stepcounter{section}
\setcounter{section}{0}
\renewcommand\theHsection{A-\thesection}
\section{Appendix: Proofs in the Gr\"obner theory for monomial subrings}\label{appendix:groebner}
We prove the statements in \S \ref{sec:groebner} in the setting where $\Omega$ is a power series ring over a field $k$. Note that the lower monomials are leading, i.e., $<\,=\,\prec$. Recall that $R\subeq \Omega$ is a monomial subring, and $F=R u_1\oplus \dots \oplus R u_d$ is a free $R$-module. We have a monomial order $<$ on $F$. 

\begin{proposition}
[Proposition {\ref{prop:division_algorithm}}]
\label{a_prop:division_algorithm}
Let $f,g_1,\dots,g_h$ be elements of $F$. Then there is a \tuparens{not necessarily unique} expression 
\begin{equation}
f=r+\sum_{i=1}^h q_i g_i
\end{equation}
with $r\in F, q_i\in R$, such that
 \begin{itemize}
  \item No term of $r$ is divisible by $\LT(g_i)$ for any $i$.
  \item $\LT(q_i g_i)\succeq \LT(f)$ for any $i$, namely, no $q_i g_i$ contains terms before \tuparens{i.e., $\prec$} $\LT(f)$. 
 \end{itemize}
\label{a_prop:division_algorithm}
\end{proposition}

\begin{proof}
Assume $f,g_1,\dots,g_h\in F-\set{0}$ without loss of generality. If no term of $f$ is divisible by any $\LT(g_i)$, when we are done by taking $r=f, q_i=0$. Otherwise, we set $r=r^{(1)}=f, q_i=q_i^{(1)}=0$ to initialize, and we will modify the expression $f=r+\sum q_i g_i$ step by step. The idea is to choose a $g_j$ and kill terms of $r$ \emph{simultaneously} using $\LT(g_j)$, and we cycle around so that each $j$ will be chosen infinitely often. In step $l$, let $j=j_l$ be such that $j\equiv l \mod h$. Let $s$ denote the sum of all terms divisible by $\LT(g_j)$. We modify $q_i$ and $r$ according to
\begin{equation}\label{eq:proof_division_step}
\begin{aligned}
q_j^{(l+1)}&=q_i^{(l)} + \frac{s}{\LT(g_j)},\\
q_i^{(l+1)}&=q_i^{(l)},\quad i\neq j,\\
r^{(l+1)}&=r^{(l)} - \frac{s}{\LT(g_j)} g_j.
\end{aligned}
\end{equation}

The construction preserves the property $f=r+\sum_{i=1}^h q_i g_i$ and $\LT(q_i g_i)\succeq \LT(f)$. If $s\neq 0$, rewriting
\begin{equation}
r^{(l+1)}=r^{(l)} - \frac{s}{\LT(g_j)} \LT(g_j) - \frac{s}{\LT(g_j)} (g_j-\LT(g_j)),
\end{equation}
we may interpret the process to go from $r^{(l)}$ to $r^{(l+1)}$ as, first \emph{kill} the terms of $s$, and then \emph{reintroduce} some terms with $- \frac{s}{\LT(g_j)} (g_j-\LT(g_j))$. Some reintroduced monomials may be equal to monomials appearing in $s$, but we always have
\begin{equation}
\LT\parens*{ -\frac{s}{\LT(g_j)} (g_j-\LT(g_j)) } \succ \LT(s),
\end{equation}
by the defining property of a monomial order. Therefore, the term $\LT(s)$ is killed but not reintroduced.

We claim that $q_i^{(l)}$ and $r^{(l)}$ converge (in the usual power series topology in $\Omega$ or $\Omega^d$) as $l\to \infty$. If not, then there exists a monomial $\mu\in F$ that is reintroduced (i.e., appears in $-\frac{s}{\LT(g_j)} (g_j-\LT(g_j))$) infinitely many times. We may assume $\mu$ is the lowest with such property, by well ordering. Reintroduction of $\mu$ must result from killing a term involving a monomial $\nu$ of $r^{(l)}$ such that
\begin{equation}
\mu = \frac{\nu}{\LT(g_j)} \lambda,
\end{equation}
where $\lambda$ is a nonleading term of $\LT(g_j)$. In particular, $\nu$ precedes $\mu$, and $\nu$ equals $\frac{\mu}{\lambda}\LT(g_j)$. Let $u_{i_0}$ be the basis vector involved in $\mu$. Then $\nu$ divides $\frac{\mu}{u_{i_0}} \LT(g_j)$. We note that there are only finitely many monomials that divide a given monomial, so $\nu$ must belong to a finite list of monomials, namely, monomials dividing one of $\frac{\mu}{u_{i_0}} \LT(g_1),\dots,\frac{\mu}{u_{i_0}} \LT(g_h)$. In particular, there is a monomial $\nu\prec \mu$ that occurs in $r^{(l)}$ for infinitely many $l$'s. By construction, every time $\nu$ occurs, it will be killed immediately, so in order for $\nu$ to appear in $r^{(l)}$ infinitely many times, $\nu$ must be reintroduced infinitely many times. Since $\nu$ is lower than $\mu$, this contradicts the minimality of $\mu$.

Now, let the limits of $q_i^{(l)}$ and $r^{(l)}$ be $q_i^{(\infty)}$ and $r^{(\infty)}$ respectively. Since $R$ is closed in $\Omega$, we have $q_i^{(\infty)}\in R$ and $r^{(\infty)}\in F$. To show that $f=r^{(\infty)}+\sum_{i=1}^h q_i^{(\infty)} g_i$ is a desired division expression, it suffices to show that $r^{(\infty)}$ contains no term divisible by any $\LT(g_i)$. Suppose a monomial $\mu$ involved in $r^{(\infty)}$ is divisible by $\LT(g_j)$. In the above discussion, we see that $\mu$ is only reintroduced finitely many times. After the last time $\mu$ is reintroduced (if any), it will be killed when we deal with $g_j$ next time (which will happen since we cycle through $g_1,\dots,g_h$), if not earlier, and $\mu$ will be never introduced again. Hence $\mu$ does not appear in $r^{(\infty)}$, a contradiction.
\end{proof}

\begin{lemma}[Lemma {\ref{lem:unique_remainder}}]
Let $f,g_1,\dots,g_h\in F$. If $G=\set{g_1,\dots,g_h}$ is a Gr\"obner basis, then the remainder of $f$ in a division expression by $G$ is unique \tuparens{even though the division expressions may not be unique}. Moreover, the remainder of $f$ by $G$ is zero if and only if $f\in (g_1,\dots,g_h)$. 
\label{a_lem:unique_remainder}
\end{lemma}
\begin{proof}
Let $r$ and $r'$ be two possible remainders of $f$ divided by $G$. Then $r-r'$ lies in the submodule $M=(g_1,\dots,g_h)$. Moreover, since no term of $r$ or $r'$ is divisible by any of $\LT(g_i)$, no term of $r-r'$ is divisible by any of $\LT(g_i)$. If $r-r'$ is nonzero, then it has a leading term $\LT(r-r')$, which lies in $\LT(M)$. Since $g_1,\dots,g_h$ form a Gr\"obner basis, $\LT(M)=(\LT(g_1),\dots,\LT(g_h))$. As a property of monomial submodules, $\LT(r-r')$ must be divisible by one of $\LT(g_i)$, a contradiction. Hence $r-r'=0$, proving the uniqueness of the remainder.

In the second assertion, we only need to prove the ``if'' direction. Let $f\in M=(g_1,\dots,g_h)S$, and let $r$ be the remainder of $f$ divided by $G$. Since $f\in M$, we have $r\in M$, so we may repeat the same argument above, with $r-r'$ replaced by $r$, and conclude that $r=0$.
\end{proof}

\begin{lemma}\label{lem:generating_if_groebner}
Let $M\subeq F$ be a submodule, and $g_1,\dots,g_h\in M$. If $\LT(M)=(\LT(g_1),\dots,\LT(g_h))$, then we must have $M=(g_1,\dots,g_h)$, so that $\set{g_1,\dots,g_h}$ is a Gr\"obner basis for $M$.
\end{lemma}
\begin{proof}
Let $f\in M$, and we shall show that $f\in (g_1,\dots,g_h)$. Let
\begin{equation}
f=\sum_{i=1}^h q_i g_i + r.
\end{equation}
be a division expression of $f$ divided by $g_1,\dots,g_h$. It suffices to show that $r=0$. If not, then $r$ has an initial term $\LT(r)$. Since $r\in M$, we have $\LT(r)\in \LT(M)=(\LT(g_1),\dots,\LT(g_h))$ by the definition of $\LT(M)$. As a result, $\LT(r)$ is divisible by one of $\LT(g_i)$, a contradiction to the requirement of a division expression.
\end{proof}

\begin{proposition}[Proposition {\ref{prop:reduced_unique}}]
Every submodule $M$ of $F$ has a unique reduced Gr\"obner basis.
\label{a_prop:reduced_unique}
\end{proposition}

\begin{proof}
Let $C(\LT(M))=\set{\mu_1,\dots,\mu_h}$ be the set of corners of $\LT(M)$. 

We first prove the existence of a reduced Gr\"obner basis for $M$. By the definition of $\LT(M)$, there are $g_1,\dots,g_h\in M$ such that $\LT(g_i)=\mu_i$. By Lemma \ref{lem:generating_if_groebner}, $M$ is generated by $G:=\set{g_1,\dots,g_h}$ and $G$ is a Gr\"obner basis for $M$.

We now modify $G$. For each $i$, divide $\mu_i$ by $G$, and call the (unique) remainder $r_i$. Define
\begin{equation}\label{a_eq:reduction}
g_i':=\mu_i - r_i;
\end{equation}
note from the division expression that $g_i'$ lies in $M$.

We claim that $g_1',\dots,g_h'$ form a reduced Gr\"obner basis for $M$. Fix $i$. From the division algorithm, we observe that (1) no term of $r_i$ is divisible by any $\mu_j$, and (2) no term of $r_i$ is before $\mu_i$. In particular, $\mu_i$ is not involved in any term of $r_i$, so every term of $r_i$ is strictly behind $\mu_i$. Therefore, we have $\LT(g_i')=\mu_i$, so $g_1',\dots,g_h'$ form a Gr\"obner basis for $M$. It remains to check the requirement (c) of Definition \ref{def:groebner}, namely, non nonleading term of $g_i'$ is divisible by any $\mu_j$. But this is just the observation (1).

Finally, we prove the uniqueness of the reduced Gr\"obner basis. Say $g_1,\dots,g_h$ and $g_1',\dots,g_h'$ are two reduced Gr\"obner bases for $M$ such that $\LT(g_i)=\LT(g_i')=\mu_i$ for all $i$. Fix $i$, and we claim that $g_i'=g_i$. Consider $r=g_i-g_i'$. The leading terms of $g_i$ and $g_i'$ cancel each other, and the nonleading terms of $g_i$ and $g_i'$ are not divisible by any $\mu_j$ (by the requirement (c) of a reduced Gr\"obner basis), so no term of $r$ is divisible by any $\mu_j$. If $r\neq 0$, noting that $r$ is in $M$, we must have $\LT(r)\in \LT(M)$, so $r$ has a term $\LT(r)$ that is divisible by some $\mu_j$, a contradiction. Hence $r=0$ and the proof is complete.
\end{proof}

\begin{lemma}[Lemma {\ref{lem:dimension_same_as_leading}}]
Let $M$ be a submodule of $F$. Then there is an isomorphism of $k$-vector spaces
\begin{equation}
\frac{F}{M}\cong \frac{F}{\LT(M)}.
\end{equation}

In fact, the set of monomials outside $\LT(M)$ is a $k$-basis for $F/M$.
\label{a_lem:dimension_same_as_leading}
\end{lemma}

\begin{proof}
Let $G=\set{g_1,\dots,g_h}$ be a Gr\"obner basis for $M$, with $\LT(g_i)$ being monomials $\mu_i$. Let $\Delta$ be the set of monomials of $F$ outside $\LT(M)$, which is equal to the set of monomials not divisible by any $\mu_i$. Consider the map
\begin{equation}
\begin{aligned}
r_G: F &\to \Span_k \Delta, \\
f &\mapsto (f \bmod G),
\end{aligned}
\end{equation}
where $(f\bmod G)$ denotes the remainder of $f$ divided by $G$, which is unique by Lemma \ref{a_lem:unique_remainder}. We claim that $r_G$ is $k$-linear and surjective with kernel $M$, thus inducing a $k$-isomorphism
\begin{equation}
\frac{F}{M}\overset{r_G}{\cong} \Span_k \Delta = \frac{F}{\LT(M)}.
\end{equation}

By the second assertion of Lemma \ref{a_lem:unique_remainder}, the kernel of $r_G$ is $M$. To prove $r_G$ is surjective, let $f\in \Span_k \Delta$. Then the expression $f=f+\sum 0\cdot g_i$ is itself a division expression, so $r_G(f)=f$.

It remains to show that $r_G$ is $k$-linear. It is obvious that $r_G(cf)=c r_G(f)$ for any scalar $c\in k$. We shall show $r_G(x)+r_G(y)+r_G(z)=0$ whenever $x+y+z=0$. 

Without loss of generality, assume $x,y,z\neq 0$ and $\LT(z)\preceq \LT(x),\LT(y)$. Let
\begin{equation}
x=r_G(x)+\sum_{i=1}^h q^x_i g_i
\end{equation}
and
\begin{equation}
y=r_G(y)+\sum_{i=1}^h q^y_i g_i
\end{equation}
be division expressions of $x$ and $y$ with repsect to $G$. Since $\LT(q^x_i g_i)\succeq \LT(x) \succeq \LT(z)$ and $\LT(q^y_i g_i)\succeq \LT(y) \succeq \LT(z)$, the following expression
\begin{equation}
-z=r_G(x)+r_G(y)+\sum_{i=1}^h (q^x_i+q^y_i) g_i
\end{equation}
is a division expression of $z$ with respect to $G$. As a result, $r_G(-z)=r_G(x)+r_G(y)$, so $r_G$ is $k$-linear.
\end{proof}

\begin{lemma}[Lemma {\ref{lem:almost_reduced}}]
Let $G=\set{g_1,\dots,g_h}, G'=\set{g_1',\dots,g_h'}$ be Gr\"obner bases for $M$ with $\LT(g_i)=\LT(g_i')=\mu_i$, and assume $G'$ is reduced. If $g_1-\mu_1$ has no terms divisible by any $\mu_i$, then we must have $g_1=g_1'$. 
\label{a_lem:almost_reduced}
\end{lemma}

\begin{proof}
We use the algorithm \eqref{a_eq:reduction} to reduce $G$ to a reduced Gr\"obner basis. In particular, $g_1'=\mu_1-r_1$, where $r_1$ is the remainder of $\mu_1$ divided by $G$. However, if we let $r=\mu_1-g_1$, the expression
\begin{equation}
\mu_1 = 1\cdot g_1 + r
\end{equation}
is a division expression because $\LT(g_1)=\mu_1$ and $r=\mu_1-g_1$ has no terms divisible by any $\mu_i$ by our hypothesis. Hence $r_1=r$ and $g_1'=\mu_1-r=g_1$.
\end{proof}

\subsection{Buchberger criterion}
Let $M=(g_1,\dots,g_h)\subeq F$ be a submodule, and assume $\LT(g_i)=\mu_i$ is a monomial. Denote $G=\set{g_1,\dots,g_h}$. We shall prove Proposition \ref{prop:buchberger} that decides when is $G$ a Gr\"obner basis. We shall divide it into two independent steps, Lemma \ref{a_lem:standard_relation} and Lemma \ref{a_lem:buchberger_on_generating_relations}.

Recall that $\LCM(\mu_i,\mu_j)$ is the set of minimal common multiples of $\mu_i$ and $\mu_j$, and the statement of Proposition \ref{prop:buchberger} involves testing all relations of the form
\begin{equation}
S^{\nu}_{ij}(g_i, g_j) := \frac{\nu}{\mu_i} g_i - \frac{\nu}{\mu_j} g_j,
\end{equation}
for all $i,j\in [h]$ and $\mu\in \LCM(\mu_i,\mu_j)$. 

In fact, we will show that the only feature we need about the collection $S^{\nu}_{ij}$ is that they generate the ``module of relations'' among monomials $\mu_1,\dots,\mu_h$. We now make the notion precise.

\begin{definition}
For any finite index set $I$, let $W_I:=\bigoplus_{i\in I} R w_i$. We shall always view elements of $W_I$ as an operator: if $f_i, i\in I$ are elements of any $R$-module $N$, and $\Theta=\sum_{i\in I} a_i w_i$ is an element of $W_I$, then we write
\begin{equation}
\Theta(f_i: i\in I):=\sum_{i\in I} a_i f_i.
\end{equation}

In particular, we have $\Theta=\Theta(w_i: i\in I)$.

If $I=[h]=\set{1,\dots,h}$, then we may abbreviate $W_I$ as $W_h$. If $I\subeq J$, then we shall view $W_I$ as a submodule of $W_J$ via the natural inclusion sending $w_i$ to $w_i$. 

For monomials $\ul{\mu}=\set{\mu_i: i\in I}$ in $F$, we define the \textbf{module of relations} of $\ul{\mu}$ to be
\begin{equation}
\Rel(\ul{\mu})=\set*{\Theta\in W_I: \Theta(\ul{\mu})=0}.
\end{equation}
\end{definition}

We will need a grading on $W_I$ in the following proofs.

\begin{definition}
Let $\ul{\mu}=\set{\mu_i: i\in I}$ be monomials in $F$. Consider a grading on $W_I$ with degrees indexed by monomials $F$, defined by
\begin{equation}
\deg_{\ul{\mu}}(\tau w_i):=\tau \mu_i
\end{equation}
for any monomial $\tau\in R$ and index $i\in I$. We say $\deg_{\ul{\mu}}$ is the grading on $W_I$ induced by $\ul{\mu}$. 

We note that the $R$-linear map $W_I\to F$, $\Theta\mapsto \Theta(\ul{\mu})$ preserves the grading, where $F$ is equipped with the natural grading by monomials in $F$.

We use the notation $[\cdot]_\nu$ to extract the degree-$\nu$ part. This will applies to $W_I$ and $F$ graded by monomials in $F$, and $R$ graded by monomials in $R$.
\end{definition}

Technically speaking, the grading is not in the usual sense: given $\Theta$, there may be infinitely many $\nu$ such that $[\Theta]_\nu$ is nonzero, just like in the case of a power series ring. However, we do have that $\Theta=0$ if and only if $[\Theta]_\nu=0$ for all $\nu$. 

\begin{lemma}~
\begin{enumerate}
\item For monomials $\mu_1,\mu_2$ in $F$, the relation submodule $\Rel(\mu_1,\mu_2)$ is generated by
\begin{equation}
S^{\nu} := \frac{\nu}{\mu_1} w_1 - \frac{\nu}{\mu_2} w_2 \in W_{\set{1,2}}
\end{equation}
for all $\nu\in \LCM(\mu_1,\mu_2)$.

\item Let $\mu_1,\dots,\mu_h$ be monomials in $F$. For each pair $1\leq i<j\leq h$, let $\set*{\Theta_{ij}^\lambda: \lambda\in\Lambda_{ij} }$ be a generating set for $\Rel(\mu_i,\mu_j)$ in $W_{\set{i,j}}$. Then the set 
\begin{equation}
\set*{\Theta_{ij}^\lambda: i<j, \lambda\in \Lambda_{ij}}
\end{equation}
generates $\Rel(\mu_1,\dots,\mu_h)$ in $W_h$, where we view $\Theta_{ij}^\lambda$ as an element of $W_h$ via the natural inclusion $W_{\set{i,j}}\subeq W_h$. 
\end{enumerate}

In consequence, elements of the form
\begin{equation}
S_{ij}^{\nu} := \frac{\nu}{\mu_i} w_i - \frac{\nu}{\mu_j} w_j
\end{equation}
for all $i,j\in [h], \nu\in \LCM(\mu_i,\mu_j)$ generate the module of relations $\Rel(\mu_1,\dots,\mu_h)$ in $W_h$. 
\label{a_lem:standard_relation}
\end{lemma}

\begin{proof}~
\begin{enumerate}
\item Let $\LCM(\mu_1,\mu_2)=\set{\nu^1,\dots,\nu^l}$, and consider an element $\Theta=a_1 w_1 + a_2 w_2$ of $\Rel(\mu_1,\mu_2)$, where $a_i\in R$. Thus $a_1 \mu_1+a_2 \mu_2=0$ by definition. Use the grading on $W_{\set{1,2}}$ induced by $\mu_1,\mu_2$. Let $\nu$ be a monomial in $F$, then it suffices to show that $[\Theta]_\nu$ is generated by $S^{\nu^\lambda}$, $1\leq \lambda\leq l$. There are two cases:
 \begin{enumerate}
 \item The monomial $\nu$ is divisible by both $\mu_1$ and $\mu_2$.
 
 For $i=1,2$, write $\tau_i=\nu/\mu_i\in R$. Extracting the degree-$\nu$ part of both sides of $a_1 \mu_1+a_2\mu_2=0$, we have
 \begin{equation}
  [a_1]_{\tau_1} \mu_1 + [a_2]_{\tau_2} \mu_2 = 0.
 \end{equation} 
 
 Write $[a_i]_{\tau_i}=c_i \tau_i$ for $i=1,2$, where $c_i\in k$, then we have
 \begin{equation}
 0=c_1 \tau_1 \mu_1 + c_2 \tau_2 \mu_2 = (c_1+c_2) \mu,
 \end{equation}
 and hence $c_2=-c_1$. Therefore,
 \begin{equation}
 [\Theta]_\nu = c_1\tau_1 w_1 + c_2\tau_2 w_2 = c_1 (\frac{\nu}{\mu_1} w_1 - \frac{\nu}{\mu_2} w_2).
 \end{equation}
 
 By the definition of $\LCM(\mu_1,\mu_2)$, there exists $\lambda$ such that $\nu^\lambda$ divides $\nu$. Thus
 \begin{equation}
 [\Theta]_\nu = c_1 \frac{\nu}{\nu^\lambda} \parens*{\frac{\nu^\lambda}{\mu_1} w_1 - \frac{\nu^\lambda}{\mu_2} w_2} = a_1 \frac{\nu}{\nu^\lambda} S^{\nu^\lambda},
 \end{equation}
 so one $S^{\nu^\lambda}$ is enough to generate $[\Theta]_\nu$. 
 
 \item The monomial $\nu$ is not divisible by at least one of $\mu_1$ or $\mu_2$. 
 
 In this case, $[\Theta]_\nu$ has at most one term, so no cancellation can occur in the degree-$\nu$ part of $\Theta(\mu_1,\mu_2)$. As a result, in order to have $[\Theta(\mu_1,\mu_2)]_\nu=0$, we must have $[\Theta]_\nu=0$.
 \end{enumerate}
 
\item Use the grading on $W_h$ induced by $\mu_1,\dots,\mu_h$. Let $\Theta$ be an element of $\Rel(\mu_1,\dots,\mu_h)$, and let $\nu$ be a monomial in $F$. We claim that $[\Theta]_\nu$ is generated by $\Theta_{ij}^\lambda$. Consider the degree-$\nu$ part of the equation $\Theta(\mu_1,\dots,\mu_h)=0$. Let $A$ be the set of $i$ such that $\mu_i$ divides $\nu$. If $\abs{A}\leq 1$, then no cancellation occurs in $[\Theta(\mu_1,\dots,\mu_h)]_\nu$, so we must have $[\Theta]_\nu=0$. From now on, assume without loss of generality that $A=\set{1,\dots,l}$ with $l\geq 2$.

Write $\tau_i=\nu/\mu_i\in R$, then $[\Theta]_\nu$ is of the form
\begin{equation}
[\Theta]_\nu = c_1 \tau_1 w_1 + \dots + c_l \tau_l w_l,
\end{equation}
where $c_i\in k$. The equation $[\Theta(\mu_1,\dots,\mu_h)]_\nu=0$ then implies
\begin{equation}
c_1 \tau_1 \mu_1 + \dots + c_l \tau_l \mu_l = (c_1+\dots+c_l) \tau = 0,
\end{equation}
and thus $c_1+\dots+c_l=0$. It follows that
\begin{align}
[\Theta]_\nu &= -(c_2+\dots+c_l)\tau_1 w_1 + c_2 \tau_2 w_2 + \dots + c_l \tau_l w_l \\
&= c_2 (\tau_2 w_2 - \tau_1 w_1) + \dots + c_l (\tau_l w_l - \tau_1 w_1). 
\end{align}

For $2\leq i\leq l$, we have $\tau_i w_i - \tau_1 w_1 \in \Rel(\mu_1,\mu_i)$, since its evaluation at the pair $(\mu_1,\mu_i)$ is $\tau_i \mu_i - \tau_1 \mu_1 = \mu - \mu = 0$. Because the relations $\Theta_{1i}^\lambda\;(\lambda\in \Lambda_{1i})$ generate $\Rel(\mu_1,\mu_i)$, we conclude that $[\Theta]_\nu$ is generated by the elements $\Theta_{ij}^\lambda$ in $W_h$.
\end{enumerate}
\end{proof}

We now use the above property for the collection $S_{ij}^\nu$ to prove the Buchberger criterion. This is the only step that uses Assumption \ref{assumption} that the monomial order on $F$ is of order type $\omega$. 

\begin{lemma}\label{a_lem:buchberger_on_generating_relations}
Assume Assumption \ref{assumption}. Let $G=\set{g_1,\dots,g_h}$ with $g_i\in F$ and $\LT(g_i)=\mu_i$. Let $\Theta^\lambda$, $\lambda\in \Lambda$ be homogeneous elements of $W_h$ \tuparens*{with respect to the grading induced by $\mu_1,\dots,\mu_h$} that generate $\Rel(\mu_1,\dots,\mu_h)$. If for each $\lambda\in \Lambda$, a possible remainder of $\Theta^\lambda(g_1,\dots,g_h)$ divided by $G$ is zero, then $G$ is a Gr\"obner basis.
\end{lemma}
\begin{proof}
Use the grading of $W_h$ induced by $\mu_1,\dots,\mu_h$. In this proof, we need to consider the notion of \textbf{leading degree} for any nonzero element $\Theta$ of $W_h$:
\begin{equation}
\deg(\Theta)=\mathrm{min}_{\prec} \set{\nu: [\Theta]_\nu\neq 0},
\end{equation}
the lowest monomial $\nu$ in $F$ such that $\Theta$ has a degree-$\nu$ component. Note that $\deg(\Theta)$ exists uniquely by well ordering. 

Suppose $G$ is not Gr\"obner, then there exists an element $f$ of $M=(g_1,\dots,g_h)$ such that $\LT(f)$ is not divisible by any of $\mu_i$. We shall fix $f$ and derive a contradiction from here.

Since $f\in M$, there exists $P\in W_h$ such that $f = P(g_1,\dots,g_h)$. Let $\nu=\deg P$. We note that $f$ is $[P]_\nu (\ul{\mu})$ (which is of the form $c\nu, c\in k$) plus terms higher than $\nu$. Therefore, $\LM(f)\succeq \nu$, and the inequality is strict if and only if $[P]_\nu (\ul{\mu})=0$.  On the other hand, we note that $\LM(f)$ cannot be $\nu$, because $\nu$ must be divisible by one of $\mu_i$ by the definition of the grading induced by $\ul{\mu}$, while $\LT(f)$ is not. Hence $\LM(f)\succ \nu$, so that $[P]_\nu (\ul{\mu})=0$. 

We have concluded that any $P\in W_h$ such that $f = P(g_1,\dots,g_h)$ must satisfy $\deg P \prec \LM(f)$. By Assumption \ref{assumption} that $\prec$ is of order type $\omega$, there are only finitely many monomials below $\LM(f)$. Hence, we may assume $P$ is chosen such that $\nu=\deg P$ is maximized. 

We claim that there is $P'\in W_h$ such that $f = P'(g_1,\dots,g_h)$ and $\deg P' \succ \nu$, which yields a contradiction. To prove the claim, recall that $[P]_\nu (\ul{\mu})=0$, so $[P]_\nu$ lies in the module of relations $\Rel(\mu_1,\dots,\mu_h)$. By our hypothesis, there exists elements $a^\lambda\in R$ such that
\begin{equation}\label{a_eq:proof_buchberger_1}
[P]_\nu = \sum_{\lambda\in \Lambda} a^\lambda \Theta^\lambda.
\end{equation}

Since $\Theta^\lambda$ is homogeneous (say, of degree $\nu^\lambda$), we may assume $a^\lambda$ is homogeneous of degree $\nu/\nu^\lambda$ by taking the degree-$\nu$ part of the right-hand side of \eqref{a_eq:proof_buchberger_1}. In other words, we may assume $a^\lambda$ is of the form $c^\lambda \nu/\nu^\lambda$, where $c^\lambda\in k$. (Here, we have assumed without loss of generality that $\nu^\lambda$ divides $\nu$ for every $\lambda$ in the sum.)

By our hypothesis and the defining requirement of a division expression, for each $\lambda$, there exists $Q^\lambda\in W_h$ such that
\begin{equation}
Q^\lambda(g_1,\dots,g_h)=\Theta^\lambda(g_1,\dots,g_h)
\end{equation}
and $\deg Q^\lambda \succeq \LM(\Theta^\lambda(g_1,\dots,g_h))$. The homogeneity of $\Theta^\lambda$ of degree $\nu^\lambda$ implies that 
\begin{equation}
\Theta^\lambda(g_1,\dots,g_h)=\Theta^\lambda(\ul{\mu}) + (\text{terms strictly higher than $\nu^\lambda$}).
\end{equation}

Since $\Theta^\lambda \in \Rel(\ul{\mu})$, we have $\Theta^\lambda(\ul{\mu})=0$, so $\LM(\Theta^\lambda(g_1,\dots,g_h))\succ \nu^\lambda$. Therefore, we reach an important conclusion
\begin{equation}
\deg Q^\lambda \succ \nu^\lambda = \deg \Theta^\lambda.
\end{equation}

This allows us to replace the ``leading terms'' of $P$ by strictly higher terms, which is what we want. To make it precise, let
\begin{equation}
P'=P+\sum_{\lambda\in \Lambda} a^\lambda \parens*{Q^\lambda - \Theta^\lambda}.
\end{equation}

Then we have $P'(g_1,\dots,g_h)=P(g_1,\dots,g_h)$. Moreover, the terms $-\sum_{\lambda\in \Lambda} a^\lambda \Theta^\lambda$ precisely cancel the degree-$\nu$ part of $P$, while the terms $+\sum_{\lambda\in \Lambda} a^\lambda Q^\lambda$ only involve degrees higher than $\nu$, because
\begin{equation}
\deg(a^\lambda Q^\lambda) \succ \deg(a^\lambda \nu^\lambda) = \nu.
\end{equation}

It follows that $\deg P' \succ \deg P$, so $P'$ is constructed as desired in the claim. We thus finish the proof of Lemma \ref{a_lem:buchberger_on_generating_relations}.
\end{proof}

Combining Lemma \ref{a_lem:standard_relation} and Lemma \ref{a_lem:buchberger_on_generating_relations}, we complete the proof of the ``if'' part of the Buchberger criterion (Proposition \ref{prop:buchberger}). The ``only if'' part is in fact elementary: if $G=\set{g_1,\dots,g_h}$ is a Gr\"obner basis for $M$, then $S^\nu_{ij}(g_i,g_j)$ is in $M$, so its remainder in every division expression by $G$ is zero by Lemma \ref{a_lem:unique_remainder}. Hence the proof of every assertion in \S \ref{sec:groebner} is complete.

\section{Appendix: motivic considerations}\label{appendix:motivic}
We generalize some of our results to the Grothendieck ring of motives for algebraic stacks (for an introduction, see \cite{Ekedahl2009}). To achieve this, we need geometric considerations of moduli spaces. However, no much heavy machinery of algebraic geometry is needed, since to work in the Grothendieck ring of motives, we only need to care about reduced structures. 

This section contains two main topics. The first is a proof of Theorem~\ref{Thm:apdxB1} (or Theorem~\ref{Thm:Main2}), which occupies \S\ref{subsec:functordesp}$\sim$ \S\ref{sub:motiform}. As noted in Remark~\ref{rmk:whyweneedappendixB}, the proof of Proposition~\ref{prop:Fibercounting}
does not naïvely generalize to the motivic setting. In order to prove Theorem~\ref{Thm:apdxB1}, we will need to develop new machinaries. The second is an exposition of Gröbner strata, which we present in \S\ref{subsec:Gbstrata}. Abundant results on Gröbner strata already exist in literature, for example, see \cite{LedererGrobnerstrata}. The results in previous chapters are easily generalizable to the geometric setting. Furthermore,  Lemma~\ref{lem:decomposition} and Lemma~\ref{lem:staircase} can be easily upgraded to their motivic versions. So we actually get motivic decompositions of $[\Quot_{d,n}^d(R)]$ with respect to Gröbner strata. We refer the readers to \S\ref{subsec:grobsubfunctor} for the precise statements. 

\begin{theorem}[A slight modification of Theorem~\ref{Thm:Main2}]\label{Thm:apdxB1}
Let $k$ be an arbitrary field and $R$ be a complete local ring of form $k[[x_1,...,x_m]]/I$. For any $n,d,r$ such that $0\leq r\leq \min\{n,d\}$, the following identity holds in $\KStck{k}$:
\begin{equation}\label{eq:apdxB1}
[\Coh_n^r(R)] = \frac{[\Quot_{d,n}^r(R)][\GL_{d-r}]}{\mathbb{L}^{d(n-d)+(d-r)^2}[\GL_d]}.
\end{equation}
\end{theorem}
The formula (\ref{eq:apdxB1}) recovers (\ref{eq:coh_in_quot_mot}) by setting $R=\widehat{\mathcal{O}}_{X,p}$, and substituting in 
\begin{equation}
    [\Gr(d,r)]=[\GL_d]/\mathbb{L}^{r(d-r)}[\GL_r][\GL_{d-r}].
\end{equation} In practice, we will only use the most important special case $r=d$. 

\subsubsection{Notation and assumptions}\label{subsub:naa} In order to keep notation clean, we will again use $\Coh_n(R)$, $\Coh_n^r(R)$, $\Quot_{d,n}(R)$, $\Quot_{d,n}^r(R)$, $C_n(R)$, and $C_n^r(R)$ to denote the induced reduced structures of the original schemes and stacks, see \S\ref{subsec:functordesp}. In the whole chapter, the symbol $S$ is reserved for denoting a reduced affine $k$-scheme, which will only be used in the situation where we consider affine scheme points of various schemes and stacks. 
\subsubsection{Outline of the proof for Theorem~\ref{Thm:apdxB1}} In \S\ref{subsub:stab}, we will first observe a group action $\GL[F]\curvearrowright\Quot_{d,n}^r(R)$ and introduce an important group scheme $\mathscr{G}$ over $\Quot_{d,n}^r(R)$, which is the ``stablizer bundle" of the $\GL[F]$-action. In \S\ref{sub:QC}, we will construct a Zariski $\GL_n$-torsor  ${\mathfrak{Q}}$ over $\Quot_{d,n}^r(R)$ that sits in the following diagram 
\begin{equation}\label{eq:torsorQ}
    \begin{tikzcd}
                   & {\mathfrak{Q}} \arrow[ld, "\pi"', two heads] \arrow[rd, "\varpi", two heads] &        \\
{\Quot_{d,n}^r(R)} &                                                                                        & C_n^r(R)
\end{tikzcd}
\end{equation}
We will show that $\varpi$ admits Zariski local sections, i.e., we have a Zariski cover $U\rightarrow C_n^r(R)$, and $\sigma:U\rightarrow \mathfrak{Q}$ such that $\varpi\sigma=\text{Id}_{U}$. We then 
show that $\GL[F]_U$ is an fppf $\mathscr{G}_U$-torsor over $\mathfrak{Q}_U$, where $\mathscr{G}_U=\sigma^*\pi^*\mathscr{G}$ is a group scheme over $U$. We then show in Lemma~\ref{lm:ZarstatiU2} that $U$ admits a Zariski stratification over which $\mathscr{G}_U$ is special (in the sense that any fppf torsor is Zariski locally trivial). We then use explicit descriptions of $\GL[F]$ and $\mathscr{G}$ from Lemma~\ref{lm:theGroup} and Lemma~\ref{lm:stblzer} to establish the formula (\ref{eq:apdxB1}).

\subsection{Functor descriptions of several geometric objects}\label{subsec:functordesp} For later reference, we summarize the functor of points for various geometric objects that we are considering. We remind the reader that $C_n(R)$ is the commuting variety (\S\ref{subsec:uramifiedmodulispcae}). Fix $V_n$ to be a vector space of dimension $n$ over $k$. In the following, let $X$ be a reduced $k$-scheme. 
\begin{equation}\label{eq:cohn}
    \Coh_{n}(R)(X)=\left\{\begin{aligned}&\text{groupoid of coherent sheaves }\mathcal{F}\text{ over } X_R,\\  &\text{such that } \mathcal{F} \text{ is flat of rank }n \text{ over } X.  \end{aligned}\right\}.
\end{equation}
\begin{equation}
    \Coh_{n}^r(R)(X)=\left\{\begin{aligned}&\text{groupoid of coherent sheaves }\mathcal{F}\text{ over }X_R,\\  &\text{such that } \mathcal{F} \text{ is flat of rank }n \text{ over } X \text{, and}\\
  &  \mathcal{F}/\mathfrak{m}\mathcal{F}\text{ is flat of rank }r \text{ over } X. \end{aligned}\right\}.
\end{equation}
\begin{equation}
    \Quot_{d,n}(R)(X)=\left\{ N\subseteq \mathcal{O}_{X_R}^d, 
      \text{ such that }
  \mathcal{O}_{X_R}^d/N\in \Coh_n(R)(X).  \right\}.
\end{equation}
\begin{equation}
    \Quot_{d,n}^r(R)(X)=\left\{N\subseteq \mathcal{O}_{X_R}^d, 
      \text{ such that }
  \mathcal{O}_{X_R}^d/N\in \Coh_n^r(R)(X).
      \right\}.
\end{equation}
\begin{equation}\label{eq:funcCn0}
   C_{n}(R)(X)=\left\{\begin{aligned}
   &\text{equivalent classes of pairs }(\mathcal{F},\iota)\text{ where }
    \mathcal{F}\in \Coh_n(R)(X),\\ 
    &\text{ and }\iota: \mathcal{O}_X\otimes  V_{n}\simeq \mathcal{F} \text{ is an isomorphism of } \mathcal{O}_X\text{-modules}.
\end{aligned}\right\}.
\end{equation}
\begin{equation}\label{eq:funcCn}
   C_{n}^r(R)(X)=\left\{\begin{aligned}
    &\text{equivalent classes of pairs }(\mathcal{F},\iota)\text{ where }
    \mathcal{F}\in \Coh_n^r(R)(X),\\ 
    &\text{ and }\iota: \mathcal{O}_X\otimes  V_{n}\simeq \mathcal{F} \text{ is an isomorphism of } \mathcal{O}_X\text{-modules}.
\end{aligned}\right\}.
\end{equation}
Note that 
\begin{equation}
C_n^r(R)=C_n(R)\times_{\Coh_n(R)}\Coh_n^r(R),
\end{equation}
and $\Coh_n^r(R)= [C_n^r(R)/\GL_n]$. 
Furthermore, when $n$ is fixed, there is a large Artin local quotient $R'$ of $R$ such that for all $d$ and $r$, $\Coh_{n}(R)=\Coh_{n}(R')$, $\Coh_{n}^r(R)=\Coh_{n}^r(R')$, $\Quot_{d,n}(R)=\Quot_{d,n}(R')$,
$\Quot_{d,n}(R)=\Quot_{d,n}(R')$, $\Quot_{d,n}^r(R)=\Quot_{d,n}^r(R')$, $C_n(R)=C_n(R')$ and $C_n^r(R)=C_n^r(R')$.

\subsection{Geometry of $\Quot_{d,n}^r(R)$}\label{sec:geoQQ} 
By the last few sentences of \S\ref{subsec:functordesp}, we may and will replace $R$ by a large Artin quotient. Write $\dim_k R=b+1$ and $F={R}^d=\bigoplus_{i=1}^d Ru_i$. If a submodule $N\subseteq F\otimes S$ lies in $\Quot_{d,n}^r(R)(S)$, then the quotient $(F\otimes S)/N$ is a flat $S$-module of rank 
$n$, and $(\mathfrak{m}F\otimes S)\cap N$ is a flat $S$-module of rank $bd-n+r$, so we obtain a locally closed immersion
\begin{equation}\label{eq:embQGr}  \Quot_{d,n}^r(R)\hookrightarrow \Gr(F, bd+d-n)
\end{equation}
sending a submodule $N\in F\otimes S$ to the corresponding module in $F\otimes S$. Indeed, $\Quot_{d,n}^r(R)$ admits a closed embedding into the locally closed subvariety of $ \Gr(F, bd+d-n)$ parametrizing subspaces whose intersections with $\mathfrak{m}F$ have dimension exactly $bd-n+r$. Note that when $d=r$, (\ref{eq:embQGr}) is a closed embedding. Let \begin{equation}\label{eq:Tatu}
0\rightarrow \mathscr{N}\rightarrow \underline{F} \rightarrow \mathscr{Q}\rightarrow 0
\end{equation}
be the restriction of the tautological sequence of universal bundles over $\Gr(F, bd+d-n)$ to $\Quot_{d,n}^r(R)$, i.e., over a point $x\in \Quot_{d,n}^r(R)(S)$ corresponding to an $N\subseteq F\otimes S$, the fiber $\mathscr{N}_x$ is  $N$, while the fiber of $\mathscr{Q}_x$ is $(F\otimes S)/N$. Furthermore, (\ref{eq:Tatu}) is not only a sequence of $\Quot_{d,n}^r(R)$-bundles, but also a sequence of coherent $\Quot_{d,n}^r(R)\times \Spec R$ modules. There is another related construction:
\begin{equation}\label{eq:fibgrass}  \Quot_{d,n}^r(R)\rightarrow \Quot_{d,r}^r(R)=\Gr(F/\mathfrak{m}F, d-r)
\end{equation}
sending $N\subseteq F\otimes S$ to $N/(\mathfrak{m}F\otimes S)$. It is not hard to see that this morphism is surjective with pointwise isomorphic fibers, but we will not use this. The pullback of the tautological sequence of universal bundles over $\Gr(F/\mathfrak{m}F, d-r)$ to $\Quot_{d,n}^r(R)$ will be denoted by
\begin{equation}\label{eq:Tatu2}
0\rightarrow \overline{\mathscr{N}}\rightarrow \underline{F}/\mathfrak{m}\underline{F} \rightarrow \overline{\mathscr{Q}}\rightarrow 0
\end{equation}
It is clear that $\overline{\mathscr{N}}=\mathscr{N}/\mathfrak{m}\underline{F}$ and $\overline{\mathscr{Q}}=\mathscr{Q}/\mathfrak{m}\mathscr{Q}$.

\subsubsection{Group action}\label{subsub:gpact} The most important feature of $\Quot_{d,n}^r(R)$ is that it is equipped with an algebraic group action. For a finite $R$-module $M$, let $\GL[M]$ be the group $\Aut_R(M)$, but viewed as a linear algebraic group over $k$. Then $\GL[F]$ acts on $\Quot_{d,n}^r(R)$, in the sense that any $g\in \GL[F](S)$ acts on $\Quot_{d,n}^r(R)(S)$ by sending a submodule $N\subseteq F\otimes S$ to $g\cdot N$. 

There is a projection $\GL[F]\rightarrow \GL[F/\mathfrak{m}F]$ sending $g$ to $(g\bmod \m)$. Let $U(\mathfrak{m}F)$ be its kernel. The canonical action of $\GL[F/\mathfrak{m}F]$ on $\Gr(F/\mathfrak{m}F, d-r)$ is compatible with the action of $\GL[F]$ on $\Quot_{d,n}^r(R)$ via (\ref{eq:fibgrass}). The structure of $\GL[F]$ is easy to understand:
\begin{lemma}\label{lm:theGroup} Notation as above, the exact sequence
\begin{equation}\label{eq:uGL}
    1\rightarrow U[\mathfrak{m}F]\rightarrow \GL[F]\rightarrow   \GL[F/\mathfrak{m}F]\rightarrow 1,
\end{equation}
admits a canonical splitting. Furthermore, $\GL[F/\mathfrak{m}F]\simeq \GL_d$ and $U[\mathfrak{m}F]$ is a $k$-split unipotent group of dimension $bd^2$ \tuparens{a unipotent group is $k$-split if it is a successive extension of $\mathbb{G}_a$ over $k$}.
\end{lemma}
\proof It is clear that $\GL[F/\mathfrak{m}F]\simeq \GL_d$. The morphism $\pi: \GL[F]\rightarrow \GL[F/\mathfrak{m}F]$ has a section $\GL[F/\mathfrak{m}F]\hookrightarrow \GL[F]$ induced by the structural morphism $k\hookrightarrow R$. Note that $U[\mathfrak{m}F](S)$ consists of $g$ of form 
\begin{equation}
    g(u_i)\in u_i+\mathfrak{m}F\otimes S,\,\,1\leq i\leq d.
\end{equation}
Therefore, $U[\mathfrak{m}F]$ is abstractly isomorphic to the affine space $\mathbb{A}^{bd^2}$. It is unipotent since for any $g\in U[\mathfrak{m}F](S)$, $(g-\text{Id})^b=0$. The fact that $U[\mathfrak{m}F]$ is $k$-split follows by induction on $b$ (the $k$-dimension of $R$): the assertion is trivial for $b=1$, otherwise there is an element $0\neq v\in R$ such that $v\mathfrak{m}=0$. The subgroup $U[vF]\subseteq U[\mathfrak{m}F]$ whose $S$-points are elements $g$ with $g(u_i)\in u_i+vF\otimes S$ lies in the center of $U[\mathfrak{m}F]$, and is isomorphic to $\mathbb{G}_a^{d^2}$. Let $F'=(R/v)^d$. We have an exact sequence \begin{equation}
    1\rightarrow U[vF]\rightarrow U[\mathfrak{m}F]\rightarrow U[\mathfrak{m}F']\rightarrow 1.
\end{equation} 
By the induction hypothesis, $U[\mathfrak{m}F']$ is $k$-split. It follows that $U[\mathfrak{m}F]$ is $k$-split. $\hfill\square$

\subsubsection{Group schemes over $\Quot_{d,n}^r(R)$}\label{subsub:stab} The constructions in \S\ref{subsub:gpact} can be globalized. They are even more important group schemes over $\Quot_{d,n}^r(R)$. Define the following:\begin{itemize}
    \item $\GL[\underline{F}]:=\Aut_{R}(\underline{F})$, $\GL[\underline{F}/\mathfrak{m}\underline{F}]:=\Aut_{R}(\underline{F}/\mathfrak{m}\underline{F})$ and $\GL[\overline{\mathscr{N}}]:=\Aut_{R}(\overline{\mathscr{N}})$.
    \item $U(\mathfrak{m}\underline{F}):=\ker(\GL[\underline{F}]\rightarrow \GL[\underline{F}/\mathfrak{m}\underline{F}])$. 
    \item $\mathscr{G}$ is the subgroup scheme of $\GL[\underline{F}]$ stabilizing $\mathscr{N}$ and reducing to identity on $\mathscr{Q}$. 
    \item $\Fil_1{\mathscr{G}}:=\ker(\mathscr{G}\rightarrow \GL[\underline{F}/\mathfrak{m}\underline{F}])=U(\mathfrak{m}\underline{F})\cap \mathscr{G}$, i.e., $\Fil_1{\mathscr{G}}$ is the subgroup scheme of  $\GL[\underline{F}]$ that reduces to identity on the whole $\underline{F}/\mathfrak{m}\underline{F}$. 
 \item $\Fil_2\mathscr{G}$ is the subgroup scheme of  $\GL[\underline{F}]$ that reduces to identity on both $\overline{\mathscr{N}}$ and  $\overline{\mathscr{Q}}$. 
\end{itemize}

Note that $\GL[\underline{F}]$ \textit{resp}. $\GL[\underline{F}/\mathfrak{m}\underline{F}]$ \textit{resp}. $U(\mathfrak{m}\underline{F})$ is just a trivial $\GL[F]$ \textit{resp}.  $\GL[F/\mathfrak{m}F]$ \textit{resp}. $U[\mathfrak{m}F]$ bundle over $\Quot_{d,n}^r(R)$, and the split exact sequence (\ref{eq:uGL}) admits an obvious global analogue. We will call $\mathscr{G}$ the \textit{stabilizer bundle}, in the sense that the fiber of $\mathscr{G}$ at an $S$-point is the subgroup of $\GL[F](S)$ that acts trivially on the corresponding module. It will play an important role in studying the action of $\GL[F]$ on $\Quot_{d,n}^r(R)$.
The group scheme $\mathscr{G}$
admits a three step filtration of normal subgroups:
\begin{equation}
    \underline{1} \triangleleft \Fil_1\mathscr{G} \triangleleft \Fil_2\mathscr{G}\triangleleft \Fil_3\mathscr{G}= \mathscr{G}.
\end{equation}
Note that in the special case $d=r$, we have $\mathscr{G}=\Fil_2\mathscr{G}$. 
\begin{lemma}\label{lm:stblzer} The followings are true: 
\begin{enumerate}
    \item $\Fil_1\mathscr{G}$ is a unipotent group scheme of rank $d(bd-n+r)$.
     \item\label{it:stblzer2} 
     $\gr_2\mathscr{G}$ is a subgroup scheme of $\GL[\underline{F}/\mathfrak{m}\underline{F}]$ which is
      Zariski locally isomorphic to $\mathbb{G}_a^{r(d-r)}$.\footnote{By this we mean that there is a Zariksi cover $U$ of $\Quot_{d,n}^r(R)$ such that $\gr_2\mathscr{G}_U\simeq \mathbb{G}_{a,U}^{r(d-r)}$. We adopt the same convention for (d).}
     \item $\Fil_2\mathscr{G}$ is a unipotent group scheme of rank $bd^2-d(n-d)-(d-r)^2$.
      \item $\gr_3\mathscr{G}\simeq\GL[\overline{\mathcal{N}}]$. In particular, it is Zariski locally isomorphic to $\GL_{d-r}$.
\end{enumerate}
\end{lemma}
\proof~\begin{enumerate}
    \item We first describe the functor of points of $\Fil_1\mathscr{G}$. Let $x\in \Quot_{d,n}^r(R)(S)$ corresponds to a submodule $N\subseteq F\otimes S$. Then the fiber of $\Fil_1\mathscr{G}$ over $x$ is 
\begin{equation}\label{eq:uN}
\Fil_1\mathscr{G}|_x=\{g\in \GL[F](S):\,\,g(u_i)\in u_i+N\cap (\mathfrak{m}F\otimes S), 1\leq i\leq d\}.
\end{equation}
This already implies that $\Fil_1\mathscr{G}$ is unipotent. Furthermore, it implies that there is a vector bundle structure  underlying $\Fil_1\mathscr{G}$, which is isomorphic to $ (\mathscr{N}\cap \mathfrak{m}\underline{F})^{\oplus d}$. The assertion follows from the fact that $\rk(\mathscr{N}\cap \mathfrak{m}\underline{F})=bd-n+r$.
\item Let $\mathscr{H}_2$ be the subgroup scheme of $\GL[\underline{F}/\mathfrak{m}\underline{F}]$ fixing the filtration (\ref{eq:Tatu2}) and reducing to identity on both $\overline{\mathscr{Q}}$ and $\overline{\mathscr{N}}$. Then $\gr_2\mathscr{G}\subseteq \mathscr{H}_2$. The canonical splitting of (\ref{eq:uGL}) gives rise to a morphism \begin{equation}
    \mathscr{H}_2\hookrightarrow \GL[\underline{F}/\mathfrak{m}\underline{F}]\rightarrow \GL[\underline{F}]
\end{equation}
whose image lies in $\Fil_2\mathscr{G}$. This implies that $\gr_2\mathscr{G}=\mathscr{H}_2$. Let $\Spec S$ be a Zariksi affine open of $\Quot_{d,n}^r(R)$ over which $\overline{\mathscr{Q}}$ and $\overline{\mathscr{N}}$ are trivialized. We can then pick a section of (\ref{eq:Tatu2}) and write $(\underline{F}/\mathfrak{m}\underline{F})_S=\overline{\mathscr{N}}_S\oplus \overline{\mathscr{Q}}_S$. Then the $S$-points of $\mathscr{H}_2|_S$ can be written as matrices 
\begin{equation}\label{eq:matH2}
    \begin{bmatrix}
        \text{Id}_{\overline{\mathscr{N}}_S} & B\\
        0 & \text{Id}_{\overline{\mathscr{Q}}_S}
    \end{bmatrix},\,\,B\in \mathcal{H}om(\overline{\mathscr{Q}},\overline{\mathscr{N}})(S).
\end{equation}
This particularly implies that $\gr_2\mathscr{G}|_S=\mathbb{G}_{a,S}^{r(d-r)}$. 
\item This follows from (a) and (b). 
\item Let $\mathscr{H}$ be the subgroup scheme of $\GL[\underline{F}/\mathfrak{m}\underline{F}]$ fixing the filtration (\ref{eq:Tatu2}), and reducing to identity on $\overline{\mathscr{Q}}$. We have $\mathscr{G}/\Fil_1\mathscr{G}\subseteq \mathscr{H}$. As in (b), the canonical splitting of (\ref{eq:uGL}) again induces a morphism $\mathscr{H}\rightarrow \GL[\underline{F}]$
whose image lies in $\mathscr{G}$. This implies that $\mathscr{G}/\Fil_1\mathscr{G}=\mathscr{H}$ and $\gr_3\mathscr{G}=\mathscr{H}/\mathscr{H}_2$.

Let $\Spec S$ be a Zariksi affine open of $\Quot_{d,n}^r(R)$ over which $\overline{\mathscr{Q}}$ and $\overline{\mathscr{N}}$ are trivialized. Pick the same splitting of (\ref{eq:Tatu2}) as in (b). Then the $S$-points of $\mathscr{H}|_S$ can be written into matrices 
\begin{equation}
    \begin{bmatrix}
        A & B\\
        0 & \text{Id}_{\overline{\mathscr{Q}}_S}
    \end{bmatrix},\,\,A\in \GL[\overline{\mathscr{N}}](S),\,\,B\in \mathcal{H}om(\overline{\mathscr{Q}},\overline{\mathscr{N}})(S).
\end{equation}
The assertion that $\gr_3\mathscr{G}\simeq\GL[\overline{\mathcal{N}}]$ is clear if one takes quotient of this by (\ref{eq:matH2}).$\hfill\square$\end{enumerate}
\subsection{Relating $\Quot_{d,n}^r(R)$ and $C_n^r(R)$}\label{sub:QC} 
We start by constructing ${\mathfrak{Q}}$ in (\ref{eq:torsorQ}) as the $\GL_n$-torsor corresponding to the bundle $\mathscr{Q}$. More precisely, let $V_n$ be a fixed $k$-vector space of dimension $n$, and $\underline{V_n}$ be the corresponding trivial bundle over $\Quot_{d,n}^r(R)$, then 
\begin{equation}
{\mathfrak{Q}}=\mathcal{I}som_{\Quot_{d,n}^r(R)}(\underline{V_n},{\mathscr{Q}}).
\end{equation} 
Therefore we have a projection $\pi:{\mathfrak{Q}}\rightarrow \Quot_{d,n}^r(R)$. On the other hand, take $X={\mathfrak{Q}}$ and $ \mathcal{F}=\pi^*{\mathscr{Q}}$ in (\ref{eq:funcCn}), since $
\pi^*{\mathfrak{Q}}\rightarrow {\mathfrak{Q}}
$ is a trivial torsor (e.g., there is a global section which comes from the diagonal), $\mathcal{F}\rightarrow {\mathfrak{Q}}$ is a trivial bundle, and is an object in $\Coh_n^r(R)(X)$. Let \begin{equation}
\iota: \mathcal{O}_X\otimes_k V_n\simeq  \mathcal{F}
\end{equation} 
be the isomorphism of $\mathcal{O}_X$-modules induced by the evaluation map \begin{equation}   \underline{V_n}\times_{\Quot_{d,n}^r(R)}  {\mathfrak{Q}}\rightarrow {\mathscr{Q}}.
\end{equation}
The pair $(\mathcal{F}, \iota)$ gives rise to a morphism $\varpi:{\mathfrak{Q}}\rightarrow C_n^r(R)$. 

On the level of $S$-points, $\varpi$ can be easily understood as follows: 
let $x$ be an $S$-point of ${\mathfrak{Q}}$ corresponding to a module $N\subseteq F\otimes S$ with an $S$-isomorphism $\iota: V_n\otimes S\simeq (F\otimes S)/N$, then $\varpi(x)=((F\otimes S)/N,\iota)\in C_n^r(R)(S)$. It is easily seen that $\varpi$ is surjective, and is invariant under $\GL[F]$-action. Indeed, let $g\in\GL[F](S)$, and $x=(N,\iota)$ as above, we have $\varpi g(x)=((F\otimes S)/gN, g\iota)\sim ((F\otimes S)/N, \iota)$. This implies that $\varpi g(x)= \varpi(x)$.
\subsubsection{Local sections} 
There is a Zariski affine cover $U=\Spec D\rightarrow C_n^r(R)$, such that the tautological point of $ C_n^r(R)(U)$ corresponds to (the equivalent class of) a pair \begin{equation}\label{eq:Apair}
  ( (F\otimes D)/ N^{\text{taut}}, \iota^{\text{taut}}).
\end{equation}
Fixing a pair (\ref{eq:Apair}) in the equivalence class, we get a $U$-point $\sigma$ of ${\mathfrak{Q}}$ such that $\varpi \sigma=\text{Id}_{U}$. Therefore, $\sigma$ is a section of $\varpi$ over $U$. Note that such sections are well defined up to an action of $\GL[F](U)$. Let ${\mathfrak{Q}}_U={\mathfrak{Q}}\times_{C_n^r(R)}{U}$. Since $\varpi$ is invariant under $\GL[F]$, the section $\sigma$ induces a morphism of $U$-schemes  \begin{equation}
p:\GL[{F}]_{U}=U\times \GL[F]\xrightarrow{\sigma\times \text{Id}} {\mathfrak{Q}}\times \GL[F]\xrightarrow{\cdot} {\mathfrak{Q}}.
\end{equation}
\begin{lemma}\label{lm:utorsor}
Let $\mathscr{G}$ be the stabilizer bundle over $\Quot_{d,n}^r(R)$ (\S\ref{subsub:stab}) and let $\mathscr{G}_{U}:=\sigma^*\pi^*\mathscr{G}$ be the pullback group scheme over $U$. Then $p$ realizes ${\mathfrak{Q}}$ as the fppf quotient of  $\GL[{F}]_{U}$ by the fppf subgroup $\mathscr{G}_U$. In particular, we have:
\begin{enumerate}
    \item\label{it:tor1} $p$ is an fppf cover.
    \item \label{it:tor2} $\GL[{F}]_{U}$ is an fppf $\mathscr{G}_{U}$-torsor over ${\mathfrak{Q}}_U$, trivialized under the covering $p$.
\end{enumerate} 
\end{lemma}
\proof First we check that ${\mathfrak{Q}}_U$ is the fppf quotient. Let $S$ be an fppf affine over $U$ and  $x$ be an $S$-point of $\sigma U$, corresponding to a module $N\subseteq F\otimes S$ with an $S$-isomorphism $\iota: V_n\otimes S\simeq (F\otimes S)/N$. For any $g\in \GL[F](S)$, we have $p(x,g)=((F\otimes S)/gN,g\iota)$. Since $p(x,g)=p(x,g')$ if and only if $g'g^{-1}$ is in the stabilizer of $x$, i.e., $g'g^{-1}$ is in the fiber of $\mathscr{G}_{S}$ over $x$, we obtain the fppf quotient assertion. Once this is established, \ref{it:tor1} and \ref{it:tor2} are just formal properties of fppf quotient, see \cite[Proposition 4.31]{EGMAV}. $\hfill\square$
\subsection{The motivic formula}\label{sub:motiform} We say that a unipotent group scheme over a base $Y$ is $Y$-\textit{split}, or simply \textit{split}, if it is a successive extension of $\mathbb{G}_{a,Y}$. We say a group scheme $\mathcal{G}$ over $Y$ is \textit{special}, if any fppf $\mathcal{G}$-torsor
is Zariski locally trivial. 
It is well-known that a split unipotent group scheme is special, and an extension of special group schemes is special.

\begin{lemma}\label{lm:ZarstatiU} Notation as in \S\ref{subsub:stab}. There is a Zariski stratification \begin{equation}
\Quot_{d,n}^r(R)=\bigsqcup_{\beta\in \mathbf{B}} Z_\beta,
    \end{equation}
    such that for each $\beta$, $\Fil_2\mathscr{G}_{Z_\beta}$ is a split $Z_\beta$-unipotent group scheme.  
\end{lemma}
\proof Note that it suffices to show that $\Fil_1\mathscr{G}_{Z_\beta}$ is split for some Zariski stratification $\{Z_\beta\}$, thanks to Lemma~\ref{lm:stblzer}\ref{it:stblzer2}. Take a nonempty affine open $\Spec S\subseteq \Quot_{d,n}^r(R)$ and view it as an $S$-point of $\Quot_{d,n}^r(R)$. It then corresponds to a submodule $N\subseteq F\otimes S$. Shrinking $S$, we can assume that $N'=N\cap (\mathfrak{m}F\otimes S)$ is free. The group scheme $\Fil_1\mathscr{G}_S$ can be described as the functor 
\begin{equation}
T/S\rightarrow \{g\in U[\mathfrak{m}F](T)|g(u_i)\in u_i+N'\otimes T, 1\leq i\leq d\}.
\end{equation}
Let $\{\omega_i\}_{i=1}^{bd+d-n}$ be a basis of $N'$. There is a nonzero element $a\in N'$ such that $\mathfrak{m}a=0$. Indeed, we know that for any element $b\in N'$, $\mathfrak{m}^{n}b=0$ for sufficiently large $n$. So multiply $b$ by some suitable element $u\in \mathfrak{m}$, we have $ub\neq 0$ but $\mathfrak{m}ub=0$. Let $a=ub$ and express $a=\sum_{i} s_i\omega_i$. There is at least one $s_i$, say $s_1$, which is nonzero. Since $S$ is reduced (recall that we have assumed everything is reduced in \S\ref{subsub:naa}), the localization $S[s_1^{-1}]$ is nonzero. Therefore, $\{a\}\cup \{\omega_i\}_{i=2}^{bd+d-n}$ is a basis of $N'[s_1^{-1}]$. Therefore, after shrinking $S$, we can assume that $N'$ admits a basis $\{\omega_i\}_{i=1}^{bd+d-n}$ such that $\mathfrak{m}\omega_1=0$. 

Replace $N'$ by $N'/\omega_1$ and do the same trick. By induction, we see that after shrinking $S$ sufficiently many times, there is a basis $\{\omega_i\}_{i=1}^{bd+d-n}$ of $N'$ such that $\mathfrak{m}\omega_j\subseteq \Span\{\omega_1,...,\omega_{j-1}\}$.  Define a central series of unipotent group schemes \begin{equation}
 0\triangleleft \mathscr{U}_{1,S}\triangleleft \mathscr{U}_{2,S}\triangleleft...\triangleleft \mathscr{U}_{bd-n+r,S}=\Fil_1\mathscr{G}_S
\end{equation} 
by setting the functors of points as \begin{equation}
   \mathscr{U}_{j,S}: T/S\rightarrow \{g\in U[\mathfrak{m}F](T)|g(u_i)\in u_i+\Span_S\{\omega_1,...,\omega_j\}\otimes_S T, 1\leq i\leq d\}.
\end{equation}
It follows easily from the construction that $\mathscr{U}_{j,S}/\mathscr{U}_{j,S}\simeq \mathbb{G}_{a,S}^{d}$. Therefore $\Fil_1\mathscr{G}_{S}$ is $S$-split. 

Replace $\Quot_{d,n}^r(R)$ by the closed subscheme $\Quot_{d,n}^r(R)-\Spec S$ and repeat the above process. Since $\Quot_{d,n}^r(R)$ is finite type over $k$, this ends in finitely many steps, and yields the desired stratification. $\hfill\square$
\begin{lemma}\label{lm:ZarstatiU2}
Notation as above. There is a Zariski stratification \begin{equation}\label{eq:staev}
U=\bigsqcup_{\gamma\in \mathbf{C}} U_{\gamma},
\end{equation}
such that $\mathscr{G}_{U_\gamma}$ is a special $U_{\gamma}$-group scheme. 
\end{lemma}
\proof From Lemma~\ref{lm:ZarstatiU} and Lemma~\ref{lm:stblzer}(d), we see that there is a stratification $\{U_\gamma\}$ such that $\Fil_2\mathscr{G}_{U_\gamma}$ is split and $\gr_3\mathscr{G}_{U_\gamma}$ is isomorphic to $\GL_{d-r,U_\gamma}$. It is classical that $\GL_{d-r,U_\gamma}$ is special. Therefore $\mathscr{G}_{U_\gamma}$, as an extension of special group schemes, is special. $\hfill\square$\\[10pt]
\textit{Proof of Theorem~\ref{Thm:apdxB1}}. By 
Lemma~\ref{lm:ZarstatiU2}, there is a stratification (\ref{eq:staev}) on $U$ such that the restriction of $\mathscr{G}_U$ to each stratum is special. It follows from Lemma~\ref{lm:utorsor} and the speciality of $\mathscr{G}_{U_\gamma}$, together with Lemma~\ref{lm:stblzer}, that
\begin{equation}
    [U_{\gamma}][\GL[F]]=[{\mathfrak{Q}}_{U_\gamma}][\GL_{d-r}]\mathbb{L}^{bd^2-d(n-d)-(d-r)^2}.  
\end{equation}
Since various ${\mathfrak{Q}}_{U_\gamma}$ also form a Zariski stratification of ${\mathfrak{Q}}_U$, and $U$ is a Zariski cover of $C_n^r(R)$, we see that 
\begin{equation}
    [C_{n}^r(R)][\GL[F]]=[{\mathfrak{Q}}][\GL_{d-r}]\mathbb{L}^{bd^2-d(n-d)-(d-r)^2}.  
\end{equation}
Since $[{\mathfrak{Q}}]=[\Quot_{d,n}^r(R)][\GL_n]$, $[C_{n}^r(R)]=[\Coh_n^r(R)][\GL_n]$ and $[\GL[F]]=[\GL_d]\mathbb{L}^{bd^2}$ by Lemma~\ref{lm:theGroup}, we obtain the formula (\ref{eq:apdxB1}). $\hfill\square$ \subsection{Gröbner strata}\label{subsec:Gbstrata} As noted at the beginning, the theory of Gr\"{o}bner strata in the Hilbert scheme of points is well developed. Our case is slightly modified in the sense that we work with (1) power series rings, and (2) Gröbner basis for modules rather than ideals. Since we only classify submodules of $F:=R^d$ with finite codimension, one shall expect that the Gröbner theory in our case bears no much difference from the Gröbner theory in the usual polynomial ring case. Moreover, one can even reduce the theory of Gröbner basis for modules to the theory of Gröbner basis for ideals, in the sense that submodules $N\subseteq F$ of rank $r$ and codimension $n$ can be regarded, via square-zero extension, as ideals in the square-zero deformation $R[[u_1,...,u_r]]/(u_iu_j)$. Therefore we can realize $\Quot_{d,n}^d(R)$ as certain subscheme of the Hilbert scheme of points $ \Hilb_n(R[[u_1,...,u_d]]/(u_iu_j))$. Ideally, the Gröbner strata on the latter (the Hilbert scheme) shall restrict to Gröbner strata on $\Quot_{d,n}^d(R)$. These observations and intuitions suggest that the theory of Gröbner strata in our case is not too different from the usual case. 

Let $R$ be a monomial subring of a power series ring over a field $k$.  In this section, we establish a theory of Gröbner strata which is just adequate for stratifying $\Quot_{d,n}^d(R)$. All treatments are similar and parallel to \cite{LedererGrobnerstrata}, hence we will make them as concise as possible.  

\subsubsection{Gröbner basis for submodules of $F\otimes S$} Let $S$ be a finite type algebra over $k$. Then the ring $R\otimes S$ is nothing other than the subring of a power series ring $S[[\underline{x}]]$ generated by certain monomials. A submodule $M\subseteq F\otimes S$ is called \textbf{monic}, if $\LT(M)$ is a monomial submodule. (This notion is exclusive to Gr\"obner theory over a ring, since over a field, any submodule is monic.) A Gröbner basis of a submodule $M$ is a finite subset $G$ of $M$ such $\LT(M)$ agrees with the submodule in $F\otimes S$ generated by $\LT(g)$ for $g \in G$. Since $R\otimes S$ is Noetherian, every submodule $M$ admits a Gröbner basis. Similar to \S\ref{sec:groebner}, a Gr\"obner basis $G=\set{g_1,\dots,g_h}$ is \textbf{reduced} if
\begin{enumerate}
\item $\LT(g_1),\dots,\LT(g_h)$ are monomials and are mutually indivisible;
\item No nonleading term of $g_i$ is divisible by $\LT(g_j)$, for any $i,j$.
\end{enumerate}
Similar to the classical setting (See  \cite[Theorem 2.11]{Asc05} and \cite[Theorem 4]{Wib07}), it is not hard to see that a submodule $M$ admits a reduced Gröbner basis if and only if $M$ is monic. We also note that the Buchberger criterion still holds in this case.

\subsubsection{Gröbner subfunctors and the induced stratification}\label{subsec:grobsubfunctor} Following \S \ref{sec:strata}, let $\set{M_\alpha: \alpha\in \Xi}$ denote the set of finite-codimensional monomial submodules of $\m F$, and let $\Delta(\alpha)$ and $C(\alpha)$ be the set of monomials not contained in $M_\alpha$ and the set of corners, respectively. Let $n(\alpha):=\dim_k \m F/M_\alpha$. Though \S \ref{sec:strata} is specifically written for the cusp, the notation above applies to any monomial subring $R$ of a power series ring over $k$. 

We begin by defining the following two subfunctors of $\Hilb_{d,n}(R)=\Quot_{d,n+d}^d(R)$. Let $\preceq$ be a monomial order over $R$, for each $\alpha \in \Xi$ with $n(\alpha)=n$, we define functors $\Hilb^\Delta(\alpha)$ and $\Hilb(\alpha)$ by their points over affine $k$-algebras
\begin{equation}
    \Hilb^\Delta(\alpha)(S)=\left\{\phi:F\otimes S \twoheadrightarrow Q, \text{ such that $Q$ is free over } S \text{ with basis } \phi(\Delta(\alpha))    \right\}/\sim
\end{equation}
\begin{equation}
\Hilb(\alpha)(S)=\left.\left\{ \begin{aligned}
   &\phi:F\otimes S\twoheadrightarrow Q, \text{ such that } \ker\phi\text{ has a reduced Gröbner basis of form}\\
   & g_\mu= \mu+\sum_{{\nu}\in \Delta(\alpha), \mu\prec \nu} s_{\mu,\nu} \nu, \text{ where }\mu \in C(\alpha).
 \end{aligned} 
\right\}\right/\sim
\end{equation}
Note that there is a bijection between $\Xi$ and the set of standard sets in the sense of \cite{LedererGrobnerstrata}. Therefore the two subfunctors defined here are the direct analogues of the functors $\Hilb^{\Delta}$ and $\Hilb^{\prec\Delta}$ defined in \textit{loc.cit}. It is easy to check that 
$\Hilb^\Delta(\alpha)$ and $\Hilb(\alpha)$ 
are sheaves in the big Zariski site over $k$. Note that the definition of $\Hilb^\Delta(\alpha)$ does not require the monomial order.

\begin{proposition}\label{standardcover}
The followings are true:
\begin{enumerate}
\item $\Hilb^{\Delta}(\alpha)$ are represented by open subschemes of $\Hilb_{d,n}(R)$, and they form an open cover 
\begin{equation}
\Hilb_{d,n}(R)=\bigcup_{\substack{\alpha\in\Xi\\n(\alpha)=n}}\Hilb^{\Delta}(\alpha).
\end{equation}
\item  $\Hilb(\alpha)$  is represented by a closed subscheme of $\Hilb^{\Delta}(\alpha)$. 
\item $\Hilb_{d,n}(R)$ admits a Zariski stratification 
\begin{equation}
\Hilb_{d,n}(R)=\bigsqcup_{\substack{\alpha\in\Xi\\n(\alpha)=n}}\Hilb(\alpha).\end{equation} 
\end{enumerate}
As a result, $[\Hilb_{d,n}(R)]=\sum_{\substack{\alpha\in\Xi\\n(\alpha)=n}}[\Hilb(\alpha)]$ in $K_0(\mathrm{Var}_k)$.
\end{proposition}
\begin{proof}~
\begin{enumerate}
\item It is clear from the definition that $\Hilb_{d,n}(R)$ is a union of subfunctors $\Hilb({\Delta}(\alpha)), \alpha\in \Xi, n(\alpha)=n$. We now prove the openness. 

Applying \eqref{eq:embQGr} with $r=d$ and $n_{\eqref{eq:embQGr}}=n+d$, we get a locally closed immersion
\begin{equation}
\Hilb_{d,n}(R)\incl \Gr:=\Gr(F,bd-n).
\end{equation}
It is a closed embedding since $d=r$, as is observed before. Consider the subspace $\Span_{k} \Delta(\alpha)\subseteq \mathfrak{m}F$. Let $\mathcal{V}(\alpha)\subseteq\Gr(k)$ be the subset parametrizing subspaces that have trivial intersection with $\Span_{k} \Delta(\alpha)$. Since $\dim\Span_{k} \Delta(\alpha)=n$, the subspaces of $\mathfrak{m}F$ of dimension $bd-n$ (i.e., of the complementary dimension to $\Span_{k} \Delta(\alpha)$) have generically trivial intersection with $\Span_{k} \Delta(\alpha)$. Therefore, $\mathcal{V}(\alpha)$ is open. Now we consider $\mathcal{V}(\alpha)$ as an open subscheme of $\Gr$. Pulling $\mathcal{V}(\alpha)$ back to $\Hilb_{d,n}(R)$ exactly gives $\Hilb^{\Delta}(\alpha)$ which is, in particular, an open subscheme of $\Hilb_{d,n}(R)$.
\item The proof is almost identical to \cite[Theorem 1]{LedererGrobnerstrata}. Let $\phi\in \Hilb^\Delta(\alpha)(S)$. For every $\mu\in C(\alpha)$ there is a unique polynomial of the form 
\begin{equation}
f_\mu=\mu+\sum_{\nu\in \Delta(\alpha)} d_{\mu,\nu} \nu \in \ker\phi.
\end{equation}
Let $J\subseteq S$ be the ideal generated by $d_{\mu,\nu}$ where $\mu$ runs over $C(\alpha)$ and $\nu$ runs over elements of $\Delta$ such that $\nu\prec \mu$. A morphism $\psi:S\rightarrow S'$ in the category of $k$-algebras factors through $S\rightarrow S/J$ if and only if $\ker(\phi\otimes \psi)$  is monic with Gröbner basis $\{\psi(f_\mu)\}$. The same argument as in  \textit{loc.cit} shows that $\Hilb(\alpha)$ is represented by a closed subscheme of $\Hilb^\Delta(\alpha)$. 
\item  This assertion follows directly from (a) and (b), since every field-valued point of $\Hilb_{d,n}(R)$ lies in exactly one stratum $\Hilb(\alpha)$.
\end{enumerate}
\end{proof} 

\begin{warning}
Our notion of stratification does not require that the closure of any stratum is a disjoint union of strata. Since we are computing the motives but not the intersection theory, our notion of Zariski stratification suffices. As is pointed out in \cite{LedererGrobnerstrata}, the general theory of Gr\"obner stratification does not guarantee that Proposition \ref{standardcover}(c) is a stratification in that stronger sense.
\end{warning}
 
\subsubsection{Motivic version of the counting results} Let $R=k[[T^2,T^3]]$ and assume the setup of \S\ref{sec:strata}. Define a Zariski subsheaf $\Hilbnil(\alpha)\subseteq \Hilb(\alpha)$ as 
\begin{equation}
    \Hilbnil(\alpha)(S)=\left.\left\{ \begin{aligned}
   &\phi:F\otimes S \twoheadrightarrow Q, \text{ such that } \ker\phi\text{ has a reduced Gröbner basis of form}\\
   & g_\mu= \mu+\sum_{\mu\in \Delta(\alpha), \mu\prec \nu} s_{\mu,\nu} \nu, \text{ where }\mu\in C(\alpha), \text{ such that }s_{\mu,\nu}=0 \text{ if }\mu=\mu_i^0 .\end{aligned}\right\}\right/\sim
\end{equation} This is the motivic version of the set $\Hilbnil(\alpha)$ defined in Lemma~\ref{lem:decomposition}. Recall that we also defined a variety $V(\alpha)$ in Conjecture \ref{conj:L_rationality_strata} for a pure-$K$ leading term datum $\alpha$. The following is a motivic upgrade of Lemma~\ref{lem:decomposition} and Lemma~\ref{lem:staircase}:
\begin{theorem}\label{Thm:groebner_strata_motivic}
$\Hilbnil(\alpha)$ is represented by a closed subscheme of $\Hilb(\alpha)$. Furthermore, we have the following identities in $K_0(\mathrm{Var}_k)$:
\begin{align}
&[\Hilb(\alpha)]=[\Hilbnil(\alpha)] \cdot \mathbb{L}^{\sum_i \abs[\big]{\Delta(\alpha)_{\succ \mu_i^0(\alpha)}}}.\\
&  [\Hilbnil(\alpha)]=[V(\alpha|_{K_\alpha})]\cdot\mathbb{L}^{\sum_{i\in K_\alpha}\abs*{C(\alpha|_{J_\alpha})_{\succ \mu_i^0}}}.
\end{align}
\end{theorem}
\proof To show that $\Hilbnil(\alpha)$ is a closed subscheme of $\Hilb(\alpha)$, we simpy replace the ideal $J$ in the proof of Lemma~\ref{standardcover}(2) by an ideal $J_0$ generated by $J$ together with $d_{\mu,\nu}$ where $\mu=\mu_i^0$. A morphism $\psi:S\rightarrow S'$ in the category of $k$-algebras factors through $S\rightarrow S/J_0$ if and only if $\ker(\phi\otimes \psi)$  is monic with Gröbner basis $\{\psi(f_\mu)\}$, such that $\psi(f_\mu)=\psi(\mu)$ if $\mu=\mu_i^0$. It then follows that $\Hilbnil(\alpha)$ is represented by a closed subscheme of $\Hilb(\Delta(\alpha))$, hence a closed subscheme of $\Hilb(\alpha)$. For the two formulas, one can simply upgrade the proofs of Lemma~\ref{lem:decomposition} and Lemma~\ref{lem:staircase}. Indeed, one first establishes a stronger structural result 
\begin{align}
&\Hilb(\alpha)\simeq \Hilbnil(\alpha) \times \mathbb{A}^{\sum_i \abs[\big]{\Delta(\alpha)_{\succ \mu_i^0(\alpha)}}}.\\
&  \Hilbnil(\alpha)\simeq V(\alpha|_{K_\alpha})\times \mathbb{A}^{\sum_{i\in K_\alpha}\abs*{C(\alpha|_{J_\alpha})_{\succ \mu_i^0}}}.
\end{align}
by working with $S$-points (instead of $k$-points as we did) in the proofs of above lemmas. It is not hard to see that the proofs of the two lemmas are functorial in nature, and still carries over to $S$-points. We will leave the details to the readers. $\hfill\square$

\bibliography{bibliography.bib}
\end{document}